\documentclass[11pt]{article}

\usepackage{fullpage}
\usepackage{amsthm,amsmath,amsfonts,amssymb,amsxtra}
\usepackage{mathrsfs,mathtools}
\usepackage{bbm}
\usepackage[usenames, dvipsnames]{xcolor}
\usepackage[bookmarks, 
plainpages=false,
colorlinks=true,
citecolor=red,
linkcolor=blue,
anchorcolor=red,
urlcolor=blue
]{hyperref}
\usepackage{enumitem}

\usepackage{algorithmicx,algorithm,algpseudocode}
\algnewcommand{\Input}{\item[\textbf{Input:}]}
\algnewcommand{\Output}{\item[\textbf{Output:}]}

\usepackage{wrapfig}
\usepackage{xr}

\newtheorem{theorem}{Theorem}[section]
\newtheorem{definition}[theorem]{Definition}
\newtheorem{lemma}[theorem]{Lemma}
\newtheorem{corollary}[theorem]{Corollary}

\newtheorem{proposition}[theorem]{Proposition}
\newtheorem{conjecture}[theorem]{Conjecture}

\usepackage{prettyref}
\newrefformat{eq}{(\ref{#1})}
\newrefformat{thm}{Theorem~\ref{#1}}
\newrefformat{th}{Theorem~\ref{#1}}
\newrefformat{chap}{Chapter~\ref{#1}}
\newrefformat{sec}{Section~\ref{#1}}
\newrefformat{seca}{Section~\ref{#1}}
\newrefformat{algo}{Algorithm~\ref{#1}}
\newrefformat{fig}{Fig.~\ref{#1}}
\newrefformat{tab}{Table~\ref{#1}}
\newrefformat{rmk}{Remark~\ref{#1}}
\newrefformat{clm}{Claim~\ref{#1}}
\newrefformat{def}{Definition~\ref{#1}}
\newrefformat{cor}{Corollary~\ref{#1}}
\newrefformat{lmm}{Lemma~\ref{#1}}
\newrefformat{prop}{Proposition~\ref{#1}}
\newrefformat{pr}{Proposition~\ref{#1}}
\newrefformat{app}{Appendix~\ref{#1}}
\newrefformat{apx}{Appendix~\ref{#1}}
\newrefformat{ex}{Example~\ref{#1}}
\newrefformat{exer}{Exercise~\ref{#1}}
\newrefformat{soln}{Solution~\ref{#1}}

\newcommand{\Iprod}[2]{\langle #1, #2 \rangle}

\newcommand{\naturals}{\mathbb{N}}
\newcommand{\R}{\mathbbm{R}}
\newcommand{\p}{\operatorname{\mathbbm{P}}}
\newcommand{\E}{\operatorname{\mathbbm{E}}}
\newcommand{\Z}{\mathbbm{Z}}

\newcommand{\fS}{\mathfrak{S}}
\newcommand{\fX}{\mathfrak{X}}
\newcommand{\fR}{\mathfrak{R}}
\newcommand{\cS}{\mathcal{S}}
\newcommand{\cR}{\mathcal{R}}
\newcommand{\cK}{\mathcal{K}}
\newcommand{\cB}{\mathcal{B}}

\newcommand{\cP}{\mathcal{P}}
\newcommand{\cQ}{\mathcal{Q}}
\newcommand{\cM}{\mathcal{M}}

\newcommand{\cE}{\mathcal{E}}
\newcommand{\cN}{\mathcal{N}}

\newcommand{\sM}{\mathsf{M}}
\newcommand{\sP}{\mathscr{P}}
\newcommand{\scrM}{\mathscr{M}}

\newcommand{\eps}{\varepsilon}
\newcommand{\1}{\mathbbm{1}}

\newcommand{\bone}{\mathbf{1}}

\newcommand{\cI}{\mathcal{I}}

\newcommand{\KL}{\mathsf{KL}}
\newcommand{\TV}{\mathsf{TV}}

\newcommand{\id}{\mathsf{id}}

\newcommand{\argmin}{\operatorname*{argmin}}

\newcommand{\defn}{\triangleq}

\newcommand{\dkt}{d_{\mathsf{KT}}}
\newcommand{\dhaus}{d_{\mathsf{H}}}

\newcommand{\poly}{\operatorname{\mathsf{poly}}}

\newcommand{\moment}{\mathfrak{m}}
\newcommand{\SR}{\texttt{SubOrder}}
\newcommand{\ID}{\texttt{InsertionDemixing}}
\newcommand{\bP}{\mathbf{P}}

\newcommand{\pth}[1]{\left( #1 \right)}

\newcommand{\calM}{{\mathcal{M}}}

\newcommand{\calP}{{\mathcal{P}}}

\newcommand{\calS}{{\mathcal{S}}}

\newcommand{\supp}{\ensuremath{\mathrm{supp}}}
\newcommand{\Ber}{\text{Bernoulli}}
\newcommand{\Bern}{\text{Bernoulli}}

\newcommand{\mopt}{m^*_k}
\newcommand{\cone}{c}

\title{Learning Mixtures of Permutations: Groups of Pairwise Comparisons and Combinatorial Method of Moments}

\author{Cheng Mao and Yihong Wu\thanks{
C.~Mao is with the School of Mathematics, Georgia Institute of Technology, Atlanta, GA, USA,  \texttt{cheng.mao@math.gatech.edu}. 
C.~Mao is supported in part by the NSF grant DMS-2053333. 
Y.~Wu is with Department of Statistics and Data Science, Yale University, New Haven CT, USA, \texttt{yihong.wu@yale.edu}.
Y.~Wu is supported in part by the NSF grant CCF-1900507, the NSF CAREER award CCF-1651588, and an Alfred Sloan fellowship. 
}}

\date{\today}

\begin{document}

\maketitle

\begin{abstract}
In applications such as rank aggregation, mixture models for permutations are frequently used when the population exhibits heterogeneity. In this work, we study the widely used Mallows mixture model. In the high-dimensional setting, we propose a polynomial-time algorithm that learns a Mallows mixture of permutations on $n$ elements with the optimal sample complexity that is proportional to $\log n$, improving upon previous results that scale polynomially with $n$. In the high-noise regime, we characterize the optimal dependency of the sample complexity on the noise parameter. Both objectives are accomplished by first studying demixing permutations under a noiseless query model using groups of pairwise comparisons, which can be viewed as moments of the mixing distribution, and then extending these results to the noisy Mallows model by simulating the noiseless oracle. 
\end{abstract}

{
  \hypersetup{linkcolor=black}
  \tableofcontents
}

\newpage

\section{Introduction}

\emph{Rank aggregation} is the task that aims to combine different rankings on the same set of alternatives, to obtain a \emph{central ranking} that best represents the population. 
The problem of rank aggregation has been studied in social choice theory since Jean-Charles de Borda~\cite{Bor81} and Marquis de Condorcet~\cite{Con85} in the 18th century. 
More recently, due to the ubiquity of preference data, rank aggregation has found applications in a variety of areas, including web search, classification and recommender systems~\cite{Dwoetal01,FagKumSiv03,Liuetal07,BalMakRic10,KorCleSib17}.

In these practical applications, 
the population of interest is often \emph{heterogeneous} in the sense that different subpopulations have divided preferences over the alternatives. 
For example, multiple groups of people may have different preferences for movies or electoral
candidates~\cite{Mar95,GorMur08b}. 
In such a scenario, rather than seeking a single central ranking, it is preferable to find \emph{a mixture of rankings} to represent the preferences of the population~\cite{JorJac94,MurMar03,BusOrbBuh07,GorMur08,Awaetal14,ZhaPieXia16,LiuMoi18,DeODoSer18}.

\subsection{Mallows mixture and related work}

In this work, we adopt a statistical approach to the problem of heterogeneous rank aggregation. 
Let $\cS_n$ denote the set of permutations on $[n] \defn \{1, \dots, n\}$. 
A ranking of $n$ alternatives is described by a permutation $\pi \in \cS_n$. 
We refer to $n$ as 
the \emph{size} of a permutation. 
Furthermore, we model the preference of the population by a distribution on the set of permutations $\cS_n$. 
Suppose that $N$ independent permutations are generated from the distribution, each of which represents an observed ranking.

In this paper, we focus on the \emph{Mallows model} $M(\pi, \phi)$ on $\cS_n$, with \emph{central permutation} $\pi \in \cS_n$ and \emph{noise parameter} $\phi \in (0,1)$~\cite{Mal57}. In the Mallows model, the probability of generating a permutation $\sigma \in \cS_n$ is equal to 
$$
\frac{1}{Z(\phi)} \phi^{\dkt(\pi, \sigma)} ,
$$ 
where $Z(\phi)$ is a normalization factor (see \eqref{eq:def-pdf}) and $\dkt(\pi, \sigma)$ denotes the \emph{Kendall tau} distance between permutations $\pi$ and $\sigma$, defined by 
\begin{align}
\dkt(\pi, \sigma) \defn \sum_{i, j \in [n]} \1\{ \pi(i) < \pi(j), \, \sigma(i) > \sigma(j) \} . 
\label{eq:def-kt}
\end{align}
There have been decades of work studying theoretical properties and efficient learning algorithms for the Mallows model and its generalizations~\cite{FliVer86,DoiPekReg04,Meietal07,BraMos09,LuBou11,CarProSha13,BusHulSzo14,Busetal19,IruCalLoz19}.

To model a heterogeneous population, we consider the \emph{Mallows mixture} 
\begin{align}
\cM \defn \sum_{i=1}^k w_i M(\pi_i, \phi)
\label{eq:k-mix}
\end{align}
with $k$ components, where the $i$th component has central permutation $\pi_i \in \cS_n$, noise parameter $\phi \in (0,1)$, and weight $w_i \ge \gamma$ for some $\gamma > 0$. 
We assume for simplicity that the noise parameter $\phi$ is known and common for all components of the mixture. 
In general, different components may have different, unknown noise parameters $\phi_i \in (0,1)$, which we briefly discuss in Section~\ref{sec:discuss}. 
Let us remark that, the number of components $k$ in a mixture of permutations is typically a small quantity, so we let $k$ be a fixed constant throughout this work. 
On the other hand, the size $n$ of the permutations is typically large because it represents the number of alternatives. 

The Mallows mixture has also received considerable attention in recent years~\cite{MeiChe10,LuBou14,Awaetal14,Chietal15,LiuMoi18,DeODoSer18}. 
More specifically, Chierichetti et al.~\cite{Chietal15} established the identifiability of the Mallows mixture given sufficiently many permutations generated from $\cM$ under mild conditions. 
The first polynomial-time algorithm to learn the Mallows mixture with two components was proposed by Awasthi et al.~\cite{Awaetal14}, who particularly showed that the central permutations can be recovered exactly with high probability, when the sample size $N$ exceeds $\poly(n, \frac{1}{\phi(1-\phi)}, \frac{1}{\gamma})$. 
In the case of the Mallows $k$-mixture for any fixed constant $k$, Liu and Moitra~\cite{LiuMoi18} introduced a polynomial-time algorithm with sample complexity $\poly(n, \frac{1}{1-\phi}, \frac{1}{\gamma})$ that exactly recovers the central permutations with high probability. 

\subsection{Major contributions}


The first main result of this work concerns the sample complexity of learning Mallows mixture when the size of the permutation is large.
\begin{theorem}[Informal statement of Corollary~\ref{cor:main-sc}]
\label{thm:informal1}
There is a polynomial-time algorithm with the following property. Fix any $0<\delta<0.1$. Given $\poly(\frac{1}{1-\phi}, \frac{1}{\gamma}) \cdot \log \frac{n}{\delta}$ i.i.d.\ observations from the Mallows $k$-mixture \eqref{eq:k-mix}, the algorithm exactly recovers the set of central permutations $\{\pi_1, \dots, \pi_k\}$ with probability at least $1-\delta$.
\end{theorem}

In the above statement, $\poly(\frac{1}{1-\phi}, \frac{1}{\gamma})$ denotes a polynomial in $\frac{1}{1-\phi}$ and $\frac{1}{\gamma}$ whose degree depends on $k$; see Corollary~\ref{cor:main-sc} for the explicit expression of this polynomial. 
Most importantly, this polynomial does not depend on the size $n$ of the permutations, and the sample complexity bound only depends on $n$ logarithmically. 
This logarithmic dependency on $n$ is a significant improvement over the previous polynomial dependency and is in fact optimal (see the remark after Corollary~\ref{cor:main-sc}). 

Complementing \prettyref{thm:informal1}, the next result makes precise the optimal dependency of the sample complexity on the noise level $\frac{1}{1-\phi}$ when the size of the permutation is fixed.

\begin{theorem}[Informal statement of Corollary~\ref{cor:high-noise}]
Consider the equally weighted Mallows $k$-mixture 
(that is, \eqref{eq:k-mix} with $w_1 = \dots = w_k = 1/k$) 
in the high-noise regime where $\phi$ is close to $1$. For fixed $n$ and $k$, the optimal sample complexity for recovering the central permutations is of the order $(\frac{1}{1-\phi})^{2 \lfloor \log_2 k \rfloor + 2}$. 
\end{theorem}


%

\subsection{Logarithmic sample complexity and groups of pairwise comparisons}
\label{sec:log-sc}

To motivate our main methodology based on \emph{pairwise comparisons}, we briefly discuss why the sample complexity for learning the central permutation $\pi$ in the single-component
Mallows model $M(\pi, \phi)$ scales as $\log n$.
Mallows showed in his original paper~\cite{Mal57} that, for indices $i, j \in [n]$ such that $\pi(i) < \pi(j)$, 
$$
\p_{\sigma \sim M(\pi, \phi)} \{\sigma(i) < \sigma(j)\} = 
\frac{ \pi(j) - \pi(i) + 1 }{ 1 - \phi^{ \pi(j) - \pi(i) + 1 } } - \frac{ \pi(j) - \pi(i) }{ 1 - \phi^{ \pi(j) - \pi(i) } } \ge \frac 12 + \frac{1- \phi}{4} . 
$$
In other words, the probability that a random permutation $\sigma$ from $M(\pi, \phi)$ agrees with $\pi$ on $\{i, j\}$ is at least $1/2$ plus the positive constant $\frac{1- \phi}{4}$. 
Therefore by Hoeffding's inequality, given $N$ i.i.d.\ random permutations from $M(\pi, \phi)$, a simple majority vote recovers $\1\{\pi(i) < \pi(j)\}$ correctly with probability at least $1 - e^{-c (1-\phi)^2 N}$ for a constant $c > 0$.
As a result, if $N \ge \frac{C \log n}{ (1-\phi)^2 }$ for a constant $C > 0$, by a union bound, we readily obtain $\1\{\pi(i) < \pi(j)\}$ for all distinct $i, j \in [n]$ with high probability,  from which any comparison sort algorithm (such as Quicksort or Heapsort) can be used to recover the central permutation $\pi$. 

Crucially, the size $n$ of the permutations does not affect the sample complexity of learning each pairwise comparison $\1\{\pi(i) < \pi(j)\}$. Instead, $n$ enters the overall sample complexity only through a union bound of exponentially small probabilities, so that the dependency on $n$ is logarithmic. 
In fact, this high-level strategy generalizes to the case of learning the Mallows $k$-mixture. 
However, the caveat is that pairwise comparisons alone are no longer sufficient for identifying a mixture of permutations; 
as such, we need to consider \emph{groups of pairwise comparisons}. 
This framework of demixing permutations using groups of pairwise comparisons is rigorously developed in Section~\ref{sec:mixture} under a noiseless oracle model, which is of independent interest. 
Later in Section~\ref{sec:high-dim}, we extend these results to the noisy case by simulating the noiseless oracle using logarithmically many observations drawn from the Mallows mixture model.

\subsection{Method of moments and comparison with Gaussian mixtures}

In the high-noise regime where $\phi \to 1$, the sample complexity $(\frac{1}{1-\phi})^{2 \lfloor \log_2 k \rfloor + 2}$ for learning the Mallows $k$-mixture is achieved by a method of moments of combinatorial flavor, which we now explain informally.  
For a distribution on the set $\cS_n$ of permutations, it is not obvious how to define an appropriate notion of moments. 
We show in Section~\ref{sec:comb-moment} that, in fact, it is natural to view the set of all groups of $m$ pairwise comparisons as the $m$th-order moment of 
the mixing distribution $\sum_{i=1}^k w_i \delta_{\pi_i}$ associated with the Mallows mixture $\cM = \sum_{i=1}^k w_i M(\pi_i, \phi)$. 
Moreover, the exponent of $\frac{1}{1-\phi}$ in the optimal sample complexity is precisely determined by the maximum number of moments two distinct mixtures can match. 
Namely, there exist two distinct $k$-mixtures with the same first $\lfloor \log_2 k \rfloor$ moments, but any $k$-mixture can be identified from the first $\lfloor \log_2 k \rfloor + 1$ moments, giving rise to the optimal sample complexity $(\frac{1}{1-\phi})^{2 \lfloor \log_2 k \rfloor + 2}$. 
From this perspective, learning a Mallows mixture from groups of pairwise comparisons can be viewed as a combinatorial method of moments.

Furthermore, we draw a comparison between the Mallows mixture and the better-studied Gaussian mixture~\cite{Pea94}.
Specifically, consider the $k$-component $n$-dimensional Gaussian location mixture $\sum_{i=1}^k w_i \, \cN(\mu_i, I_n)$, where $n$ and $k$ are both fixed constants.  
It is known~\cite{MV2010,HK2015,WY18,DosWuYanZho20} that the sharp sample complexity of learning the mixing distribution $\sum_{i=1}^k w_i \delta_{\mu_i}$ up to an error $\eps$ in the Wasserstein $W_1$-distance is of the order $\eps^{4k-2}$, which can be achieved by a version of the method of moments. 
In contrast to the exponential growth of the sample complexity in the Gaussian mixture model, for Mallows mixtures the optimal sample complexity scales \emph{polynomially} with the number of components, thanks to the discrete nature of permutations.

\subsection{Relation to Zagier's work on group determinant}
\label{sec:zagier}

It is worth mentioning that the identifiability of the Mallows mixture model is related to a result of Zagier in mathematical physics~\cite{Zag92}. 
In \cite[Theorem 2]{Zag92}, Zagier computed the determinant of the matrix $A(\phi) \in \R^{n! \times n!}$ indexed by permutations in $\cS_n$ and defined by 
\begin{equation}
A(\phi)_{\pi, \sigma} \defn \phi^{\dkt(\pi, \sigma)}.
\label{eq:Aphi}
\end{equation}
This is an instance of the \emph{group determinant} associated with the symmetric group $\calS_n$; see Section~\ref{sec:conj} for details.
In particular, Zagier showed that
\begin{equation}
\det(A(\phi)) \neq 0, \quad \text{ for all } \phi \in (0, 1). 
\label{eq:zagier-A}
\end{equation}

Note that, up to the normalization factor $1/Z(\phi)$, the row of $A(\phi)$ indexed by $\pi$ is precisely the probability mass function (PMF) of the Mallows model $M(\pi, \phi)$. 
Moreover, the rows of $A(\phi)$ are linearly independent since the determinant of $A(\phi)$ is nonzero. 
Therefore, if two Mallows mixtures $\sum_{i=1}^k w_i M(\pi_i, \phi)$ and $\sum_{i=1}^k w'_i M(\pi'_i, \phi)$ are identical, then the two sets of central permutations must coincide and so do the corresponding weights. 
Therefore, Zagier's result implies the \emph{identifiability} of the Mallows mixture. 

However, in the finite-sample setting, as noted by Liu and Moitra~\cite{LiuMoi18}, the direct quantitative implication of~\cite{Zag92} is very weak, as it only guarantees a sample complexity that is exponential in $n$ for learning the mixture. 
While the sample complexity is reduced to a polynomial in $n$ in \cite{LiuMoi18}, in this paper we take a step further to achieve the optimal logarithmic sample complexity. As in~\cite{LiuMoi18}, we also use Zagier's result as a building block; see Lemma~\ref{lem:pre-bd}. 

Furthermore, we remark that another group determinant (defined in \eqref{eq:def-m} which is a variant of the one studied in~\cite[Section 3]{Zag92}) appears naturally in one of our technical proofs. See Section~\ref{sec:conj} for details.


\subsection{Organization}

The remainder of the paper is organized as follows. In Section~\ref{sec:mixture}, we define groups of pairwise comparisons and interpret them as moments of a mixture. Moreover, we study learning a mixture of permutations from groups of pairwise comparisons under a generic, noiseless model. 
Extending these results to the noisy case, in Section~\ref{sec:high-dim}, we consider the Mallows mixture and present an algorithm that achieves the sample complexity logarithmic in the size of the permutations. 
In Section~\ref{sec:fin-dim}, we study the sample complexity of learning the Mallows mixture in the high-noise regime. 
Section~\ref{sec:discuss} discusses potential extensions of our results and proof techniques. 
The proofs are presented in Section~\ref{sec:proofs}. 

\subsection{Notation}
\label{sec:notation}

Let $[n] \defn \{1, \dots, n\}$ and $\naturals\triangleq\{1,2,\ldots\}$. 
Let $\TV(\cP, \cQ)$ stand for the total variation distance between two probability distributions $\cP$ and $\cQ$. 

Let $\cS_n$ denote the set of permutations on $[n]$. 
When presenting concrete instances of permutations, we use the notation 
$$
\pi = \big( \pi^{-1}(1) , \pi^{-1}(2) , \cdots , \pi^{-1}(n) \big) , 
$$
so that when $\pi$ is understood as a ranking, $\pi^{-1}(i)$ is the element that is ranked in the $i$th place by $\pi$. For example, $(3, 2, 4, 1)$ denotes the permutation $\pi$ with $\pi(3) = 1$, $\pi(2) = 2$, $\pi(4) = 3$ and $\pi(1) = 4$. 

For a permutation $\pi \in \cS_n$ and a subset $J \subset [n]$, we use the notation $\pi(J) \defn \{\pi(j) : j \in J\}$. 
We let $\pi|_J$ denote the restriction of $\pi$ on $J$, which is an injection from $J$ to $[n]$. 
Moreover, let $\pi\|_J$ denote the bijection from $J$ to $[|J|]$ induced by $\pi|_J$. 
That is, if $\sigma$ is the increasing bijection from $\pi(J)$ to $[|J|]$, then $\pi\|_J = \sigma \circ \pi|_J$.

For example, consider $\pi = (3, 2, 4, 6, 1, 5)$ and $J = \{1, 4, 5\}$. Then $\pi|_J(1) = 5$, $\pi|_J(4) = 3$ and $\pi|_J(5) = 6$, while $\pi\|_J(1) = 2$, $\pi\|_J(4) = 1$ and $\pi\|_J(5) = 3$. 
We also write $\pi\|_J = (4, 1, 5)$, which can be easily obtained from the notation $\pi = (3, 2, 4, 6, 1, 5)$ by retaining only the elements of $J$. 

Note that $\pi\|_J$ can be viewed as a total order on $J$. Moreover, by identifying the elements of $J$ with $1, \dots, |J|$ in the ascending order, we can identify bijections from $J$ to $[|J|]$ with permutations in $\cS_{|J|}$. 
Hence $\pi\|_J$ can be equivalently understood as a permutation in $\cS_{|J|}$. 
We may therefore refer to $\pi\|_J$ informally as a permutation or a relative order on $J$. 
Moreover, for nested sets $J \subset J' \subset [n]$, we clearly have $( \pi\|_{J'} ) \|_J = \pi\|_J$. 


\section{Demixing permutations with groups of pairwise comparisons} 
\label{sec:mixture}

In this section, we set up a general approach to learning mixtures of permutations: We first formalize the notions of groups of pairwise comparisons and comparison moments, and then characterize when a mixture of permutations can be learned from groups of pairwise comparisons in a generic noiseless model.

\subsection{Groups of pairwise comparisons}

Let $\sM$ denote a distribution on $\cS_n$. 
In this work, we are interested in the situation where $\sM$ is a certain model for a mixture of permutations. 
To motivate the method of learning the mixture $\sM$ from groups of pairwise comparisons, let us first consider some simple examples: 
\begin{itemize}[leftmargin=*]
\item
If $\sM$ is the Dirac delta measure $\delta_{\pi}$ for a fixed permutation $\pi \in \cS_n$, we are tasked with identifying the single permutation $\pi$. 
Let us consider the \emph{pairwise comparison oracle}: Given any pair of distinct indices $(i, j) \in [n]^2$, the oracle returns whether $i$ is placed before $j$ by $\pi$, that is, $\1 \{ \pi(i) < \pi(j) \}$. Based on this oracle, any comparison sorting algorithm (for example, quicksort) can be deployed to identify $\pi$. 

\item
For a general distribution $\sM$,
the pairwise comparison oracle naturally extends to the following: 
Given any pair of distinct indices $(i, j) \in [n]^2$, the oracle returns 
the \emph{distribution} of $\1 \{ \pi(i) < \pi(j) \}$ where $\pi \sim \sM$. 

However, as pointed out by Awasthi et al.~\cite{Awaetal14}, even for the noiseless $2$-mixture $\sM = \frac 12 ( \delta_{\pi_1} + \delta_{\pi_2})$, the pairwise comparison oracle is not sufficient for identifying $\sM$. 
For example, if the permutations $\pi_1$ and $\pi_2$ are reversals of each other, then for any pair of distinct indices $(i, j)$, 
the output of the pairwise comparison oracle is always $\Bern(\frac{1}{2})$, which is uninformative. 

\item
Now that comparing one pair of indices at a time does not guarantee identifiability, how about comparing two pairs simultaneously? 
This motivates the following oracle that returns a group of two pairwise comparisons: Given pairs of distinct indices $(i_1, j_1), (i_2, j_2) \in [n]^2$, the oracle returns 
the distribution of 
$$
\begin{pmatrix}
\1\{\pi(i_1) < \pi(j_1)\} \\
\1\{\pi(i_2) < \pi(j_2)\} 
\end{pmatrix} ,
\text{ where } \pi \sim \sM .
$$
To illustrate why groups of two pairwise comparisons are sufficient for identifying a mixture of two permutations, we consider a mixture $\sM = \frac 12 ( \delta_{\pi_1} + \delta_{\pi_2})$ where the two permutations satisfy $\pi_1(1) < \pi_1(2)$ and $\pi_2(1) > \pi_2(2)$. 
When we make a query on the group of pairs $(1, 2), (i, j)$ for any distinct indices $i, j \in [4]$, the oracle returns the mixture of two delta measures at 
$$
\begin{pmatrix}
1 \\
\1\{\pi_1(i) < \pi_1(j)\} 
\end{pmatrix}
\quad \text{ and } \quad
\begin{pmatrix}
0 \\
\1\{\pi_2(i) < \pi_2(j)\} 
\end{pmatrix}
$$ 
respectively. 
Therefore, using the pair $(1, 2)$ as a signature for the two permutations in the mixture, we can demix the pairwise comparisons $\1\{\pi_1(i) < \pi_1(j)\}$ and $\1\{\pi_2(i) < \pi_2(j)\}$ for every pair of indices $(i, j)$, from which $\pi_1$ and $\pi_2$ can be recovered. 

It turns out that this argument can be made rigorous and extended to the case of a general $k$-mixtures (Theorem~\ref{thm:min-pair}). 

%
\end{itemize}


Given these considerations, we are ready to formally define a group of pairwise comparisons. 

\begin{definition}[Group of $m$ pairwise comparisons, the strong oracle]
\label{def:group-pair}
Consider a distribution $\sM$ on $\cS_n$ and a random permutation $\pi \sim \sM$. 
For $m \in \naturals$, let $\cI$ be the tuple of $m$ pairs of distinct indices $(i_1, j_1), \dots, (i_m, j_m) \in [n]^2$. 
Upon a query on $\cI$, the strong oracle of group of $m$ pairwise comparisons returns the distribution of the random vector $\chi(\pi, \cI)$ in $\{0, 1\}^m$, whose $r$th coordinate is defined by
\begin{align}
\chi(\pi, \cI)_r \defn \1 \{ \pi(i_r) < \pi(j_r) \}
\quad \text{ for } r \in [m] . 
\label{eq:def-chi}
\end{align}
\end{definition}

We emphasize that in the tuple $\cI$ of pairs of distinct indices, $i_r$ and $j_r$ are required to be distinct for each $r \in [m]$, but we allow the scenarios where $i_1 = i_2$ or $i_1 = j_2$, for example. 
Moreover, throughout this work, the queries we consider are \emph{adaptive}: Our algorithms make queries to the oracle in a sequential fashion, where a given query is allowed to depend on the outcomes of previous ones. 

In addition, we introduce a weaker oracle of group of pairwise comparisons. 
This definition is motivated by interpreting a ``mixture'' as a set of permutations in $\cS_n$, rather than a distribution. 

\begin{definition}[Group of $m$ pairwise comparisons, the weak oracle] \label{def:weak-group}
Consider a set $\{ \pi_1, \dots, \pi_k \}$ of $k$ permutations in $\cS_n$. 
For $m \in \naturals$, let $\cI$ be a tuple of $m$ pairs of distinct indices in $[n]$. 
Upon a query on $\cI$, the weak oracle of group of $m$ pairwise comparisons returns the set of binary vectors $\{ \chi(\pi_i, \cI) : i \in [k] \}$, where $\chi(\pi_i, \cI)$ is defined by \eqref{eq:def-chi}. 
\end{definition}

If $\sM$ is a distribution on $\cS_n$ supported on $\{\pi_1, \dots, \pi_k\}$, then the set $\{ \chi(\pi_i, \cI) : i \in [k] \}$ returned by Definition~\ref{def:weak-group} is simply the support of the random vector $\chi(\pi, \cI)$ returned by Definition~\ref{def:group-pair}. In this sense, the oracle in Definition~\ref{def:weak-group} is weaker. 
If $|\supp(\sM)|=k$, then the strong and the weak oracle are equivalent; otherwise the weak oracle is strictly less informative. In the special case of $k=2$, they are always equivalent. 
We emphasize that the weak oracle only returns $\{ \chi(\pi_i, \cI) : i \in [k] \}$ as a collection of (possibly less than $k$) distinct, unlabeled elements---it does not specify what each $\chi(\pi_i, \cI)$ is. 
This weaker notion will be useful later when we study noisy mixtures of permutations.

Besides groups of pairwise comparisons, it is also natural to consider $\ell$-wise comparisons, whose strong and weak versions are defined as follows. 
Recall the notation $\pi\|_J$ for relative order as defined Section~\ref{sec:notation}.

\begin{definition}[$\ell$-wise comparison, the strong oracle]
\label{def:listwise}
Consider a distribution $\sM$ on $\cS_n$ and a random permutation $\pi \sim \sM$. 
For $\ell \in \naturals$, let $J$ be a subset of $[n]$ of cardinality $|J| = \ell$. 
Upon a query on $J$, the strong oracle of $\ell$-wise comparison returns the distribution of the relative order $\pi\|_J$.
\end{definition}

\begin{definition}[$\ell$-wise comparison, the weak oracle]
\label{def:listwise-weak}
Consider a set $\{ \pi_1, \dots, \pi_k \}$ of $k$ permutations in $\cS_n$. 
For $\ell \in \naturals$, let $J$ be a subset of $[n]$ of cardinality $|J| = \ell$. 
Upon a query on $J$, the weak oracle of $\ell$-wise comparison returns the set of relative orders $\{ \pi_i\|_J : i \in [k] \}$.
\end{definition}

For $\ell = 2$, the oracle of $\ell$-wise comparison simply reduces to the pairwise comparison oracle. 
Moreover, for $\ell = 2m$, the (strong or weak) oracle of $\ell$-wise comparison is stronger than the corresponding oracle of group of $m$ pairwise comparisons. 
This is because for any tuple $\cI$ of $m$ pairs of indices in $[n]$, we can choose $J \subset [n]$ with $|J| = \ell = 2m$ that contains all indices appearing in $\cI$. 
Then, for any permutation $\pi_i$, we can obtain the binary vector $\chi(\pi_i, \cI)$ from the relative order $\pi_i\|_J$.

\subsection{Comparison moments} \label{sec:comb-moment}

We now interpret groups of pairwise comparisons in Definition~\ref{def:group-pair} as moments of the random permutation $\pi$. Toward this end, we adopt the following notation throughout this paper. 
For any (random) permutation $\pi$ in $\cS_n$ and a pair of distinct indices $(i, j) \in [n]^2$, we define 
\begin{align}
X^\pi_{i, j} \defn \1 \{ \pi(i) < \pi(j) \} . 
\label{eq:X-pi}
\end{align}
In this work, we frequently identify the permutation $\pi$ with the array $X^\pi = \{X^\pi_{i,j}\}_{i \ne j}$. 
There is certainly redundancy in $X^\pi$ as we lift $\pi \in \cS_n$ to $X^\pi \in \{0, 1\}^{ n^2-n }$. For example, $X^\pi_{i, j} + X^\pi_{j, i} = 1$, and if $X^\pi_{i,j} = 1$ and $X^\pi_{j,k} = 1$, then we must have $X^\pi_{i,k} = 1$. 

In Definition~\ref{def:group-pair}, consider the oracle that returns the distribution of $\chi(\pi, \cI)$ in the form of its PMF:
$$
f_{\chi(\pi, \cI)} (v) \defn \p \{ \chi(\pi, \cI) = v \} 
\quad \text{ for each } v \in \{0, 1\}^m . 
$$ 
For example, at the all-ones vector $\bone_m \in \{0, 1\}^m$, 
\begin{align*}
f_{\chi(\pi, \cI)}( \bone_m ) = \E[ \1 \{ \chi(\pi, \cI) = \bone_m \} ] = \E \Big[ \prod_{r=1}^m \1 \{ \pi(i_r) < \pi(j_r) \} \Big] = \E \Big[ \prod_{r=1}^m X^\pi_{i_r, j_r} \Big] , 
\end{align*}
which is an $m$th moment of $X^\pi$. This motivates the following definition. 

\begin{definition}[Comparison moment] \label{def:comb-moment}
Consider a distribution $\sM$ on $\cS_n$. 
For a random permutation $\pi \sim \sM$, let $X^\pi$ be defined by \eqref{eq:X-pi}. 
For $m \in \naturals$, let $\cI$ denote the tuple of $m$ pairs of distinct indices $(i_1, j_1), \dots, (i_m, j_m) \in [n]^2$. 
The comparison moment of $\pi$ with index $\cI$ is the vector $\moment(\pi, \cI) \in \R^{2^m}$, defined by 
\begin{align}
\moment(\pi, \cI)_v \defn \E \Big[ \prod_{r=1}^m \big( X^\pi_{i_r, j_r} \big)^{v_r} \big( 1-X^\pi_{i_r, j_r} \big)^{1-v_r} \Big] 
\quad \text{ for } v \in \{0, 1\}^m .
\label{eq:def-moment}
\end{align}
\end{definition}

Note that the comparison moment defined above is of order at most $m$ in the usual sense, as 
\begin{align}
\big( X^\pi_{i_r, j_r} \big)^{v_r} \big( 1-X^\pi_{i_r, j_r} \big)^{1-v_r} = 
\begin{cases}
X^\pi_{i_r, j_r} & \text{ if } v_r = 1 , \\
X^\pi_{j_r, i_r} & \text{ if } v_r = 0 .
\end{cases}
\label{eq:switch}
\end{align}
Moreover, by \eqref{eq:def-chi}, \eqref{eq:X-pi} and \eqref{eq:def-moment}, we see that the PMF of the random vector $\chi(\pi, \cI)$ is precisely the comparison moment $\moment(\pi, \cI)$ as 
\begin{align*}
f_{\chi(\pi, I)} (v) = \E[ \1 \{ \chi(\pi, \cI) = v \} ] &= \E \Big[ \prod_{r=1}^m \1 \big\{ X^\pi_{i_r, j_r} = v_r \big\} \Big] \\
& = \E \Big[ \prod_{r=1}^m \big( X^\pi_{i_r, j_r} \big)^{v_r} \big( 1-X^\pi_{i_r, j_r} \big)^{1-v_r} \Big] 
= \moment(\pi, \cI)_v . 
\end{align*}
As a result, the group of pairwise comparisons on $\cI$ can be equivalently defined as the oracle that returns the comparison moment $\moment(\pi, \cI)$. Learning a mixture of permutations from groups of pairwise comparisons can therefore be viewed a combinatorial method of moments.

\subsection{Efficient learning in a generic model}
\label{sec:noiseless}

With the above definitions formulated, we are ready to study demixing permutations with groups of pairwise comparisons or $\ell$-wise comparisons. 
In this section, we consider the following generic noiseless model for a mixture of $k$ permutations: 
$$
\sM \defn \sum_{i=1}^k w_i \delta_{\pi_i},
$$ 
where $\pi_1, \dots, \pi_k$ are permutations in $\cS_n$ and $w_1, \dots, w_k$ are positive weights that sum to one.

It is clear that the more pairs we compare in a group, the more information we obtain. In other words, the larger $m$ is, the stronger the oracle in Definition~\ref{def:group-pair} becomes. 
Similarly, the larger $\ell$ is, the stronger the oracle in Definition~\ref{def:listwise} becomes. 
Is there a polynomial-time algorithm that learns the $k$-mixture $\sM$ from a polynomial number of groups of $m$ pairwise comparisons for any large $n$, where $m$ only depends on $k$ but not on $n$? 
Furthermore, for a fixed $k$, what is the weakest oracle we can assume, that is, what is the smallest $m$, so that such an algorithm exists? 
The analogous questions can also be asked for the oracle of $\ell$-wise comparison. 
As the main result of this section, the following theorem answers these questions. 



\begin{theorem}
\label{thm:min-pair}
Let $k$ be a positive integer, and define 
\begin{equation}
\mopt \defn \lfloor \log_2 k \rfloor + 1. 
\label{eq:mopt}
\end{equation}
\begin{enumerate}[leftmargin=*,label={(\alph*)}]
\item
For any mixture $\sM = \sum_{i=1}^k w_i \delta_{\pi_i}$ of permutations in $\cS_n$,  
there is a polynomial-time algorithm that recovers $\sM$ from groups of $\mopt$ pairwise comparisons, 
with at most $1 + \frac k2 (n-2) (n+1)$ adaptive queries to the strong oracle in Definition~\ref{def:group-pair}.

\item
Conversely, for $n \ge 2 \mopt$ and $\ell \le 2 \mopt - 1$, there exist distinct mixtures $\sM = \frac 1k \sum_{i=1}^k \delta_{\pi_i}$ and $\sM' = \frac 1k \sum_{i=1}^k \delta_{\pi_i'}$ of permutations in $\cS_n$, which cannot be distinguished even if all $\binom{n}{\ell}$ $\ell$-wise comparisons are queried from the strong oracle in Definition~\ref{def:listwise}.
\end{enumerate}
\end{theorem} 

As we have noted, if $\ell \ge 2m$, then the oracle of $\ell$-wise comparison is stronger than the oracle of group of $m$ pairwise comparisons. 
Therefore, the above theorem implies: 
(1) The oracle of group of $m$ pairwise comparisons is sufficient for identifying the $k$ mixture if and only if $m \ge \mopt$; 
(2) The oracle of $\ell$-wise comparison is sufficient for identifying the $k$ mixture if and only if $\ell \ge 2 \mopt$. 



In addition to the above theorem which studies the permutation demixing problem assuming the strong oracles, we also have the following result that assumes the weak oracle given by Definition~\ref{def:weak-group}. 
Recall that here we view the mixture as a set of permutations rather than a distribution. 

\begin{theorem} \label{thm:assemble}
There is an $O(k^4 n^2)$-time algorithm (see Algorithm~\ref{alg:insert-demix} in \prettyref{sec:pf-assemble}) that learns a set $\{\pi_1, \dots, \pi_k\}$ of permutations in $\cS_n$ from groups of $k+1$ pairwise comparisons, with at most $1 + \frac k2 (n-2) (n-1)$ adaptive queries to the weak oracle in Definition~\ref{def:weak-group}. 
\end{theorem}

Unlike Theorem~\ref{thm:min-pair} where the smallest number $m$ of pairs compared in a query is precisely $m^*_k$, 
Theorem~\ref{thm:assemble} only shows that $m$ is at most $k+1$ and we do not have a matching lower bound. 
Nevertheless, the crucial observation is that $m$ again only depends on $k$, the number of components, but not on $n$, the size of the permutations.

The algorithms for Theorems~\ref{thm:min-pair}(a) and~\ref{thm:assemble} are similar in nature and are both generalizations of Insertion Sort. The latter, called Insertion Demixing, is detailed in Algorithm~\ref{alg:insert-demix}. 
Furthermore, note that the query complexity for both algorithms is $O(k n^2)$. 
This is not optimal in general: When $k = 1$ (single component), the problem reduces to comparison sort
and the optimal query complexity is $O(n \log n)$, which is achieved by Heapsort for example. 
However, Insertion Sort has query complexity $O(n^2)$, and because our algorithms are generalizations of Insertion Sort to the mixture setting, the query complexity with respect to $n$ cannot be improved. 
It is an interesting open problem to determine the optimal query complexity for the mixture models.

In addition to interest in their own right, the above results 
have laid the foundation for studying the Mallows mixture in the next two sections. 
On the one hand, Theorem~\ref{thm:assemble} provides a ``meta-algorithm'' for learning the central permutations, so it suffices to simulate the weak oracle using sample from the Mallows mixture, which we do in Section~\ref{sec:high-dim}. 
On the other hand, Theorem~\ref{thm:min-pair} sheds light on the fundamental limit of learning mixtures of permutations, which we further explore in Section~\ref{sec:fin-dim} for the Mallows mixture in the high-noise regime.

\section{Mallows mixture in high-dimensional regime}
\label{sec:high-dim}

Moving from the noiseless to the noisy case, we now turn to the popular Mallows mixture model. 
Denote the Kendall tau distance between two  permutations $\pi, \sigma \in \cS_n$ by $\dkt(\pi, \sigma)$ as defined in \eqref{eq:def-kt}. 
For a central permutation $\pi \in \cS_n$ and a noise parameter $\phi \in (0, 1)$, the Mallows model denoted by $M(\pi, \phi)$ is the distribution on $\cS_n$ with PMF  
\begin{align}
f_{M(\pi, \phi)} (\sigma) = \frac{ \phi^{\dkt(\sigma, \pi)} }{Z(\phi)} \quad \text{ for } \sigma \in \cS_n , \quad
\text{ where } Z(\phi) \defn \sum_{\sigma \in \cS_n} \phi^{\dkt(\sigma, \id)} . 
\label{eq:def-pdf}
\end{align}
Note that $\phi$ determines the noise level of the Mallows model. As $\phi \to 0$, $M(\pi, \phi)$ converges to the noiseless model, a delta measure at $\pi$. On the other hand, as $\phi \to 1$, $M(\pi, \phi)$ converges to the noisiest model, the uniform distribution on $\cS_n$. 
In fact, it is also common \cite{Meietal07,BraMos09,IruCalLoz19} to parametrize the noise level by $\beta = 1/ \log (1/\phi)$ so that $\phi = e^{-1/\beta}$. Particularly, we have $\beta \approx \frac{1}{1 - \phi} \to \infty$ as $\phi \to 1$.

In this work, we consider a mixture $\cM$ of $k$ Mallows models $M(\pi_1, \phi), \dots, M(\pi_k, \phi)$ with a common noise parameter $\phi \in (0, 1)$ and respective weights $w_1, \dots, w_k > 0$ such that $\sum_{i=1}^k w_i = 1$. In other words, $\cM$ is the distribution on $\cS_n$ with PMF 
$$
f_{\cM} (\sigma) = \sum_{i=1}^k w_i  \frac{ \phi^{\dkt(\sigma, \pi_i)} }{Z(\phi)} \quad \text{ for } \sigma \in \cS_n .  
$$ 
We also write $M(\pi_i) \equiv M(\pi_i, \phi)$ and 
$$
\cM = \sum_{i=1}^k w_i M(\pi_i) 
$$
for brevity. Note that if $\phi=0$, then $\cM$ reduces to the noiseless model $\sum_{i=1}^k w_i \delta_{\pi_i}$ considered in Section~\ref{sec:mixture}.

Furthermore, suppose that we are given $N$ i.i.d.\ observations $\sigma_1, \dots, \sigma_N$ 
from the mixture $\cM$. 
Let 
$$
\cM_N\triangleq \frac{1}{N} \sum_{i=1}^N \delta_{\sigma_i}
$$ 
denote the empirical distribution with PMF 
$$
f_{\cM_N}(\sigma) = \frac 1N \sum_{i=1}^N \1\{ \sigma_i = \sigma \} \quad \text{ for } \sigma \in \cS_n .
$$
Assuming that the number of components $k$ and the noise parameter $\phi$ are known, we aim to exactly recover the set of central permutations $\{\pi_1, \dots, \pi_k\}$ in the mixture. 
We assume the knowledge of $\phi$ for technical convenience. In principle, this assumption can be removed, which we discuss in Section~\ref{sec:discuss}.

\begin{wrapfigure}{r}{0.55\textwidth}
\vspace{-0.4cm}
\includegraphics[width=0.55\textwidth]{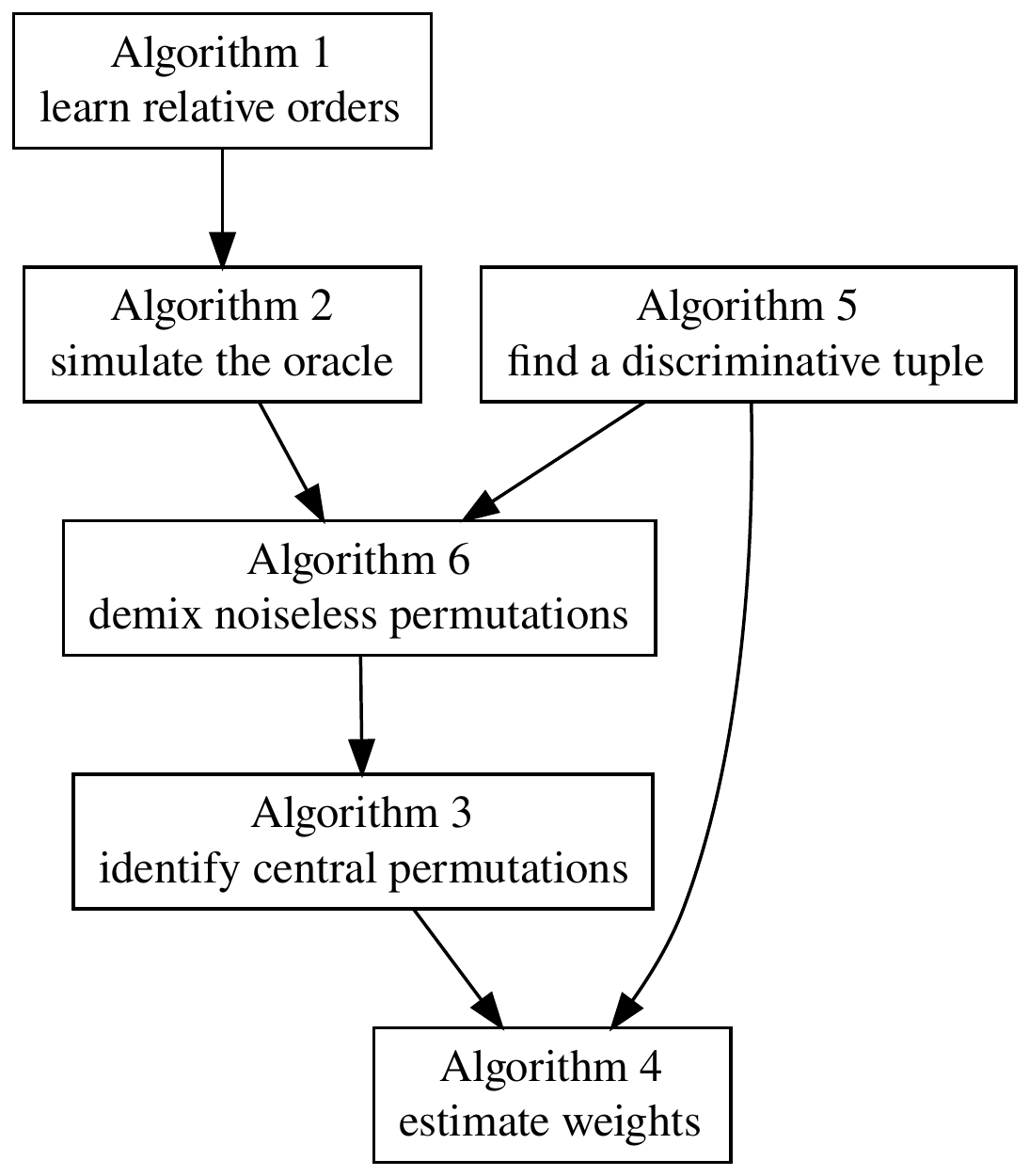}
\caption{Dependency graph of the algorithms.}
\vspace{-0.5cm}
\label{fig:dep}
\end{wrapfigure}

In this section, we consider the ``high-dimensional'' setting where the size $n$ of the permutations is large, and establish the logarithmic dependency of the sample complexity on $n$. 
As hinted earlier, our strategy is to use Algorithm~\ref{alg:insert-demix} (from Theorem~\ref{thm:assemble} for the noiseless case) as a ``meta-algorithm'' to recover the central permutations of the Mallows mixture. 
To this end, we need to simulate the weak oracle in Definition~\ref{def:weak-group} using noisy observations from the Mallows mixture (which is done by Algorithm~\ref{alg:oracle} below). 
Furthermore, recall that the weak oracle of $\ell$-wise comparison in Definition~\ref{def:listwise-weak} is stronger than the weak oracle of group of $m$ pairwise comparisons in Definition~\ref{def:weak-group}, provided that $\ell \ge 2m$. 
Therefore, a main goal of this section is to introduce a subroutine (Algorithm~\ref{alg:suborder}) which simulates the weak oracle in Definition~\ref{def:listwise-weak} using logarithmically many observations from the Mallows mixture. 

Figure~\ref{fig:dep} illustrates the dependency among various algorithms in this paper. 
Specifically, 
Algorithm~\ref{alg:suborder} learns a set of relative orders on a small set of indices given noisy samples from the Mallows mixture. 
Algorithm~\ref{alg:oracle} then uses it to simulate the key oracle of groups of pairwise comparisons. 
This oracle is repeatedly called by Algorithm~\ref{alg:insert-demix}, a recursion, which is the demixing algorithm for the noiseless case. 
Algorithm~\ref{alg:main} is the main algorithm that learns the central permutations for the Mallows mixture based on noisy observations. 
Given these exactly recovered central permutations, Algorithm~\ref{alg:weight} then estimates their respective weights in the mixture. 
Algorithm~\ref{alg:find-tuple} is a simple subroutine that is used in both Algorithms~\ref{alg:insert-demix} and~\ref{alg:weight}.

\subsection{Marginalization of Mallows mixture}

Given i.i.d.~observations $\sigma_1, \dots, \sigma_N$ from the Mallows mixture $\cM=\sum_{i=1}^k w_i M(\pi_i) $ and a subset $J \subset [n]$, the goal of Algorithm~\ref{alg:suborder}, denoted by \SR, is to learn the set of
relative orders $\pi_1\|_J, \dots, \pi_k\|_J$. 
We recall that the relative order $\pi_i\|_J$ is the bijection from $J$ to $[|J|]$ induced by $\pi_i|_J$; we are not aiming at recovering $\pi_i|_J$ itself. 

Toward this end, we consider the marginalization of the Mallows mixture, as well as the observations, as follows.  
For any distribution $\cM$ on $\cS_n$ and a set of indices $J \subset [n]$, we let $\cM|_J$ denote the marginal distribution of $\sigma|_J$ where $\sigma \sim \cM$. 
That is, the PMF of $\cM|_J$ is given by  
\begin{equation}
f_{\cM|_J} (\rho) = \p_{\sigma \sim \cM} \big\{ \sigma|_J = \rho \big\}  
\label{eq:fMJ}
\end{equation}
for every injection $\rho: J \to [n]$. 
Moreover, given $N$ i.i.d.\ observations $\sigma_1, \dots, \sigma_N$  from $\cM$, 
the empirical version of \eqref{eq:fMJ} 
is given by 
\begin{align}
f_{\cM_N|_J} (\rho) = \frac 1N \sum_{m=1}^N \1 \big\{ \sigma_m|_J = \rho \big\} . 
\label{eq:emp}
\end{align}

Note that although our goal is to learn the relative order $\pi_i\|_J : J \to [|J|]$ for $i \in [k]$, not the actual values of $\pi_i(j)$ for $j \in J$, the marginalization is with respect to the restriction on $J$ only, and does maintain the values of $\sigma(j)$ for $j \in J$. This is crucial to establishing the following identifiability result for marginalized Mallows mixtures. 

\begin{proposition} \label{prop:tv-lower}
Consider the Mallows mixtures $\cM = \sum_{i=1}^k w_i M(\pi_i)$ and $\cM' = \sum_{i=1}^k w'_i M(\pi'_i)$ on $\cS_n$ with a common noise parameter $\phi \in (0, 1)$. Let us define $\gamma \defn \min_{i \in [k]} (w_i \land w'_i) > 0$. 
Fix a set of indices $J \subset [n]$ and let $\ell \defn |J|$.  
Suppose that the two sets of central permutations $\{\pi_1\|_J, \dots, \pi_k\|_J\}$ and $\{\pi'_1\|_J, \dots, \pi'_k\|_J\}$ are not equal (as sets).  
Then we have 
$$
\TV( \cM|_J, \cM'|_J ) \ge \eta(k, \ell, \phi, \gamma) ,
$$ 
where 
\begin{align}
\eta(k, \ell, \phi, \gamma) \defn 
\Big( \frac{ \gamma }{ 6 k } \Big)^{ (3 \ell)^{\ell+1} } \Big( \frac{1-\phi}{\ell} \Big)^{ (4 \ell)^\ell + 2 k \ell^2 } . 
\label{eq:def-eta}
\end{align}
\end{proposition}

Crucially, the above lower bound is \emph{dimension-free}, that is, it does not depend on $n$. 
This is one of the two key ingredients (the other being the concentration inequality in Proposition~\ref{prop:tv-conv} below) that enable us to achieve a sample complexity that ultimately depends logarithmically on $n$. 
The proof of Proposition~\ref{prop:tv-lower} leverages the notion of \emph{block structure} introduced by Liu and Moitra~\cite{LiuMoi18}; see Section~\ref{sec:block} for details.

In addition, we observe a useful property of marginalized Mallows models. 

\begin{lemma} \label{lem:restrict}
For any subset $J \subset [n]$, if the central permutations $\pi, \pi' \in \cS_n$ satisfy $\pi|_J = \pi'|_J$, then the marginalized Mallows models $M(\pi, \phi)|_J$ and $M(\pi', \phi)|_J$ coincide for all $\phi \in (0,1)$. 
\end{lemma}

\begin{proof}
Let $\tau \in \cS_n$ be a relabeling of indices $1, \dots, n$ such that $\pi' = \pi \circ \tau$. Since $\pi|_J = \pi'|_J$, we have $\tau(j) = j$ for every $j \in J$. 
It follows that $(\sigma \circ \tau)|_J = \sigma|_J$ for any $\sigma \in \cS_n$. 
Moreover, it holds that $\dkt(\sigma \circ \tau, \pi') = \dkt(\sigma \circ \tau, \pi \circ \tau) = \dkt(\sigma, \pi)$ by the right invariance of the Kendall-tau distance. In view of the definition of the Mallows model and marginalization on $J$, we reach the conclusion. 
\end{proof}

Proposition~\ref{prop:tv-lower} and Lemma~\ref{lem:restrict} together motivate the subroutine introduced in the sequel.

\subsection{The subroutine}

We are ready to define the subroutine formally. 
The first step is to define a set of polynomially many candidate models. 
Let $\cS_{n,J}$ denote the set of injections $\rho : J \to [n]$, which has cardinality at most $n^\ell$ where $\ell = |J|$. For each $\rho \in \cS_{n,J},$ fix an arbitrary permutation $\pi_\rho$ in $\cS_n$ such that $\pi_\rho|_J = \rho$. 
Let $L$ be a positive integer to be determined later. 
For $\phi \in (0, 1)$ and $\gamma \in (0, 1/k]$, we define a set of Mallows mixtures by discretizing the weights
\begin{align}
\scrM \equiv \scrM(n, k, \phi, \gamma, J, L) \defn 
\bigg\{ \sum_{i=1}^k \frac{r_i}{L} M(\pi_{\rho_i}, \phi) :  \rho_i \in \cS_{n,J} , \, r_i \in [L] , \, r_i \ge \gamma L , \, \sum_{i=1}^k r_i = L \bigg\} .
\label{eq:mix-class}
\end{align}
Note that the weights $r_i/L$ sum to $1$ and each weight is at least $\gamma$. 
Since there are at most $L$ choices for each weight and at most $|\cS_{n,J}| \le n^\ell$ choices for each $\rho_i$, we have $|\scrM| \le L^k n^{k \ell}$.

In view of the total variation lower bound in Proposition~\ref{prop:tv-lower}, it is natural to consider the minimum-distance estimator that selects
the Mallows mixture model in $\scrM$ whose marginal is closest in total variation to that of the empirical distribution $\cM_N$; however, without an explicit formula for the marginalized distribution $\cM'|_J$ for $\cM' \in \scrM$ it is difficult to directly compute the total variation. 
Fortunately, we can efficiently sample from $\cM'|_J$ and thus approximate the marginalized distribution sufficiently well in polynomial time. 
This motivates Algorithm~\ref{alg:suborder}. 

\begin{algorithm}[ht]
\normalsize
\caption{\SR}
\label{alg:suborder}
\begin{algorithmic}[1]
\Input $\sigma_1, \dots, \sigma_N \in \cS_n$, $k \in \naturals$, $\phi \in (0,1)$, $\gamma \in (0, 1/k]$, $N' \in \naturals$, and $J \subset [n]$
\Output a set $\fR$ of relative orders on $J$
\State{$\ell \leftarrow |J|$}
\State{$\eta \leftarrow \eta(k, \ell, \phi, \gamma)$ as defined in \eqref{eq:def-eta}}
\State{$L \leftarrow \lceil 3k / \eta \rceil$}
\State{$\scrM \leftarrow \scrM(n, k, \phi, \gamma, J, L)$ as defined in \eqref{eq:mix-class}}
\State{$\cM_N|_J \leftarrow \frac{1}{N} \sum_{m=1}^{N} \delta_{ \sigma_m|_J }$}
\Comment{{\it \small compute the marginalized empirical distribution as in \eqref{eq:emp}}}
\State{$\fR \leftarrow \varnothing$}
\For{$\cM' = \sum_{i=1}^k \frac{r_i}{L} M(\pi_{\rho_i}, \phi) \in \scrM$} \label{line:model-p}
\State{generate $N'$ i.i.d.\ random permutations $\sigma'_1, \dots, \sigma'_{N'}$ from $\cM'$} 
\State{$\cM'_{N'}|_J \leftarrow \frac{1}{N'} \sum_{m=1}^{N'} \delta_{ \sigma'_m|_J }$} \label{line:model-p-e}
\If{$\TV( \cM'_{N'}|_J , \cM_N|_J ) \le \eta / 2$} 
\label{line:return}
\State{$\fR \leftarrow \{ \pi_{\rho_i}\|_J : i \in [k] \}$}
\EndIf
\EndFor
\State{\Return $\fR$}
\end{algorithmic}
\end{algorithm}

%

Let us remark that sampling from the Mallows model is computationally efficient with the help of the Repeated Insertion Model of Doignon, Peke\u{c}, and Regenwetter~\cite{DoiPekReg04} (see also Section~2.2.3 of~\cite{LuBou14}). 
In short, to sample from $M(\id, \phi)$, it suffices to start from the empty ranking and repeatedly insert index $i = 1, \dots, n$ into the current ranking at position $j \le i$ with probability $\phi^{i-j}/(1 + \phi + \cdots + \phi^{i-1})$. 
This sampling procedure can be easily done in $O(n^2)$ time. 
Furthermore, as an anonymous reviewer pointed out, the time complexity of each insertion step can be improved to $O(\log n)$ by considering a stochastic transition rule on a binary tree with the possible rank positions as the leaves, where the transition probabilities on each edge can be computed explicitly so that the probabilities of outputting each leaf agrees with the values specified above. 
As a result, sampling from the Mallows model can be done in $O(n \log n)$ time. 

Consequently, in Algorithm~\ref{alg:suborder}, the computation of $\cM'_{N'}|_J$ takes $O(N' n \log n)$ time (where $N'$ will be taken to be logarithmic in $n$). 
Moreover, since $\cM_N|_J$ and $\cM'_{N'}|_J$ are distributions with at most $N$ and $N'$ atoms respectively, computing $\TV( \cM'_{N'}|_J , \cM_N|_J )$ takes time less than $O \big( N N' n \big)$. 
Furthermore, as we have seen, there are at most $L^k n^{k \ell}$ candidate models where $L = \lceil 3k/\eta \rceil$. 
We conclude that Algorithm~\ref{alg:suborder} runs in $O\big( (\frac{3k}{\eta})^k N N' n^{k\ell + 1} \log n \big)$ time.

To analyze Algorithm~\ref{alg:suborder}, we first state a concentration inequality for the marginalized empirical distribution for the Mallows mixture. 

\begin{proposition} \label{prop:tv-conv}
For $J \subset [n]$, let $\cM|_J$ and $\cM_N|_J$ be the marginalized Mallows mixture and the marginalized empirical distribution defined by \eqref{eq:fMJ} and \eqref{eq:emp} respectively.  
For any $s \in (0,1)$, 
\begin{align}
\p \big\{ \TV( \cM|_J, \cM_N|_J ) > s \big\} \le \exp \Big( - N \frac{ 3 s }{ 10 } \Big) + 2 (2kq)^\ell \exp \Big( - N \frac{ s^2 }{ (2kq)^{2 \ell} } \Big)
\label{eq:tv-prob}
\end{align}
where $\ell \defn |J|$ and $q \defn 1 +  \frac{1}{1 - \phi} \log \frac{8 \ell}{s (1-\phi)}$. 
\end{proposition}

Similar to the total variation lower bound in Proposition~\ref{prop:tv-lower}, the above concentration inequality is also dimension-free (independent of $n$).
This is possible because although $\calM|_J$ is a distribution on $\Theta(n^{\ell})$ elements, its ``effective support size'' is independent of $n$ thanks to a basic property of the Mallows model (Lemma~\ref{lem:dev}).
Propositions~\ref{prop:tv-lower} and~\ref{prop:tv-conv} together enable us to establish the following theoretical guarantee for Algorithm~\ref{alg:suborder}. 

\begin{theorem}
\label{thm:subroutine}
Suppose that we are given i.i.d.\ observations $\sigma_1, \dots, \sigma_N$ from the Mallows mixture $\cM = \sum_{i=1}^k w_i M(\pi_i)$ on $\cS_n$ with a noise parameter $\phi \in (0, 1)$. 
Fix a set of indices $J \subset [n]$ and let $\ell \defn |J|$. 
Fix a positive constant $\gamma \le \min_{i \in [k]} w_i$ and a probability of error $\delta \in (0, 0.1)$. 
Let 
\begin{align}
\zeta(k, \ell, \phi, \gamma) \defn e^{ (9 \ell)^{\ell+1} } \Big( \frac{ k }{ \gamma } \Big)^{ (6 \ell)^{\ell+1} } \Big( \frac{\ell}{1-\phi} \Big)^{ 3 (4 \ell)^\ell +8 k \ell^2 }  . 
\label{eq:def-zeta}
\end{align}
If the sample size satisfies $N \ge \zeta(k, \ell, \phi, \gamma) \cdot \log \frac 1\delta$ and we choose an integer $N' \ge \zeta(k, \ell, \phi, \gamma) \cdot \log \frac n\delta$, then Algorithm~\ref{alg:suborder} returns the set of relative orders $\{ \pi_i\|_J : i \in [k] \}$ with probability at least $1 - \delta$ in $O\big( (\frac{3k}{\eta})^k N N' n^{k\ell + 1} \log n \big)$ time, where $\eta = \eta(k, \ell, \phi, \gamma)$ is defined in \eqref{eq:def-eta}. 
\end{theorem}

\subsection{Exact recovery of the central permutations}

Consider a set of indices $J \subset [n]$ and a tuple $\cI$ of pairs of distinct indices $(i_1, j_1), \dots, (i_m, j_m) \in J^2$. 
For any permutation $\pi \in \cS_n$, we have $\1\{ \pi\|_J(i_r) < \pi\|_J(j_r) \} = \1\{ \pi(i_r) < \pi(j_r) \}$ for $r \in [m]$ by the definition of the relative order $\pi\|_J$. 
Since Algorithm~\ref{alg:suborder} returns the set of relative orders $\{ \pi_i\|_J : i \in [k] \}$ with high probability, in particular, we can obtain the set of binary vectors $\{ \chi(\pi_i, \cI) : i \in [k] \}$, where $\chi(\pi_i, \cI)$ is defined by \eqref{eq:def-chi}. 
This step is formulated as Algorithm~\ref{alg:oracle}. 

\begin{algorithm}[ht]
\normalsize
\caption{\texttt{SimulateOracle}}
\label{alg:oracle}
\begin{algorithmic}[1]
\Input $\sigma_1, \dots, \sigma_N \in \cS_n$, $k \in \naturals$, $\phi \in (0,1)$, $\gamma \in (0, 1/k]$, $N' \in \naturals$, and a tuple $\cI$ of pairs of distinct indices $(i_1, j_2), \dots, (i_m, j_m) \in [n]^2$
\Output a set $V$ of vectors in $\{0, 1\}^m$ 
\State{$J \leftarrow$ the set of all indices that appear in the tuple $\cI$}
\State{$\fR \leftarrow$ the set of relative orders on $J$ returned by \SR\ (Algorithm~\ref{alg:suborder}) with inputs $\sigma_1, \dots, \sigma_N, k, \phi, \gamma, N',$ and $J$}
\State{$V \leftarrow \{ v^\tau : \tau \in \fR \}$, where $v^\tau \in \{0,1\}^m$ is defined by $v^\tau_r \defn \1\{\tau(i_r) < \tau(j_r)\}$ for $r \in [m]$}
\State{\Return $V$}
\end{algorithmic}
\end{algorithm}

Recall that the set $\{ \chi(\pi_i, \cI) : i \in [k] \}$ is precisely what we assume the weak oracle in  Definition~\ref{def:weak-group} returns. 
Therefore, this oracle is available with high probability for the Mallows mixture provided that the sample size $N$ is sufficiently large. 
Consequently, Algorithm~\ref{alg:insert-demix} (recall Theorem~\ref{thm:assemble}) can be used as a meta-algorithm to recover the central permutations $\{\pi_i : i \in [k]\}$ in the Mallows mixture. 
We formulate this main algorithm as Algorithm~\ref{alg:main}. 
Corollary~\ref{cor:main-sc} then provides theoretical guarantees for Algorithm~\ref{alg:main}.

\begin{algorithm}[ht]
\normalsize
\caption{\texttt{DemixMallows}}
\label{alg:main}
\begin{algorithmic}[1]
\Input $\sigma_1, \dots, \sigma_N \in \cS_n$, $k \in \naturals$, $\phi \in (0,1)$, $\gamma \in (0, 1/k]$, and $\delta \in (0,0.1)$
\Output a set $\fS$ of permutations in $\cS_n$
\State{$N' \leftarrow \zeta(k, 2k+2, \phi, \gamma) \cdot \log \frac{n^{2k+3}}{\delta}$, where $\zeta$ is defined in \eqref{eq:def-zeta}}
\State{$\fS \leftarrow$ the set of permutations in $\cS_n$ returned by \ID\ (Algorithm~\ref{alg:insert-demix}) with the oracle given by \texttt{SimulateOracle} (Algorithm~\ref{alg:oracle}) with inputs $\sigma_1, \dots, \sigma_N,$ $k, \phi, \gamma$, and $N'$} 
\State{\Return $\fS$}
\end{algorithmic}
\end{algorithm}

\begin{corollary} \label{cor:main-sc}
Suppose that we are given i.i.d.\ observations $\sigma_1, \dots, \sigma_N$ from the Mallows mixture $\cM = \sum_{i=1}^k w_i M(\pi_i)$ on $\cS_n$ with a known noise parameter $\phi \in (0, 1)$.
Fix a positive constant $\gamma \le \min_{i \in [k]} w_i$ and a probability of error $\delta \in (0,0.1)$. 
If the sample size satisfies $N \ge \zeta(k, 2k+2, \phi, \gamma) \cdot \log \frac{n^{2k+2}}{\delta}$ where $\zeta$ is defined in \eqref{eq:def-zeta}, then with probability at least $1 - \delta$,  Algorithm~\ref{alg:main} successfully returns the set of central permutations $\{ \pi_1, \dots, \pi_k \}$ with time complexity $O \Big( N n^{2(k^2+k+2)} \big(\frac{e k}{\gamma(1-\phi)}\big)^{(2 k)^{2k+9}} \log \frac{1}{\delta} \Big)$. 
\end{corollary}

\begin{proof}
It suffices to show that Algorithm~\ref{alg:suborder} indeed simulates the oracle in Definition~\ref{def:weak-group} with high probability, so that Algorithm~\ref{alg:insert-demix} returns the set of permutations $\{ \pi_1, \dots, \pi_k \}$ as guaranteed by Theorem~\ref{thm:assemble}. 
More precisely, we need to prove that Algorithm~\ref{alg:oracle} returns the set of binary vectors $\{ \chi(\pi_i, \cI) : i \in [k] \}$ for every tuple $\cI$ of $k+1$ pairs of distinct indices in $[n]$ with probability at least $1 - \delta$. 

Recall that $J$ is defined to be the set $\{i_1, j_1, \dots, i_m, j_m\}$ if $\cI$ consists of the pairs $(i_1, j_1)$, $\dots$, $(i_m, j_m)$. 
As we have noted at the beginning of this subsection, it then holds that $\1\{ \pi\|_J(i_r) < \pi\|_J(j_r) \} = \1\{ \pi(i_r) < \pi(j_r) \}$ for any $\pi \in \cS_n$ and $r \in [m]$. 
Therefore, Algorithm~\ref{alg:oracle} returns the set of binary vectors $\{ \chi(\pi_i, \cI) : i \in [k] \}$ whenever Algorithm~\ref{alg:suborder} returns the set of relative orders $\{ \pi_i\|_J : i \in [k] \}$. 

Moreover, Algorithm~\ref{alg:insert-demix} requires the tuple $\cI$ to consist of $k+1$ pairs of indices. Hence $J$ has cardinality at most $2k+2$. Since there are less than $n^{2k+2}$ possible subsets $J$ of $[n]$ that have cardinality at most $2k+2$, we can replace the error probability $\delta$ in Theorem~\ref{thm:subroutine} by $\delta n^{-2k-2}$ and take a union bound to guarantee that Algorithm~\ref{alg:suborder} returns $\{ \pi_i\|_J : i \in [k] \}$ for all such $J$ with probability at least $1 - \delta$. 
This then guarantees the success of Algorithm~\ref{alg:oracle} in simulating the oracle. 
(Note that although Algorithm~\ref{alg:insert-demix} only makes at most $1 + \frac k2 (n-2) (n-1)$ queries to the oracle of Definition~\ref{def:weak-group} (see Theorem~\ref{thm:assemble}), we still need to take the union bound over all possible subsets of $[n]$ that have cardinality at most $2k+2$ because the queries are made adaptively.)

Finally, for the time complexity, recall that Algorithms~\ref{alg:insert-demix} runs in $O(k^4 n^2)$ time and requires $O(k n^2)$ queries (Theorem~\ref{thm:assemble}). 
For each query, Algorithm~\ref{alg:oracle} simulates the oracle and the bottleneck of time complexity lies in Algorithm~\ref{alg:suborder}. 
We take $\ell \le 2k+2$ in Algorithm~\ref{alg:suborder}, giving a time complexity $O\big( (\frac{3k}{\eta})^k N N' n^{2k^2 + 2k + 1} \log n \big)$ according to Theorem~\ref{thm:subroutine} where $\eta = \eta(k, 2k+2, \phi, \gamma)$. 
Therefore, the overall time complexity is $O \Big(N n^{2(k^2+k+2)} \big(\frac{e k}{\gamma(1-\phi)}\big)^{(2 k)^{2k+9}} \log \frac{1}{\delta} \Big)$ by plugging in the definitions of $\eta$, $N'$, and $\zeta$ and then simplifying the formula. 
\end{proof}

Note that the factor $N$ in the time complexity can be easily incorporated into the other factors, because we can just use a logarithmic number of samples in the algorithm and ignore the rest, which will not hurt the theoretical guarantee on recovering the central permutations.

We remark that the logarithmic dependency of the sample complexity $N$ on the size $n$ of the permutations is optimal, even in the case $k = 1$ where we aim to learn a single central permutation in the Mallows model. 
More precisely, the proof of Lemma~10 of \cite{Busetal19} established the following information-theoretic lower bound: 
Given $N$ random observations from the Mallows model $M(\pi, 1/2)$ on $\cS_n$, if $N \le c \log n$ for a sufficiently small constant $c>0$, then any algorithm fails to exactly recover the central permutation $\pi$ with a constant probability. 

\subsection{Learning the weights}

Once the central permutations in the Mallows mixture are recovered exactly according to Corollary~\ref{cor:main-sc}, their corresponding weights can be learned as well. 
To see the identifiability of the weights, we first establish a total variation bound for two Mallows mixtures with the same set of central permutations but different weights. 

\begin{proposition} \label{prop:tv-weight}
Consider the Mallows mixtures $\cM = \sum_{i=1}^k w_i M(\pi_i)$ and $\cM' = \sum_{i=1}^k w'_i M(\pi_i)$ on $\cS_n$ with a common noise parameter $\phi \in (0, 1)$.  Suppose that $\xi \defn \max_{i \in [k]} |w_i - w'_i| > 0$. 
Let $J$ be a subset of $[n]$ such that $\pi_i\|_J \ne \pi_j\|_J$ for any distinct $i, j \in [k]$. 
Define $\ell \defn |J|$ and define $\eta(k/2, \ell, \phi, 1)$ as in \eqref{eq:def-eta}. 
Then we have 
$$
\TV( \cM|_J , \cM'|_J ) \ge \xi \cdot \eta(k/2, \ell, \phi, 1) . 
$$
\end{proposition}

Based on the above total variation lower bound, Algorithm~\ref{alg:weight} provides a method for estimating the weights in the Mallows mixture. 

\begin{algorithm}[ht]
\normalsize
\caption{\texttt{EstimateWeights}}
\label{alg:weight}
\begin{algorithmic}[1]
\Input $\sigma_1, \dots, \sigma_N \in \cS_n$, $\phi \in (0,1)$, $\gamma \in (0, 1/k]$, and distinct permutations $\hat \pi_1, \dots, \hat \pi_k \in \cS_n$
\Output a vector of weights $\hat w \in [0,1]^k$
\State{$\cI \leftarrow$ the tuple returned by \texttt{FindTuple} (Algorithm~\ref{alg:find-tuple}) with inputs $\hat \pi_1, \dots, \hat \pi_k$} \label{line:returned-tuple}
\State{$J \leftarrow$ the set of all indices that appear in the tuple $\cI$} \label{line:subset-j}
\State{$\cM_N|_J \leftarrow \frac{1}{N} \sum_{m=1}^N \delta_{\sigma_m|_J}$} 
\State{$L \leftarrow \lceil k N^{1/2} \rceil$}
\State{$N' \leftarrow \lceil k N \log N \rceil$}
\State{$\cR(L) \leftarrow 
\big\{ r \in [L]^k : r_i \ge \gamma L , \, \sum_{i=1}^k r_i = L \big\}$} \label{line:def-r}
\For{$r \in \cR(L)$}
\State{$\cM'(r) \leftarrow
\sum_{i=1}^k \frac{r_i}{L} M(\hat \pi_i, \phi)$}
\State{generate $N'$ i.i.d.\ random permutations $\sigma'_1, \dots, \sigma'_{N'}$ from the Mallows mixture $\cM'(r)$}
\State{$\cM'_{N'}(r)|_J \leftarrow \frac{1}{N'} \sum_{m=1}^{N'} \delta_{\sigma'_m|_J}$}
\EndFor
\State{$\hat w \leftarrow \frac{1}{L} \argmin_{r \in \cR(L)} \TV \big( \cM'_{N'}(r)|_J , \cM_N|_J \big)$} \label{line:weight}
\State{\Return $\hat w$}
\end{algorithmic}
\end{algorithm}

%

Similar to Algorithm~\ref{alg:suborder}, it is not hard to see that Algorithm~\ref{alg:weight} runs in polynomial time: First, a call to Algorithm~\ref{alg:find-tuple} takes $O(k^3 n^2)$ time by Lemma~\ref{lem:distinct}. Then we need to search through the set $\cR(L)$ of at most $L^k \approx k^k N^{k/2}$ candidate models, yielding a total time complexity $O\big( k^k N^{k/2} N N' n \log n \big)$ for comparing all the empirical models. 
Therefore, the overall time complexity is $O\big( k^3 n^2 + k^{k+1} N^{k/2+2} n (\log N) (\log n) \big)$ in view of the definition of $N'$. 

The following theorem bounds the entrywise error for $\hat w$ returned by Algorithm~\ref{alg:weight} and concludes this section. 

\begin{theorem} \label{thm:weight}
Suppose that we are given $N$ i.i.d.\ observations sampled from the Mallows mixture $\cM = \sum_{i=1}^k w_i M(\pi_i)$ on $\cS_n$ with distinct central permutations $\pi_1, \dots, \pi_k$ and a known noise parameter $\phi \in (0, 1)$. 
Fix a positive constant $\gamma \le \min_{i \in [k]} w_i$ and a probability of error $\delta \in (0,0.1)$. 
Let $\{\hat \pi_1, \dots, \hat \pi_k\}$ be the set of permutations returned by Algorithm~\ref{alg:main}. 
Furthermore, let $\hat w \in [0,1]^k$ be the vector of weights returned by Algorithm~\ref{alg:weight}. 
If $N \ge \zeta(k, 2k+2, \phi, \gamma) \cdot \log \frac{2 n^{2k+2}}{\delta}$ where $\zeta$ is defined in \eqref{eq:def-zeta}, 
then 
the following holds with probability at least $1 - \delta$: Up to a relabeling, for every $i \in [k]$, it holds that $\hat \pi_i = \pi_i$ and 
$$
|\hat w_i - w_i| \le \frac{(\log N)^{\ell+1}}{N^{1/2}} \big( \zeta(k, 2k-2, \phi, 1) \cdot \log (4/\delta) \big)^{1/2} .
$$ 
The time complexity of the entire algorithm is 
$$
O \bigg( N n^{2(k^2+k+2)} \Big(\frac{e k}{\gamma(1-\phi)}\Big)^{(2 k)^{2k+9}} \log \frac{1}{\delta} + k^{k+1} N^{k/2+2} n (\log N) (\log n) \bigg).
$$
\end{theorem}

\section{Mallows mixture in high-noise regime}
\label{sec:fin-dim}

We turn to study the sample complexity for learning the Mallows mixture in the high-noise regime. 
For simplicity, we focus on the equally-weighted case.
For a Mallows model on $\cS_n$ with noise parameter $\phi \in (0, 1)$, we let $\eps \defn 1 - \phi$ and consider the high-noise regime where $n$ is fixed and
$\eps \to 0$, as which the Mallows model converges to the uniform distribution on $\cS_n$. 
We are interested in how the sample complexity scales with $1/\eps$. 

More formally, let $\scrM_*$ denote the collection of $k$-mixtures of Mallows models on $\cS_n$ with equal weights and a common noise parameter $\phi \in (0, 1)$, that is, 
\begin{align}
\scrM_* \equiv \scrM_*(n, k, \phi) \defn \Big\{\frac 1k \sum_{i=1}^k  M(\pi_i, \phi) : \pi_1, \dots, \pi_k \in \cS_n \Big\} .
\label{eq:class}
\end{align}
Some results in this section can be generalized to mixtures with different weights. However, we focus on the case of equally weighted mixtures to ease the notation, which already includes all the main ideas. 
The following result characterizes the total variation distance between two Mallows mixtures in the high-noise regime up to constant factors. 

\begin{theorem}
\label{thm:high-noise}
For $\mopt = \lfloor \log_2 k \rfloor + 1$ as defined in \eqref{eq:mopt}, the following statements hold as $\eps = 1 - \phi \to 0$: 
\begin{enumerate}[leftmargin=*,label={(\alph*)}]
\item
Suppose that $k \le 255$. 
For any distinct Mallows mixtures $\cM$ and $\cM'$ in $\scrM_*$, we have
$
\TV( \cM, \cM' ) = \Omega( \eps^{\mopt} ) . 
$

\item
On the other hand, for $n \ge 2 \mopt$, there exist distinct Mallows mixtures $\cM$ and $\cM'$ in $\scrM_*$ for which 
$
\TV( \cM, \cM' ) = O( \eps^{\mopt} ) . 
$
\end{enumerate}
The hidden constants in $\Omega(\cdot)$ and $O(\cdot)$ above may depend on $n$ and $k$. 
\end{theorem}

The key to proving the above theorem is to view groups of pairwise comparisons as moments and relate them to the total variation distance between two Mallows mixtures. After establishing this link, the upper and lower bounds follow naturally from the two parts of Theorem~\ref{thm:min-pair} respectively. 

Note that there is a condition $k \le 255$ in part~(a) of the above theorem. This is purely a technical assumption used in one step of the proof. 
We conjecture that the same result holds without this restriction on the number of components $k$ in the mixture. See Section~\ref{sec:conj} for details.

Theorem~\ref{thm:high-noise} characterizes the precise exponent of $\eps$ in the total variation distance between two Mallows mixtures. 
From this, we easily obtain matching upper and lower bounds of order $1/\eps^{2 \mopt}$ on the optimal sample complexity for learning a Mallows $k$-mixture in the high-noise regime. 

\begin{corollary}
\label{cor:high-noise}
Suppose that for a Mallows mixture $\cM \in \scrM_*$, we are given i.i.d.\ observations $\sigma_1, \dots, \sigma_N \sim \cM$, and let $\p_{\cM}$ denote the associated probability measure. We let $\eps \defn 1 - \phi$ and consider the setting where $n$ is fixed and $\eps \to 0$.  
For $\mopt = \lfloor \log_2 k \rfloor + 1$ as defined in \eqref{eq:mopt}, the following statements hold: 
\begin{enumerate}[leftmargin=*,label={(\alph*)}]
\item
Suppose that $k \le 255$, and that $k$ and $\phi$ are known.  
Let $\cM_N$ denote the empirical distribution of $\sigma_1, \dots, \sigma_N$ with PMF $f_{\cM_N}(\sigma) = \frac 1N \sum_{i=1}^N \1\{ \sigma_i = \sigma \}$ for each $\sigma \in \cS_n$. 
Consider the minimum total variation distance estimator 
\begin{align}
\widehat{\cM} \defn \argmin_{\cM' \in  \scrM_*} \TV( \cM', \cM_N ) . 
\label{eq:min-tv}
\end{align}
If $N \ge C \log( \frac{1}{\delta} ) / \eps^{2 \mopt}$ for a sufficiently large constant $C = C(n, k) >0$ and any $\delta \in (0, 1)$, then we have 
$$
\max_{\cM \in \scrM_*} \p_{\cM} \{ \widehat{\cM} \ne \cM \} \le \delta . 
$$

\item
On the other hand, if $n \ge 2 \mopt$ and $N \le c / \eps^{2 \mopt}$ for a sufficiently small constant $c = c(n, k) >0$, then we have 
$$
\min_{\widetilde{\cM}} \max_{\cM \in \scrM_*} \p_{\cM} \{ \widetilde{\cM} \ne \cM \} \ge 1/8 ,
$$
where the estimator $\widetilde{\cM}$ of the mixture is measurable with respect to the observations $\sigma_1, \dots, \sigma_N$. 
\end{enumerate}
\end{corollary}

The computational complexity of the minimum total variation distance estimator $\widehat{\cM}$ is polynomial in the sample size $N$, which itself depends polynomially on $1/\eps$. Therefore, the estimator is polynomial-time in the high-noise regime where $n$ is fixed and $1/\eps$ grows. 
On the other hand, the computational cost depends exponentially on $n$, as it involves an exhaustive search over the class $\scrM_*$ in \prettyref{eq:class}, whose cardinality grows as $(n!)^k$. 
Finding a statistically optimal estimator that is polynomial-time in $n$ is an interesting open question.

Before ending this section, we remark that Liu and Moitra~\cite{LiuMoi18} proved an algorithmic lower bound for learning the Mallows mixture based on a local query model they proposed. In their model, upon receiving a query over a pair of sets $\{j_1, \dots, j_m\}, \{i_1, \dots, i_m\} \subset [n]$, the oracle returns the probability 
$$
\p_{\sigma \sim \cM} \big\{ \sigma(j_1) = i_1, \dots, \sigma(j_k) = i_k \big\} 
$$
up to an additive error $\tau > 0$. 
The cost of each query is defined to be $1/\tau^2$, and the total cost of an algorithm is the sum of its query costs. 
For $k = 2^{m-1}$ so that $m = \log_2 k + 1$, they presented two Mallows mixtures $\cM$ and $\cM'$ with a common noise parameter $\phi = 1 - \sqrt{k/n}$ that cannot be distinguished if $\tau \le (\frac{2k}{n})^{m/2}$. 
As a result, the query complexity for identifying $\cM$ is at least $(\frac{n}{2k})^m = \big( \frac{1}{\sqrt{2} (1-\phi)} \big)^{2m}$. 

This local query complexity is defined in a different way from the sample complexity that we study. However, note that their lower bound $\big( \frac{1}{\sqrt{2} (1-\phi)} \big)^{2m}$ is very similar to our lower bound of order $(\frac{1}{\eps})^{2m} = \big( \frac{1}{1-\phi} \big)^{2m}$ in Corollary~\ref{cor:high-noise}(b); in particular, the exponent $2m = 2(\log_2 k + 1)$ is exactly the same in both bounds. 
This is because, ultimately, both lower bounds are proved by matching the combinatorial moments of two distinct mixtures of permutations. 
Compared to the particular instance considered by Liu and Moitra, we have formalized the combinatorial method of moments more generally and established matching upper and lower bounds on the sample complexity.

\section{Discussion}
\label{sec:discuss}

In this work, we proposed a methodology to learn a mixture of permutations based on groups of pairwise comparisons. 
We first set up the framework using a generic noiseless model for a mixture of permutations. 
Then, we studied the Mallows mixture model, and introduced a polynomial-time algorithm for learning the central permutations with a sample complexity logarithmic in the size of the permutations. 
Finally, we studied the sample complexity for learning the Mallows mixture in a high-noise regime.

For the algorithms in this work, we assumed the knowledge of the noise parameter $\phi$. 
This is indeed restrictive, but we conjecture that our main result on the logarithmic sample complexity in Section~\ref{sec:high-dim} continues to hold without this assumption.
Specifically, the value of $\phi$ is needed in the definition of the class of mixtures \eqref{eq:mix-class}, which is used in Algorithm~\ref{alg:suborder}. 
In the case where $\phi$ is unknown, we can augment the class of models  \eqref{eq:mix-class} by allowing $\phi$ to take values in a fine grid in $(0, 1)$. 
In view of the continuity of the model in $\phi$ and good concentration properties of the model, we believe the same sample complexity can be proved without the knowledge of $\phi$. 
We choose not to introduce this technical complication which does not add much to our general methodology. 

Moreover, in general, a Mallows $k$-mixture model allows its components to have different noise parameters $\phi_1, \dots, \phi_k$. 
While the results in Section~\ref{sec:fin-dim} depend strictly on the assumption of a common noise parameter $\phi$, it is possible to adapt part of our approach in Section~\ref{sec:high-dim} to the heterogeneous setting. 
However, there is a fundamental obstacle which our current proof techniques cannot resolve. 
Namely, the success of Algorithm~\ref{alg:suborder} relies on Proposition~\ref{prop:tv-lower}, which is a dimension-free lower bound on the total variation distance between two marginalized Mallows mixtures whose central permutations do not yield the same set of relative orders on $J$. 
The current proof of this lower bound (see Lemma~\ref{lem:tv-general}, which is a more general version of Proposition~\ref{prop:tv-lower}) leverages a block structure that makes up $J$ and is ultimately based on Lemma~\ref{lem:pre-bd}, an identifiability result for each block in the block structure. 
However, Lemma~\ref{lem:pre-bd} does not generalize to the setting where we have different noise parameters $\phi_i$. For example, in the case where $n = 2$ and $k = 2$, identifiability no longer holds due to the extra degrees of freedom given by $\phi_1$ and $\phi_2$. 
We do not know how to get around this difficulty and defer a potential solution to future work. 

Last but not least, our general approach of learning a mixture of permutations from groups of pairwise comparisons has potential applications beyond the Mallows mixture model. 
It would be interesting to apply the framework proposed in Section~\ref{sec:noiseless} to other models for mixtures of permutations, such as the Plackett-Luce model~\cite{ZhaPieXia16} and variations of the Mallows model~\cite{DeODoSer18}.

\section{Proofs}
\label{sec:proofs}


\subsection{Proof of Theorem~\ref{thm:min-pair}(a)}

Throughout the proof, we write $m \equiv \mopt \defn \lfloor \log_2 k \rfloor + 1$ as in \eqref{eq:mopt}. 
We start with a lemma which is the source of the logarithmic dependency of $m$ on $k$. 

\begin{lemma}
\label{lem:identify}
Consider a set $\Sigma$ of $k$ distinct permutations in $\cS_n$ where $n \ge 2$.  
There exists $\pi^* \in \Sigma$ and a tuple $\cI$ of $\ell$ pairs of distinct indices $(i_1, j_1), \dots, (i_{\ell}, j_{\ell}) \in [n]^2$, such that $\ell \le \lfloor \log_2 k \rfloor$ and $\chi(\pi^*, \cI) \ne \chi(\pi, \cI)$ for all $\pi \in \Sigma \setminus \{\pi^*\}$, where $\chi(\pi, \cI)$ is defined by \eqref{eq:def-chi}. 
In addition, this tuple $\cI$ can be found in polynomial time. 
\end{lemma}

\begin{proof}
Let us start with $\Sigma_0 = \Sigma$ and apply the following bisection argument iteratively. 
Given a nonempty set $\Sigma_{r-1}$ of distinct permutations where $r \ge 1$, it is easy to find a pair of indices $i_r, j_r \in [n]$ such that both of the following sets are nonempty: 
\begin{align}
\Sigma_r^+ = \{ \pi \in \Sigma_{r-1} : \pi (i_r) > \pi (j_r) \} 
\quad \text{ and } \quad 
\Sigma_r^- = \{ \pi \in \Sigma_{r-1} : \pi (i_r) < \pi (j_r) \}  . 
\label{eq:bisect}
\end{align}
Since $\Sigma_r^+ \sqcup \Sigma_r^- = \Sigma_{r-1}$, either $\Sigma_r^+$ or $\Sigma_r^-$ has size at most $|\Sigma_{r-1}|/2$. 
We call it $\Sigma_r$ so that $\Sigma_r \subset \Sigma_{r-1}$ and $|\Sigma_r| \le |\Sigma_{r-1}|/2$. 
This procedure is iterated until we have $|\Sigma_r| = 1$. 

For any $r \ge 1$, we have $|\Sigma_r| \le k/2^r$ by construction. 
In particular, 
$$|\Sigma_{\lfloor \log_2 k \rfloor}| \le k/2^{\lfloor \log_2 k \rfloor} < 2.$$ 
Thus there exists $\ell \le \lfloor \log_2 k \rfloor$ such that $|\Sigma_{\ell}| = 1$. 
We denote the permutation in $\Sigma_{\ell}$ by $\pi^*$. 
Note that by \eqref{eq:bisect} and the definition of $\Sigma_r$, we have $\1\{ \sigma(i_r) < \sigma(j_r) \} \ne \1\{ \pi(i_r) < \pi(j_r) \}$ for any $\sigma \in \Sigma_r$ and $\pi \in \Sigma_{r-1} \setminus \Sigma_r$. 
Since the sets $\Sigma_r$'s are nested, it holds that $\1\{ \pi^*(i_r) < \pi^*(j_r) \} \ne \1\{ \pi(i_r) < \pi(j_r) \}$ for any $\pi \in \Sigma_{r-1} \setminus \Sigma_r$ where $r \in [\ell]$.  
As a result, if we define $\cI \defn \big( (i_1, j_1), \dots, (i_\ell, j_\ell) \big)$, then $\chi(\pi^*, \cI) \ne \chi(\pi, \cI)$ for any $\pi \in \Sigma_0 \setminus \{\pi^*\}$. 
It is clear that $\cI$ can be found in polynomial time, so the proof is complete. 
\end{proof}

We now prove Theorem~\ref{thm:min-pair}(a). Recall that our goal is to recover the $k$-mixture $\sM = \sum_{i=1}^k w_i \delta_{\pi_i}$ of permutations in $\cS_n$ from groups of $m$ pairwise comparisons of the form $\chi(\pi, \cI)$ defined in \eqref{eq:def-chi} where $\pi \sim \sM$. For this, we do an induction on $n \ge 2$, as the case $n=1$ is vacuous.

\paragraph*{Base case}
For $n=2$ and any $k \ge 1$, we can simply take $\cI$ to be the tuple of $m$ copies of $(1, 2)$. 
The oracle of Definition~\ref{def:group-pair} then returns the distribution of $\chi(\pi, \cI)$, from which we immediately read off the distribution of $\1\{\pi(1) < \pi(2)\}$ and thus the distribution of $\pi$. 


\paragraph*{Induction hypothesis}
As the induction hypothesis, we assume that the statement of Theorem~\ref{thm:min-pair}(a) holds for $n-1$ where $n \ge 3$. 
Consider a mixture $\sum_{s=1}^k w_s \delta_{\pi_s}$ of permutations in $\cS_n$ which we aim to learn. 
Then each $\pi_s\|_{[n-1]}$ is a permutation in $\cS_{n-1}$, and by definition \eqref{eq:def-chi}, we have $\chi(\pi_s\|_{[n-1]}, \cI) = \chi(\pi_s, \cI)$ for any tuple $\cI$ of pairs of indices in $[n-1]$. 
Hence the induction hypothesis implies that we can obtain the mixture $\sum_{s=1}^k w_s \delta_{ \pi_s\|_{[n-1]} }$. 
To recover the mixture $\sum_{s=1}^k w_s \delta_{ \pi_s }$ of permutations on $[n]$ from those on $[n-1]$, our task is to insert the index $n$ into each permutation on $[n-1]$ at the correct position. 

\paragraph*{Induction step}
Toward this end, let us apply Lemma~\ref{lem:identify} to the distinct elements of the set $\{ \pi_1\|_{[n-1]}, \dots, \pi_k\|_{[n-1]} \}$ of permutations in $\cS_{n-1}$. 
Thus there exists $s^* \in [k]$ and a tuple $\cI$ of $\ell$ pairs of distinct indices in $[n-1]$,  such that $\ell \le \lfloor \log_2 k \rfloor$ and $\chi(\pi_s, \cI) \ne \chi(\pi_{s^*}, \cI)$ for all $s \in [k] \setminus S^*$ where we define 
$$
S^* \defn \big\{ s \in [k] : \pi_s\|_{[n-1]} = \pi_{s^*}\|_{[n-1]} \big\} . 
$$

Next, for any index $r \in [n-1]$, we choose an $m$-tuple $\cI_r$ consisting of all pairs of indices in $\cI$ and also the pair $(r, n)$. 
Such a tuple $\cI_r$ can be chosen because $\ell \le \lfloor \log_2 k \rfloor = m - 1$. 
Then we query the group of $m$ pairwise comparisons on $\cI_r$ (Definition~\ref{def:group-pair}) to obtain the distribution $\sum_{s=1}^k w_s \delta_{\chi(\pi_s, \cI_r)}$ for each $r \in [n-1]$. 
Recall that the definitions of $\cI$ and $S^*$ guarantee that $\chi(\pi_s, \cI) = \chi(\pi_{s^*}, \cI)$ if and only if $s \in S^*$. 
Since $\cI_r$ includes all pairs of indices in $\cI$, we can distinguish those components of $\sum_{s=1}^k w_s \delta_{\chi(\pi_s, \cI_r)}$ supported at $\chi(\pi_s, \cI_r)$ with $s \in S^*$ from those with $s \in [k] \setminus S^*$. 
Therefore, we obtain the measure $\sum_{s \in S^*} w_s \delta_{\chi(\pi_s, \cI_r)}$ for any $r \in [n-1]$. 

Moreover, since $(r, n) \in \cI_r$, from the measure $\sum_{s \in S^*} w_s \delta_{\chi(\pi_s, \cI_r)}$, we can easily compute the function $f(r) \defn |\{ \sum_{s \in S^*} w_s : \pi_s(r) < \pi_s(n) \}|$ where $r \in [n-1]$. 
In addition, we set $f(0) \defn |S^*|$. 
The measure $\sum_{s \in S^*} w_s \delta_{\pi_s}$ can be recovered from the sequence of numbers $\{ f(r) \}_{r=1}^{n-1}$ as follows. 
By definition, the permutations $\pi_s\|_{[n-1]}$ for $s \in S^*$ are all the same, so by re-indexing $1, \dots, n-1$, we can assume that they are all equal to the identity permutation $(1, 2, \dots, n-1)$ to ease the notation. 
Then $f(r)$ is simply the total weight of permutations $\pi_s$ that place $n$ after $r$, so particularly the sequence $\{f(r)\}_{r=0}^{n-1}$ is nonincreasing. 
Moreover, $f(r) - f(r+1)$ is equal to the total weight of the permutations in the mixture $\sum_{s \in S^*} w_s \delta_{\pi_s}$ satisfying $\pi_s(n) = r+1$. 
Therefore, we can recover the measure $\sum_{s \in S^*} w_s \delta_{\pi_s}$ from the sequence $\{ f(r) \}_{r=1}^{n-1}$.

Finally, once we have learned the measure $\sum_{s \in S^*} w_s \delta_{\pi_s}$, the task becomes recovering the measure $\sum_{s \notin S^*} w_s \delta_{\pi_s}$ from the measure $\sum_{s \notin S^*} w_s \delta_{\pi_s}\|_{[n-1]}$, which can be done by repeating the above procedure. 
Indeed, when querying a group of pairwise comparisons $\sum_{s=1}^k w_s \delta_{\chi(\pi_s, \cI_r)}$, we can easily subtract the components with $s \in S^*$ to obtain $\sum_{s \notin S^*} w_s \delta_{\chi(\pi_s, \cI_r)}$. 
Therefore, the above procedure can be iterated to eventually yield the entire mixture $\sum_{s=1}^k w_s \delta_{\pi_s}$. 
This completes the induction. 


\paragraph*{Time and sample complexity}
To finish the proof, note that every step in this algorithmic construction is clearly polynomial-time. 
For the total number of groups of pairwise comparisons, recall that in the base case $n = 2$, we need one query, and in the induction step from $n-1$ to $n$, we learn at least one component of the mixture from $n-1$ queries. 
In summary, the total number of queries needed is at most $1 + k \sum_{i=2}^{n-1} i = 1 + \frac k2 (n-2) (n+1)$.

\subsection{Proof of Theorem~\ref{thm:min-pair}(b)}

Throughout the proof, we write $m \equiv \mopt \defn \lfloor \log_2 k \rfloor + 1$ as in \eqref{eq:mopt} and fix $\ell \le 2m - 1$. 
Intuitively, it is harder to identify a $k$-mixture of permutations in $\cS_n$ for larger $n$ and larger $k$. 
Indeed, let us justify that we can assume without loss of generality that $n = 2m$ and $k = 2^{m-1}$: 

\begin{itemize}[leftmargin=*]
\item
Suppose that we can prove the statement of part~(b) for $n = 2m$, that is, we have two mixtures $\frac 1k \sum_{i=1}^k \delta_{\pi_i}$ and $\frac 1k \sum_{i=1}^k \delta_{\pi_i'}$ of permutations in $\cS_{2m}$ that cannot be identified using $\ell$-wise comparisons. 
For any $n > 2m$, we may extend each of the above permutation to a permutation in $\cS_n$ by defining $\pi_s(j) = \pi'_s(j) = j$ for all $s \in [k]$ and $2m < j \le n$. 
Then $\ell$-wise comparisons still cannot distinguish the two mixtures, because indices larger than $2m$ are completely uninformative. 
As a result, we may assume $n = 2m$ without loss of generality. 

\item
Suppose that we can establish the desired result for $k$-mixtures. 
Then for any $k' > k$, if we define $\pi_s = \pi'_s = \id$ for all $k < s \le k'$, then the mixtures $\frac{1}{k'} \sum_{i=1}^{k'} \delta_{\pi_i}$ and $\frac{1}{k'} \sum_{i=1}^{k'} \delta_{\pi_i'}$ still cannot be distinguished using groups of $\ell$-wise comparisons. 
Hence the statement of the theorem also holds for $k'$ in replace of $k$. 
For any fixed $m \in \naturals$, the smallest $k$ such that $\lfloor \log_2 k \rfloor + 1 = m$ is equal to $2^{m-1}$. 
Therefore, we may assume that $k = 2^{m-1}$ without loss of generality. 
\end{itemize}

With these simplifications, for a fixed $m \in \naturals$, we now construct two sets $\Sigma_1$ and $\Sigma_2$ of permutations in $\cS_{2m}$ such that $|\Sigma_1| = |\Sigma_2| = 2^{m-1}$, and such that groups of $\ell$-wise comparisons cannot distinguish the two mixtures $\frac 1k \sum_{\pi \in \Sigma_1} \delta_\pi$ and $\frac 1k \sum_{\pi \in \Sigma_2} \delta_\pi$.  
For each vector $v \in \{0, 1\}^m$, we define a permutation $\pi_v \in \cS_{m}$ by 
\begin{align}
\begin{cases}
\pi_v(2j-1) = 2j-1 , \, \pi_v(2j) = 2j & \text{ if } v_j = 0 \\
\pi_v(2j-1) = 2j , \, \pi_v(2j) = 2j-1 & \text{ if } v_j = 1 
\end{cases}
\quad \text{ for all } j \in [m] . 
\label{eq:def-pi}
\end{align}
Moreover, we define 
$$
\Sigma_1 \defn \{ \pi_v : v \in \{0, 1\}^m , \, \|v\|_1 \text{ is odd} \} 
\quad \text{ and } \quad 
\Sigma_2 \defn \{ \pi_v : v \in \{0, 1\}^m , \, \|v\|_1 \text{ is even} \} . 
$$
It is clear that both $\Sigma_1$ and $\Sigma_2$ have cardinality $2^{m-1}$. 

Next, consider an arbitrary set of indices $J \subset [2m]$ with $|J| = 2m-1$.  
We claim that 
\begin{align}
\big\{ \pi_v\|_J : v \in \{0, 1\}^m , \, \|v\|_1 \text{ is odd} \big\} 
= \big\{ \pi_v\|_J : v \in \{0, 1\}^m , \, \|v\|_1 \text{ is even} \big\} . 
\label{eq:set-eq}
\end{align}
To prove this claim, let $j_1$ denote the only element of $[n] \setminus J$. 
If $j_1$ is odd, we let $j_2 = j_1 + 1$; otherwise, we let $j_2 = j_1 - 1$. 
For any $v \in \{0, 1\}^m$ with odd $\|v\|_1$, we define $v' \in \{0, 1\}^m$ by $v'_i = 1 - v_i$ for $i = \lceil j_1/2 \rceil$ and $v'_i = v_i$ for $i \neq \lceil j_1/2 \rceil$. 
Since $v'$ differs from $v$ in only one coordinate,  $\|v'\|_1$ must be even. As a result, we have $\pi_v \in \Sigma_1$ and $\pi_{v'} \in \Sigma_2$. 
This clearly gives a bijection between the sets $\Sigma_1$ and $\Sigma_2$. 
Furthermore, by definition \eqref{eq:def-pi}, $\pi_v$ and $\pi_{v'}$ only differ on the pair $(j_1, j_2)$. 
Since $J$ does not contain $j_1$, we must have $\pi_v\|_J = \pi_{v'}\|_J$. 
Consequently, equation \eqref{eq:set-eq} holds, so that any $\ell$-wise comparison (Definition~\ref{def:listwise}) returns the same distribution for the two mixtures $\frac 1k \sum_{\pi \in \Sigma_1} \delta_\pi$ and $\frac 1k \sum_{\pi \in \Sigma_2} \delta_\pi$. 
This completes the proof.


\subsection{Proof of Theorem~\ref{thm:assemble}}
\label{sec:pf-assemble}

We first establish a lemma which guarantees the success of Algorithm~\ref{alg:find-tuple} which finds a discriminative tuple for any given set of permutations. 

\begin{algorithm}[ht]
\normalsize
\caption{\texttt{FindTuple}}
\label{alg:find-tuple}
\begin{algorithmic}[1]
\Input distinct permutations $\pi_1, \dots, \pi_k$ in $\cS_n$
\Output a tuple $\cI$ of pairs of indices in $[n]$
\State{$\cI \leftarrow [\ ]$}
\Comment{{\it \small $[\ ]$ denotes the empty tuple}}
\For{$j = 2$ to $k$}
\If{there exists $i \in [j-1]$ for which $\chi(\pi_i, \cI) = \chi(\pi_j, \cI)$}

\Comment{{\it \small by convention, $\chi(\pi_1, [\ ]) = \chi(\pi_2, [\ ])$}}
\State{find $r, s \in [n]$ such that $\pi_i(r) < \pi_i(s)$ and $\pi_j(r) > \pi_j(s)$}
\State{$\cI \leftarrow [\cI, (r, s)]$}
\Comment{{\it \small that is, append $(r, s)$ to $\cI$}}
\EndIf
\EndFor
\State{\Return $\cI$}
\end{algorithmic}
\end{algorithm}

\begin{lemma} \label{lem:distinct}
Let $\pi_1, \dots, \pi_k$ be $k$ distinct permutations in $\cS_n$. Algorithm~\ref{alg:find-tuple} finds in $O(k^3 n^2)$ time a tuple $\cI$ of $\ell$ pairs of distinct indices in $[n]$ such that $\ell \le k-1$ and $\chi(\pi_i, \cI) \ne \chi(\pi_j, \cI)$ for any distinct $i, j \in [k]$. 
\end{lemma}

\begin{proof}
First, it is clear that Algorithm~\ref{alg:find-tuple} returns a tuple $\cI$ of $\ell \le k-1$ pairs of indices. 
It suffices to inductively show that, at step $j$ of the loop in the algorithm where $j = 2, \dots, k$, we have $\chi(\pi_i, \cI) \ne \chi(\pi_{i'}, \cI)$ for any distinct $i, i' \in [j]$. 

For $j = 2$, because $\pi_1$ and $\pi_2$ are assumed to be distinct, we can find indices $r, s \in [n]$ such that $\pi_1(r) < \pi_2(s)$ and $\pi_2(r) > \pi_2(s)$. 
Therefore, if we let $\cI$ consist of the single pair $(r, s)$, then $\chi(\pi_1, \cI) \ne \chi(\pi_2, \cI)$. 

Next, at the beginning of step $j$ of the loop where $3 \le j \le k$, we have $\chi(\pi_i, \cI) \ne \chi(\pi_{i'}, \cI)$ for any distinct $i, i' \in [j-1]$ by the induction hypothesis. 
If $\chi(\pi_i, \cI) \ne \chi(\pi_j, \cI)$ for all $i \in [j-1]$, then we are done with the induction. 
Otherwise, there exists exactly one $i \in [j-1]$ such that $\chi(\pi_i, \cI) = \chi(\pi_j, \cI)$. 
Since $\pi_i$ and $\pi_j$ are assumed to be distinct, we can find indices $r, s \in [n]$ such that $\pi_i(r) < \pi_i(s)$ and $\pi_j(r) > \pi_j(s)$. 
As a result, once we append the pair $(r, s)$ to the tuple $\cI$, it then holds that $\chi(\pi_i, \cI) \ne \chi(\pi_j, \cI)$. 
We now have $\chi(\pi_i, \cI) \ne \chi(\pi_{i'}, \cI)$ for any distinct $i, i' \in [j]$, finishing the induction. 

Finally, the time complexity of the algorithm is $O(k^3 n^2)$ because it searches through $j = 2, \dots, k$ and $i = 1, \dots, j-1$; for a pair $(i, j)$, binary vectors of length at most $k-1$ are compared; and finding a pair $(r, s)$ takes time $O(n^2)$. 
\end{proof}

We now prove Theorem~\ref{thm:assemble} for $n \ge 2$ as the case $n=1$ is trivial. 
This proof is structurally similar to Theorem~\ref{thm:min-pair}(a), but the key step in the induction is different. 
We note an intricacy throughout this proof: 
When queried with the tuple $\cI$, the oracle in Definition~\ref{def:weak-group} returns the \emph{set} $\{ \chi(\pi_i, \cI) : i \in [k] \}$ where $\chi(\pi_i, \cI)$ is defined by \eqref{eq:def-chi}.  
This set (as opposed to an ordered tuple or multiset) is represented by its distinct elements without labels, so it is possible that it contains less than $k$ distinct elements and we do not know their multiplicities.

\begin{algorithm}[H]
\normalsize
\caption{\ID}
\label{alg:insert-demix}
\begin{algorithmic}[1]
\Input $n$, $k$, and $\{\chi(\pi_i, \cI) : i \in [k]\}$ for tuples $\cI$ of $k+1$ pairs of distinct indices in $[n]$ queried in the algorithm
\Output a set $\fS_n$ of permutations in $\cS_n$
\If{$n=2$} \label{line:base-start}
\State{$\cI \leftarrow$ the tuple of $k+1$ copies of $(1, 2)$}

\Comment{{\it \small all entries of $\chi(\pi_i, \cI)$ are equal to $\1\{\pi_i(1) < \pi_i(2)\}$}}
\If{$\{ \chi(\pi_i, \cI)_1 : i \in [k] \} = \{1\}$}
\State{$\fS_2 \leftarrow \{(1, 2)\}$}
\ElsIf{$\{ \chi(\pi_i, \cI)_1 : i \in [k] \} = \{0\}$}
\State{$\fS_2 \leftarrow \{(2, 1)\}$}
\Else
\Comment{{\it \small $\{ \chi(\pi_i, \cI)_1 : i \in [k] \} = \{0, 1\}$}}
\State{$\fS_2 \leftarrow \{(1, 2), (2, 1)\}$}
\EndIf \label{line:base-end}
\Else \label{line:ind-1}
\Comment{{\it \small $n \ge 3$}}
\State{$\fS_{n-1} \leftarrow$ the subset of $\cS_{n-1}$ returned by \ID\ run with inputs $n-1$, $k$, and $\{\chi(\pi_i, \cI) : i \in [k]\}$ for tuples $\cI$ of $k+1$ pairs of indices in $[n-1]$}
\State{let $\sigma_1, \dots, \sigma_{k'} \in \cS_{n-1}$ denote the elements of $\fS_{n-1}$} \label{line:ind-2}
\If{$k'>|\fS_{n-2}|$}
\State{$\cI \leftarrow$ the tuple returned by \texttt{FindTuple} (Algorithm~\ref{alg:find-tuple}) with inputs $\sigma_1, \dots, \sigma_{k'}$} \label{line:ind-3} \label{line:find-a-tuple}
%
\EndIf
\State{$\fS_n \leftarrow \varnothing$}
\For{$j = 1$ to $k'$} \label{lind:loop-j} \label{line:loop-1}
\For{$r = 2$ to $n-1$}
\State{$\cI_r \leftarrow \big[\cI, \dots, \big( \sigma_j^{-1}(r-1), n \big), \big( n, \sigma_j^{-1}(r) \big) \big]$} \label{line:ir}
\Comment{{\it \small here $\dots$ means some arbitrary pairs of indices we append to $\cI$ for concreteness so that $\cI_r$ consists of exactly $k+1$ pairs}}
\State{call the oracle in Definition~\ref{def:weak-group} to obtain the set $\{\chi(\pi_i, \cI_r) : i \in [k]\}$}
\State{$\fX_j(r) \leftarrow \{\chi(\pi_i, \cI_r) : i \in [k] , \, \chi(\pi_i, \cI) = \chi(\sigma_j, \cI) \}$}
\If{there is $v \in \fX_j(r)$ such that $v_k = v_{k+1} = 1$}
\State{$\fS_n \leftarrow \fS_n \cup \big\{\big(\sigma_j^{-1}(1), \dots, \sigma_j^{-1}(r-1), n, \sigma_j^{-1}(r), \dots, \sigma_j^{-1}(n-1)\big)\big\}$} \label{line:recover}
\EndIf
\EndFor
\If{there is $v \in \fX_j(2)$ such that $v_k = 0$} \label{line:sc-1}
\State{$\fS_n \leftarrow \fS_n \cup \big\{\big(n, \sigma_j^{-1}(1), \dots, \sigma_j^{-1}(n-1)\big)\big\}$}
\EndIf
\If{there is $v \in \fX_j(n-1)$ such that $v_{k+1} = 0$}
\State{$\fS_n \leftarrow \fS_n \cup \big\{\big(\sigma_j^{-1}(1), \dots, \sigma_j^{-1}(n-1), n\big)\big\}$}
\EndIf \label{line:sc-2}
\EndFor \label{line:ind-4}
\EndIf
\State{\Return $\fS_n$}
\end{algorithmic}
\end{algorithm}

Note that Algorithm~\ref{alg:insert-demix} is recursive with respect to $n$; correspondingly, we prove Theorem~\ref{thm:assemble} by induction on $n \ge 2$. 

\paragraph*{Base case (lines \ref{line:base-start}--\ref{line:base-end} of Algorithm~\ref{alg:insert-demix})}
For $n=2$, as in Algorithm~\ref{alg:insert-demix}, we simply take $\cI$ to be the tuple of $k+1$ copies of $(1, 2)$. Then every entry of $\chi(\pi_i, \cI)$ is equal to $\1\{\pi_i(1) < \pi_i(2)\}$. 
The oracle returns the set $\{ \chi(\pi_i, \cI) : i \in [k] \}$, from which we immediately read off the set $\{ \1\{\pi_i(1) < \pi_i(2)\} : i \in [k] \}$ and thus the set $\{ \pi_i : i \in [k] \}$ as detailed in the algorithm. 

\paragraph*{Induction hypothesis (lines \ref{line:ind-1}--\ref{line:ind-2} of Algorithm~\ref{alg:insert-demix})}
Fix $n \ge 3$ and assume that the conclusion of the theorem holds for $n-1$. 
Consider a mixture of permutations $\pi_1, \dots, \pi_k \in \cS_n$ which we aim to learn. 
Then each $\pi_i\|_{[n-1]}$ where $i \in [k]$ is a permutation in $\cS_{n-1}$. 
By definition \eqref{eq:def-chi}, we have $\chi(\pi_i\|_{[n-1]}, \cI) = \chi(\pi_i, \cI)$ for any tuple $\cI$ of pairs of indices in $[n-1]$. 
Hence the induction hypothesis implies that the algorithm returns the set of permutations $\fS_{n-1} = \{ \pi_i\|_{[n-1]} : i \in [k] \}$. 

Let us denote the distinct elements of $\fS_{n-1}$ by $\sigma_1, \dots, \sigma_{k'} \in \cS_{n-1}$, where $k' \le k$. 
To obtain the set $\{ \pi_i : i \in [k] \}$ of permutations on $[n]$ from those on $[n-1]$, our task is to insert the index $n$ into $\sigma_j$ at the correct position for each $j \in [k']$. 
Note that $|\{ \pi_i\|_{[n-1]} : i \in [k] \}| \le |\{ \pi_i : i \in [k] \}|$ and we may need to obtain more than one permutation in $\cS_n$ from each $\sigma_j$ where $j \in [k']$. 

\paragraph*{Induction step (lines \ref{line:ind-3}--\ref{line:ind-4} of Algorithm~\ref{alg:insert-demix})}
Run Algorithm~\ref{alg:find-tuple} with inputs $\sigma_1, \dots, \sigma_{k'}$. Lemma~\ref{lem:distinct} guarantees that we obtain an $\ell$-tuple $\cI$ of pairs of distinct indices in $[n-1]$ such that $\ell \le k'-1$ and $\chi(\sigma_i, \cI) \ne \chi(\sigma_j, \cI)$ for any distinct $i, j \in [k']$. 
This guarantees that $\pi_i\|_{[n-1]} = \sigma_j$ if and only if $\chi(\pi_i, \cI) = \chi(\sigma_j, \cI)$ for $i \in [k]$ and $j \in [k']$. 

We now fix $j \in [k']$ (line~\ref{lind:loop-j}) and aim to recover those $\pi_i$ such that $\pi_i\|_{[n-1]} = \sigma_j$. 
To simplify the notation in the sequel, we assume 
that $\sigma_j$ is equal to the identity permutation on $[n-1]$, denoted by $(1, 2, \dots, n-1)$. 
Algorithm~\ref{alg:insert-demix} is stated in full generality, and the proof of its validity is essentially the same. 
With the assumption $\pi_i\|_{[n-1]} = (1, 2, \dots, n-1)$, to recover $\pi_i$, it suffices to determine where $n$ should be inserted into $(1, 2, \dots, n-1)$. 

Note that $k+1 \ge k' + 1 \ge \ell + 2$. 
Hence, for each $r = 2, \dots, n-1$, we can define an $(k+1)$-tuple $\cI_r$ (line~\ref{line:ir}) containing all the $\ell$ pairs of indices in $\cI$ and the pairs $(r-1, n)$ and $(n, r)$. 
In the case that $k+1 > \ell + 2$, the remaining $k-\ell-1$ pairs of indices in $\cI_r$ can be defined arbitrarily for concreteness---we will not use the comparison information on those pairs. 
Then, we query the group of pairwise comparisons on $\cI_r$ according to Definition~\ref{def:weak-group} to obtain the set $\{ \chi(\pi_i, \cI_r) : i \in [k] \}$.

Since $\cI_r$ includes all pairs of indices in $\cI$, we can compute 
$$
\fX_j(r) \defn \{\chi(\pi_i, \cI_r) : i \in [k] , \, \chi(\pi_i, \cI) = \chi(\sigma_j, \cI) \} \subset \{ \chi(\pi_i, \cI_r) : i \in [k] \} . 
$$
Note that $\pi_i\|_{[n-1]} = \sigma_j$ if and only if $\chi(\pi_i, \cI) = \chi(\sigma_j, \cI)$. By the definition of $\fX_j(r)$, for each fixed $r = 2, \dots, n-1$, we have that $\pi_i\|_{[n-1]} = \sigma_j$ if and only if $\chi(\pi_i, \cI_r) \in \fX_j(r)$ for $i \in [k]$. 
In particular, when we take $v \in \fX_j(r)$ in the algorithm, $v = \chi(\pi_i, \cI_r)$ for some $\pi_i$ such that $\pi_i\|_{[n-1]} = \sigma_j$. 

It remains to recover $\pi_i$ for which $\pi_i\|_{[n-1]} = \sigma_j$ from the collection of sets $\{ \fX_j(r) \}_{r=2}^{n-1}$. 
First, fix $r$ and recall that the pairs $(r-1, n)$ and $(n, r)$ are both in $\cI_r$. 
If 
$$
\pi_i = (1, \dots, r-1, n, r, \dots, n-1)
$$ 
for some $i \in [k]$, then $\pi_i\|_{[n-1]} = \sigma_j$ and the set $\fX_j(r)$ must contain a vector $v = \chi(\pi_i, \cI_r)$ whose entries $v_k = \1\{ \pi_i(r-1) < \pi_i(n) \}$ and $v_{k+1} = \1\{ \pi_i(n) < \pi_i(r) \}$ are both equal to $1$. 
Conversely, if $\fX_j(r)$ contains some vector $v$ with $v_k = v_{k+1} = 1$, then we know that $\pi_i\|_{[n-1]} = \sigma_j$ and $\pi_i$ must be equal to $(1, \dots, r-1, n, r, \dots, n-1).$ 
As a result, we successfully recover this $\pi_i$ (line~\ref{line:recover}). 

This argument clearly works for $\pi_i$ equal to $(n, 1, \dots, n-1)$ or $(1, \dots, n-1, n)$ as well, in the respective cases (lines~\ref{line:sc-1}--\ref{line:sc-2}): 
\begin{itemize}
\item 
$r = 2$ and $v_k = \chi(\pi_i, \cI_2)_k = \1\{\pi_i(1) < \pi_i(n)\} = 0$; 
\item 
$r = n-1$ and $v_{k+1} = \chi(\pi_i, \cI_{n-1})_{k+1} = \1\{\pi_i(n) < \pi_i(n-1)\} = 0$. 
\end{itemize}
Therefore, we are able to recover all distinct $\pi_i$ such that $\pi_i\|_{[n-1]} = \sigma_j$ for $i \in [k]$. 

Finally, repeating the above procedure for each $j \in [k']$ yields the set $\{\pi_i : i \in [k]\}$. 

\paragraph*{Time and sample complexity}
In each recursion of Algorithm~\ref{alg:insert-demix}, the bottleneck of time complexity is a call to Algorithm~\ref{alg:find-tuple} (line~\ref{line:find-a-tuple}) which takes $O(k^3 n^2)$ time, but this step only needs to be run at most $k$ times. Moreover, the step of inserting $n$ to the current mixture takes less than $O(k^4 n)$ time (lines~\ref{line:loop-1}--\ref{line:ind-4}). As a result, the overall time complexity is $O(k^4 n^2)$. 

For the total number of groups of pairwise comparisons, recall that in the base case $n = 2$, we need one query, and in the induction step from $n-1$ to $n$, we learn at least one component of the mixture from $n-2$ queries. 
In summary, the total number of queries needed is at most $1 + k \sum_{i=3}^{n} (i-2) = 1 + \frac k2 (n-2) (n-1)$.




\subsection{Basic facts about the Mallows model}

We state some basic facts about the Mallows model that are known in the literature. 

%

\begin{lemma}
\label{lem:dev}
Consider a Mallows model $M(\pi, \phi)$. 
Then for any fixed integers $j \in [n]$ and $r \ge 1$, it holds that 
$$
\p_{\sigma \sim M(\pi, \phi)} \{ |\sigma(j) - \pi(j)| \ge r \} \le \frac{2 \phi^r}{ 1 - \phi } . 
$$
\end{lemma}

\begin{proof}
See, for example, Lemma~17 of~\cite{BraMos09}. 
\end{proof}

The following lemma is essentially Lemma~3 of~\cite{LiuMoi18} in a different form, which gives a preliminary identifiability result for the Mallows mixture. 
Although this result, which follows from Zagier's work~\cite{Zag92}, appears to be extremely weak, it can be used as a building block to establish much stronger bounds later. 

\begin{lemma}
\label{lem:pre-bd}
Consider Mallows models $M(\pi_1), \dots, M(\pi_k)$ on $\cS_n$ with distinct central permutations $\pi_1, \dots, \pi_k$ and a common noise parameter $\phi \in (0, 1)$.  
There exists a test function $u : \cS_n \to [-1,1]$ such that 
\begin{align*}
\E_{M(\pi_1)}[ u ] \ge \frac{1}{n!} \Big[ \frac{ (1-\phi)^n }{ \sqrt{ n! } } \Big]^k ,
\end{align*}
and $\E_{M(\pi_i)}[ u ]  = 0$ for $i = 2, \dots, k$, where we write $\E_{M(\pi_i)}[ u ] \equiv \E_{\sigma \sim M(\pi_i)} [ u(\sigma) ]$. 
\end{lemma}

\begin{proof}
Using the main result from~\cite{Zag92}, Lemma~4 of~\cite{LiuMoi18} establishes the following result: Let $A=A(\phi)$ be the $n! \times n!$ matrix defined in \eqref{eq:Aphi}, where $A_{\sigma, \pi} = \phi^{\dkt(\sigma, \pi)}$ for $\sigma, \pi \in \cS_n$. 
By \eqref{eq:zagier-A}, $A$ is non-singular.
Let $A_\pi$ denote the column of $A$ indexed by $\pi$. 
Then the orthogonal projection of $A_{\pi_1}$ onto the orthogonal complement of $A_{\pi_2}, \dots, A_{\pi_k}$ has Euclidean norm at least $\big[ \frac{ (1-\phi)^n }{ \sqrt{ n! } } \big]^k$. 

Normalizing this orthogonal projection of $A_{\pi_1}$ yields a unit  vector $u$, which can be identified with a function $u: \cS_n\to[-1,1]$.
Then the above result shows that $\E_{M(\pi_1)} [u] = \frac{1}{Z(\phi)}\Iprod{A_{\pi_1}}{u}  \ge \frac{1}{Z(\phi)} \big[ \frac{ (1-\phi)^n }{ \sqrt{ n! } } \big]^k$, and $\E_{M(\pi_i)} [u] = \frac{1}{Z(\phi)}\Iprod{A_{\pi_i}}{u} = 0$ for $i = 2, \dots, k$. 
Finally, applying the crude bound $Z(\phi) = \sum_{\sigma\in \cS_n} \phi^{\dkt(\sigma, \pi)} \le n!$ finishes the proof.
\end{proof}

\subsection{Block structure}
\label{sec:block}

Liu and Moitra~\cite{LiuMoi18} introduced the notion of block structure which was key to analyzing (mixtures of) Mallows models. 
In this work, we define a block structure in a slightly different way. 
We say that a set $B$ of integers is \emph{contiguous} if it is of the form $\{i, i+1, \dots, i+|B|-1\}$ for some integer $i$. 
For a permutation $\pi \in \cS_n$ and a subset $B \subset [n]$, we let $\pi(B)$ denote the set $\{\pi(i) : i \in B\}$. 

\begin{definition}
\label{def:block}
Consider pairwise disjoint sets $B_1, \dots, B_m \subset [n]$ and pairwise disjoint contiguous sets $B'_1, \dots, B'_m \subset [n]$, such that $|B_j| = |B'_j| > 0$ for each $j \in [m]$ and $\max B'_j < \min B'_{j+1}$ for each $j \in [m-1]$. 
We refer to the sequence of pairs $\cB = (B_1, B'_1), \dots, (B_m, B'_m)$ as a \emph{block structure}. 
Moreover, we say that a permutation $\pi \in \cS_n$ \emph{satisfies the block structure} $\cB$ if $\pi(B_j) = B'_j$ for each $j \in [m]$. 
\end{definition}

\noindent
For example, the permutation $(3, 2, 8, 4, 6, 1, 7, 5)$ satisfies the block structure $$\cB = (\{2, 8\}, \{2, 3\}), (\{1, 5, 7\}, \{6, 7, 8\}).$$ 

\subsubsection{Conditioning on satisfying a block structure} 

Let $M(\pi, \phi)$ be a Mallows model, and let $\cB$ be a block structure. 
Later in the proofs, we use the technique developed in~\cite{LiuMoi18} of conditioning on $\sigma \sim M(\pi, \phi)$ satisfying $\cB$. 
This technique of conditioning is helpful thanks to Lemma~\ref{lem:block-cond} below, which in particular restates Fact~1 and Corollary~2 of~\cite{LiuMoi18}. 

Recall from Section~\ref{sec:notation} for a subset $B \subset [n]$, 
their ``relative ordering'' under $\pi$ is denoted by $\pi\|_B$, which is the bijection from $B$ to $[|B|]$ induced by $\pi|_B$. In addition, $\pi\|_B$ can also be viewed as permutation in $\cS_{|B|}$ by identifying the elements of $B$ with $1, \dots, |B|$ in the ascending order. Therefore, it is valid to consider the Mallows model $M(\pi\|_B, \phi)$ on $\cS_{|B|} \equiv \{\text{bijections from } B \text{ to } [|B|] \}$. 
For instance, in the example after Definition~\ref{def:block} above, $\pi\|_{\{1,5,7\}}$ can be identified with the permutation $(1,3,2)$ in $\calS_3$.

\begin{lemma} \label{lem:block-cond}
Consider a Mallows model $M(\pi, \phi)$ where $\pi \in \cS_n$ and $\phi \in (0,1)$, and a block structure $\cB = (B_1, B'_1), \dots, (B_m, B'_m)$. Let
$J \defn \bigcup_{j=1}^m B_j$ and $J' \defn \bigcup_{j=1}^m B_j'$. 
Fix a bijection $\tau : [n] \setminus J \to [n] \setminus J'$. 
%
%
For $\sigma \sim M(\pi, \phi)$, conditional on the event 
$$\{ \sigma \text{ satisfies } \cB \text{ and } \sigma|_{[n] \setminus J} = \tau \},$$ 
the relative orderings $\pi\|_{B_j}$ for $j \in [m]$ are independent, and each $\pi\|_{B_j}$ (when identified as an element of $\calS_{|B_j|}$) is distributed as the Mallows model $M(\pi\|_{B_j}, \phi)$. 



Consequently, given any functions $u_j : \cS_{|B_j|} \to \R$ for $j \in [m]$, we have 
\begin{align}
\E_{\sigma \sim M(\pi)} \bigg[ \1 \{ \sigma \text{ satisfies } \cB \} \cdot \prod_{j = 1}^m u_j (\sigma\|_{B_j}) \bigg] =  \p_{\sigma \sim M(\pi) } \{ \sigma \text{ satisfies } \cB \} \cdot \prod_{j=1}^m \E_{M(\pi\|_{B_j})} [ u_j ] , 
\label{eq:prod}
\end{align}
where we write $M(\pi) \equiv M(\pi, \phi)$ and $ \E_{M(\pi\|_{B_j})} [ u_j ] \equiv \E_{\sigma_j \sim M(\pi\|_{B_j})} [ u_j(\sigma_j) ] $. 
\end{lemma}

\begin{proof}
Consider permutations $\sigma, \sigma' \in \cS_n$ such that $\sigma$ and $\sigma'$ both satisfy the block structure $\cB$, and $\sigma|_{[n] \setminus J} = \sigma'|_{[n] \setminus J} = \tau$. 
Let $\sigma_j = \sigma\|_{B_j}$ and  $\sigma_j' = \sigma'\|_{B_j}$ for each $j \in [m]$. 
Since each $B_j'$ is contiguous, it is possible to have $\1 \{ \sigma (s) > \sigma (t) \} \ne \1 \{ \sigma' (s) > \sigma' (t) \}$ only if the indices $s$ and $t$ are in the same block $B_j$. It follows that 
\begin{align*}
&\dkt(\pi, \sigma) - \dkt(\pi, \sigma') \\
&= \sum_{s, t \in [n] : \pi(s) < \pi(t)} \Big( \1 \{ \sigma (s) > \sigma (t) \} - \1 \{ \sigma' (s) > \sigma' (t) \} \Big) \\
&= \sum_{j=1}^m \sum_{s, t \in B_j : \pi(s) < \pi(t)} \Big( \1 \{ \sigma (s) > \sigma (t) \} - \1 \{ \sigma' (s) > \sigma' (t) \} \Big) \\
&= \sum_{j=1}^m \sum_{s, t \in B_j : \pi\|_{B_j}(s) < \pi\|_{B_j}(t)} \Big( \1 \{ \sigma_j (s) > \sigma_j (t) \} - \1 \{ \sigma'_j (s) > \sigma'_j (t) \} \Big) \\
&= \sum_{j=1}^m \big[ \dkt(\pi\|_{B_j}, \sigma_j) - \dkt(\pi\|_{B_j}, \sigma_j') \big] . 
\end{align*}

As a result, the ratio between the probability masses at $\sigma$ and $\sigma'$ under the original Mallows model $M(\pi, \phi)$ is equal to the ratio between the probability masses at $(\sigma_1, \dots, \sigma_m)$ and $(\sigma'_1, \dots, \sigma'_m)$ under the product distribution $\otimes_{j=1}^m M(\pi\|_{B_j}, \phi)$: 
$$
\frac{ \phi^{\dkt(\pi, \sigma)} }{ \phi^{\dkt(\pi, \sigma')} } 
=  \phi^{\dkt(\pi, \sigma) - \dkt(\pi, \sigma') } 
= \phi^{ \sum_{j=1}^m [ \dkt(\pi\|_{B_j}, \sigma_j) - \dkt(\pi\|_{B_j}, \sigma_j') ] }
= \prod_{j=1}^m \frac{ \phi^{ \dkt(\pi\|_{B_j}, \sigma_j) } }{ \phi^{ \dkt(\pi\|_{B_j}, \sigma_j')  } } . 
$$
Therefore, the first statement of the lemma holds. 

Furthermore, since the product distribution $\otimes_{j=1}^m M(\pi\|_{B_j}, \phi)$ does not depend on $\tau = \sigma|_{[n] \setminus J}$, we see that the conditional distribution of $\sigma \sim M(\pi, \phi)$ on $\sigma$ satisfying $\cB$, marginalized over $\sigma|_{[n] \setminus J}$, is also the product distribution $\otimes_{j=1}^m M(\pi\|_{B_j}, \phi)$. 
Therefore, both sides of \eqref{eq:prod} are equal to 
\begin{align*}
\p_{\sigma \sim M(\pi) } \{ \sigma \text{ satisfies } \cB \} \cdot \E_{\sigma \sim M(\pi)} \bigg[  \prod_{j = 1}^m u_j (\sigma\|_{B_j}) \, \Big| \, \sigma \text{ satisfies } \cB \bigg] ,
\end{align*}
so the proof is complete. 
\end{proof}

\subsubsection{Probability of satisfying a block structure}

We now establish a crucial lower bound on the probability that a permutation from the Mallows model satisfies a certain block structure. 
Let $\dhaus$ denote the Hausdorff distance between two sets $A, B \subset \Z$, that is, 
$$
\dhaus (A, B) \defn \max \Big\{ \max_{a \in A} \min_{b \in B} |a - b|, \,  \max_{b \in B} \min_{a \in A} |a - b| \Big\} . 
$$
The following lemma provides a dimension-free lower bound on the probability of satisfying a block structure, whenever the central permutation satisfies 
the same block structure approximately (up to distance $D$). 
This allows us to ``localize" the analysis to a block structure without sacrificing the dependency of the sample complexity on $n$. 
This result significantly improves Lemma~1 of~\cite{LiuMoi18}, which gives a lower bound of $n^{-2\ell}$ assuming that the central permutation satisfies the block structure exactly ($D=0$).


\begin{lemma}
\label{lem:bd-block}
Let $M(\pi, \phi)$ be a Mallows model on $\cS_n$. 
For a block structure $\cB = (B_1, B'_1),$ $\dots,$ $(B_m, B'_m)$, suppose that $\dhaus( \pi(B_i), B'_i ) \le D$ for each $i \in [m]$, where $\dhaus$ denotes the Hausdorff distance and $D \ge 0$. 
Let $\ell \defn \sum_{i=1}^m |B_i|$. 
Then we have 
$$
\p_{\sigma \sim M(\pi, \phi)} \{\sigma \text{ satisfies } \cB\} \ge 
\frac{ \phi^{\ell D } (1 - \phi)^{3 \ell} }{ 2 (6 \ell)^{2 \ell} } .
$$
\end{lemma}

\begin{proof}
Set $
R \defn \big\lceil \frac{\log [ (1 - \phi) / ( 4 \ell) ]}{\log \phi} \big\rceil .
$
By Lemma~\ref{lem:dev} and a union bound, 
we have 
\begin{align} 
\p_{\sigma \sim M(\pi, \phi)} \big\{ |\sigma(j) - \pi(j)| \le R 
\text{ for all } j \in {\textstyle \bigcup_{i=1}^m } B_i \big\} \ge 1/2 . 
\label{eq:uni-bd}
\end{align}

Let us define a collection of $m$-tuples of sets  
\begin{align*}
\cK &\defn \big\{ (K_1, \dots, K_m) : K_i \subset [n], \, |K_i| = |B_i|, \, \dhaus (\pi(B_i), K_i) \le R  , \\
&\qquad \qquad \qquad \qquad \qquad K_i \cap K_j = \varnothing \text{ for any distinct } i, j \in [m] \big\} . 
\end{align*}
Note that there are at most $\binom{|B_i| + 2 R}{ |B_i| }$ choices for $K_i$ with $|K_i| = |B_i|$ and $\dhaus (\pi(B_i), K_i) \le R$ for each $i \in [m]$, so the cardinality of $\cK$ can be bounded as 
\begin{align}
|\cK| \le \prod_{i=1}^m \binom{|B_i| + 2 R}{ |B_i| } 
\le \prod_{i=1}^m \big( |B_i| + 2 R \big)^{|B_i|} 
\le \big( \ell + 2 R \big)^{\ell} . 
\label{eq:card}
\end{align}

In addition, for each tuple $(K_1, \dots, K_m) \in \cK$, we define an event 
$$ 
\fS (K_1, \dots, K_m) \defn \big\{ \sigma ( B_i ) = K_i \text{ for all } i \in [m] \big\} . 
$$
Then by definition, we have 
$$
\big\{ |\sigma(j) - \pi(j)| \le R \text{ for all } j \in {\textstyle \bigcup_{i=1}^m} B_i \big\} 
\subset \bigcup_{(K_1, \dots, K_m) \in \cK} \fS(K_1, \dots, K_m) . 
$$
Writing $\p_{M(\pi, \phi)} \equiv \p_{\sigma \sim M(\pi, \phi)}$ for brevity, we obtain from the above inclusion and \eqref{eq:uni-bd} that 
\begin{align}
\sum_{(K_1, \dots, K_m) \in \cK} \p_{M(\pi, \phi)} \big\{ \fS(K_1, \dots, K_m) \big\} \ge 1/2 . 
\label{eq:sum-bd}
\end{align}

Next, 
note that
$$ 
\fS (B'_1, \dots, B'_m) = \big\{ \sigma ( B_i ) = B'_i \text{ for all } i \in [m] \big\} 
= \big\{ \sigma \text{ satisfies } \cB \big\} . 
$$
We claim that for any $(K_1, \dots, K_m) \in \cK$, 
\begin{align}
\p_{M(\pi, \phi)} \big\{ \fS(B'_1, \dots, B'_m) \big\} \ge \phi^{s ( D + R)} \cdot  \p_{M(\pi, \phi)} \big\{ \fS(K_1, \dots, K_m) \big\}  . 
\label{eq:claim-pr}
\end{align}
Assuming this claim, we conclude from \eqref{eq:claim-pr}, \eqref{eq:sum-bd} and  \eqref{eq:card}  that 
\begin{align*}
\p_{M(\pi, \phi)} \big\{ \fS(B'_1, \dots, B'_m) \big\} 
&\ge \frac{ \phi^{\ell ( D + R)} }{ |\cK| } \sum_{(K_1, \dots, K_m) \in \cK} \p_{ M(\pi, \phi)} \big\{ \fS(K_1, \dots, K_m) \big\} \\
&\ge \frac{ \phi^{\ell ( D + R)} }{ 2 |\cK| } 
\ge \frac{ \phi^{\ell ( D + R)} }{ 2 ( \ell + 2 R )^\ell }  
\ge \frac{ \phi^{\ell D } (1 - \phi)^{3 \ell} }{ 2 (6 \ell)^{2 \ell} }  , 
\end{align*}
where the last step follows elementary algebra using the fact 
$
R = \frac{\log [ (1 - \phi) / ( 4 \ell) ]}{\log \phi}
\le \frac{ 4 \ell }{ (1-\phi)^2 } . 
$

It remains to prove \eqref{eq:claim-pr}. 
There is a natural bijection between the events $\fS(K_1, \dots, K_m)$ and $\fS(B'_1, \dots, B'_m)$ (viewed as subsets of $\cS_n$) as follows:
For each $\sigma \in \fS(K_1, \dots, K_m)$, there is a corresponding $\sigma' \in \fS(B'_1, \dots, B'_m)$ defined so that 
$$
\sigma'\|_{B_i} = \sigma\|_{B_i}
\text{ for all } i \in [m]
\quad \text{ and } \quad \sigma'\|_{[n] \setminus (\bigcup_{i=1}^m B_i) } = \sigma\|_{[n] \setminus (\bigcup_{i=1}^m B_i) } . 
$$ 
That is, $\sigma$ maps each $B_i$ to $K_i$ and $\sigma'$ maps each $B_i$ to $B'_i$, and their relative orders agree on each $B_i$ as well as on the complement of $\bigcup_{i=1}^m B_i$. 

Recall that $\dhaus( \pi(B_i), K_i) \le R$ and $\dhaus( \pi(B_i), B'_i ) \le D$ for each $i \in [m]$, so we have 
$
\dhaus(K_i, B'_i) \le D + R . 
$
Since each $B'_i$ is contiguous, we easily see that 
$$
|\sigma(k) - \sigma'(k)| \le D + R
\quad \text{ for each }
k \in \textstyle{ \bigcup_{i=1}^m } B_i .
$$ 
As a result, it takes at most $\ell (D + R)$ adjacent transpositions to change the permutation $\sigma$ to $\sigma'$, so
$\dkt(\sigma', \sigma) \le  \ell (D + R)$.
By the triangle inequality, we have 
\begin{align*}
\dkt(\sigma', \pi) - \dkt(\sigma, \pi)
\le \dkt(\sigma', \sigma) 
\le  \ell (D + R) . 
\end{align*}
Denoting the PMF of $M(\pi, \phi)$ by $f_{M(\pi, \phi)}$, we have 
$$
f_{M(\pi, \phi)} (\sigma') 
= \phi^{ \dkt(\sigma' , \pi) } / Z(\phi) 
\ge \phi^{ \ell (D + R) } \phi^{ \dkt(\sigma , \pi) } / Z(\phi) 
= \phi^{ \ell (D + R) } f_{M(\pi, \phi)} (\sigma) . 
$$
Summing up this inequality over $\sigma \in \fS(K_1, \dots, K_m)$ (that is, over $\sigma' \in \fS(B'_1, \dots, B'_m)$) yields \eqref{eq:claim-pr}, thereby completing the proof. 
\end{proof}

\subsection{Main technical lemma for total variation lower bounds}

The following lemma is at the crux of proving the main identifiability result of Lemma~\ref{lem:tv-general}. 

\begin{lemma}
\label{lem:const-gap}
Consider Mallows models $M(\pi_1), \dots, M(\pi_k)$ on $\cS_n$ with a common noise parameter $\phi \in (0,1)$, and consider a set of indices $J = \{j_1, \dots, j_\ell\} \subset [n]$. 
Suppose that $\pi_1\|_J \ne \pi_i\|_J$ for any $i = 2, \dots, k$. 
Then for any fixed $C_0 \ge 1$, there exists a block structure $\cB = (B_1, B'_1), \dots, (B_m, B'_m)$ where $\bigcup_{j=1}^m B_j = J$, 
such that: 
\begin{enumerate}[leftmargin=*,label={(\arabic*)}]
\item \label{cond:1}
$\p_{\sigma \sim M(\pi_1)} \{ \sigma \text{ satisfies } \cB \} \ge \cone \defn \big[ \frac{ (1-\phi)^{3 \ell + 1} }{  8 C_0 (6 \ell)^{2 \ell} } \big]^{ (3 \ell)^\ell } $; 

\item \label{cond:2}
for each $2 \le i \le k$, we have either 
$\p_{\sigma \sim M(\pi_i)} \{ \sigma \text{ satisfies } \cB \} \le \cone/C_0$, 
or $\pi_i\|_{B_j} \ne \pi_1\|_{B_j}$ for some $j \in [m]$.  
\end{enumerate}
\end{lemma}


\begin{proof}
We use an iterative argument to prove the lemma. 
At each step $t \ge 0$, we define a block structure $\cB^{(t)} = (B^{(t)}_1, (B')^{(t)}_1), \dots, (B^{(t)}_{m^{(t)}}, (B')^{(t)}_{m^{(t)}})$ and a constant $\cone^{(t)}$ that potentially satisfy the above conditions. If not, we redefine a coarser block structure $\cB^{(t+1)}$ and a smaller constant $c^{(t+1)}_1$ in the next step, and show that the procedure must end in $\ell-1$ steps with success. 

\paragraph{Iterative construction}
Up to a relabeling of indices in $J$, we may assume without loss of generality that $\pi_1(j_1) < \cdots < \pi_1(j_\ell)$. 
Let us start with $D^{(0)} \defn 0$ and the finest block structure $\cB^{(0)} \defn (\{j_1\}, \{\pi_1(j_1) \}), \dots, (\{j_\ell\}, \{ \pi_1(j_\ell) \})$. 
Note that we have $m^{(0)} = \ell$. 


Suppose that at step $t \ge 0$, we have a block structure 
$$\cB^{(t)} = \big( B^{(t)}_1, (B')^{(t)}_1 \big), \dots, \big( B^{(t)}_{m^{(t)}}, (B')^{(t)}_{m^{(t)}} \big)$$ 
and a constant $D^{(t)} \ge 0$, such that: 
\begin{enumerate}[leftmargin=*,label={(\alph*)}]
\item \label{cond:c}
the blocks $B^{(t)}_1, \dots, B^{(t)}_{m^{(t)}}$ form an ordered partition of the ordered set $\{j_1, \dots, j_\ell\}$;

\item \label{cond:d}
$\dhaus( \pi_1(B^{(t)}_j), (B')^{(t)}_j ) \le D^{(t)}$ for all $j \in [m^{(t)}]$. 
\end{enumerate}
These conditions are clearly satisfied at step $t = 0$. 

Let us define 
\begin{align}
\cone^{(t)} \defn \frac{ \phi^{\ell D^{(t)} } (1 - \phi)^{3 \ell} }{ 2 (6 \ell)^{2 \ell} } . 
\label{eq:def-c1}
\end{align}
It then follows from  Lemma~\ref{lem:bd-block} that
$$
\p_{\sigma \sim M(\pi_1)} \{ \sigma \text{ satisfies } \cB^{(t)} \} \ge \cone^{(t)} , 
$$ 
that is, condition~\ref{cond:1} in the statement of the lemma holds. 
If condition~\ref{cond:2} also holds, then we are done. 
Otherwise, there exists $2 \le i \le k$ such that 
$$
\p_{\sigma \sim M(\pi_i)} \{ \sigma \text{ satisfies } \cB^{(t)} \} > \cone^{(t)} / C_0 
\quad \text{ and } \quad 
\pi_i\|_{B^{(t)}_j} = \pi_1\|_{B^{(t)}_j} \text{ for all } j \in [m^{(t)}] . 
$$

Note that the relative orders of $\pi_i$ and $\pi_1$ are the same on each block $B^{(t)}_j$ but different on their union $J$ (recall the assumption of $\pi_i\|_J \ne \pi_1\|_J$). 
In view of the ordering of the blocks, it is not hard to see that, there exists $j^* \in [m^{(t)}-1]$, $r_1 \in B^{(t)}_{j^*}$ and $r_2 \in B^{(t)}_{j^*+1}$ such that $\pi_1(r_1) < \pi_1(r_2)$ while $\pi_i(r_1) > \pi_i(r_2)$. 
If no such $j^*$ exists, then every element of $\pi_i(B^{(t)}_{j^*})$ is smaller than every element of $\pi_i(B^{(t)}_{j^*+1})$ for all $j^* \in [m^{(t)}-1]$. Together with the fact that the relative orders of $\pi_i$ and $\pi_1$ coincide on each block, this implies $\pi_i\|_J = \pi_1\|_J$, which is a contradiction. 

Let us set $s_1 \defn \max \, (B')^{(t)}_{j^*}$ and $s_2 \defn \min \, (B')^{(t)}_{j^*+1}$, and we have $s_1 < s_2$ by the definition of a block structure. 
Note that 
every $\sigma$ satisfying $\cB^{(t)}$ must have $\sigma(r_1) \le s_1$ and $\sigma(r_2) \ge s_2$. 
Since $\pi_i(r_1) > \pi_i(r_2)$, it holds that either $|\sigma(r_1) - \pi_i(r_1)| > (s_2 - s_1)/2$ or $|\sigma(r_2) - \pi_i(r_2)| > (s_2 - s_1)/2$. 
Consequently, 
\begin{align*}
&\p_{\sigma \sim M(\pi_i)} \big\{ |\sigma(r_1) - \pi_i(r_1)| > (s_2 - s_1)/2 \big\}
+ \p_{\sigma \sim M(\pi_i)} \big\{ |\sigma(r_2) - \pi_i(r_2)| > (s_2 - s_1)/2 \big\} \\
&\ge \p_{\sigma \sim M(\pi_i)} \{ \sigma \text{ satisfies } \cB^{(t)} \} 
> \cone^{(t)} / C_0  . 
\end{align*}
Lemma~\ref{lem:dev}, on the other hand, gives the upper bound 
\begin{align*}
\p_{\sigma \sim M(\pi_i)} \big\{ |\sigma(r_1) - \pi_i(r_1)| > (s_2 - s_1)/2 \big\}
\le \frac{2 \phi^{(s_2 - s_1)/2}}{1 - \phi} , 
\end{align*}
and the same bound also holds with $r_1$ replaced by $r_2$. 
Combining the above inequalities yields 
\begin{align}
\frac{4 \phi^{(s_2 - s_1)/2}}{1 - \phi} > \frac{c^{(t)}_1}{C_0} 
\quad \Longrightarrow \quad 
s_2 - s_1 < 2 \log_\phi \frac{ c^{(t)}_1 (1- \phi) }{ 4 C_0 } ,
\label{eq:close-block}
\end{align}
where $\log_\phi(\cdot)$ denotes the logarithm with respect to base $\phi$. 

Intuitively, this shows that the blocks $(B')^{(t)}_{j^*}$ and $(B')^{(t)}_{j^*+1}$ are not too far apart. 
We now merge them to define a coarser block structure 
$$\cB^{(t+1)} = \big( B^{(t+1)}_1, (B')^{(t+1)}_1 \big), \dots, \big( B^{(t+1)}_{m^{(t+1)}}, (B')^{(t+1)}_{m^{(t+1)}} \big)$$ 
with $m^{(t+1)} = m^{(t)} - 1$ as follows. 
We set $B^{(t+1)}_{j^*} \defn B^{(t)}_{j^*} \cup B^{(t)}_{j^*+1}$, and define $(B')^{(t+1)}_{j^*}$ to be the contiguous block which extends $(B')^{(t)}_{j^*}$ by $|(B')^{(t)}_{j^*+1}|$ elements to the right. 
Moreover, we set $B^{(t+1)}_j \defn B^{(t)}_j$ and $(B')^{(t+1)}_j \defn (B')^{(t)}_j$ for $1 \le j < j^*$, and set $B^{(t+1)}_j \defn B^{(t)}_{j+1}$ and $(B')^{(t+1)}_j \defn (B')^{(t)}_{j+1}$ for $j^* < j \le m^{(t+1)} = m^{(t)} - 1$. 
Note that $\cB^{(t+1)}$ is a valid block structure per Definition~\ref{def:block}, that is, $|B_j|=|B_j'|$ and $B_j'$ is contiguous for each $j \in [m^{(t+1)}]$.

Moreover, it is clear from the definition of the only new block $(B')^{(t+1)}_{j^*}$ that 
$$
\dhaus \Big( (B')^{(t+1)}_{j^*},  (B')^{(t)}_{j^*} \cup (B')^{(t)}_{j^*+1} \Big) < s_2 - s_1 < 2 \log_\phi \frac{ c^{(t)}_1 (1- \phi) }{ 4 C_0 }  
$$ 
thanks to \eqref{eq:close-block}. 
As a result, we have 
\begin{align*}
&\dhaus \Big( \pi_1(B^{(t+1)}_{j^*}), (B')^{(t+1)}_{j^*} \Big) \\
&\stackrel{(i)}{\le} \dhaus \Big( \pi_1 \big( B^{(t)}_{j^*} \cup B^{(t)}_{j^*+1} \big), (B')^{(t)}_{j^*} \cup (B')^{(t)}_{j^*+1} \Big) + \dhaus \Big( (B')^{(t)}_{j^*} \cup (B')^{(t)}_{j^*+1}, (B')^{(t+1)}_{j^*} \Big) \\ 
&\stackrel{(ii)}{\le} D^{(t)} + 2 \log_\phi \frac{ c^{(t)}_1 (1- \phi) }{ 4 C_0 } , 
\end{align*}
where (i) follows from the definition of $B^{(t+1)}_{j^*}$ and the triangle inequality, and (ii) follows from condition~\ref{cond:d} above for step $t$. 
Hence if we define 
\begin{align}
D^{(t+1)} \defn D^{(t)} + 2 \log_\phi \frac{ c^{(t)}_1 (1- \phi) }{ 4 C_0 } , 
\label{eq:def-d}
\end{align}
then 
condition~\ref{cond:d} is also satisfied for step $t+1$. 
By construction, condition~\ref{cond:c} continues to hold for step $t+1$. 
Therefore, we can iterate this construction. 

Finally, since $m^{(0)} = \ell$ and $m^{(t+1)} = m^{(t)} - 1$, the procedure has to end in $\ell-1$ steps when $m^{(\ell)} = 1$. 
In this situation, we simply has one block in the block structure $\cB^{(\ell-1)} = (J, (B')^{(\ell-1)}_1)$, and condition~\ref{cond:2} in the statement of the lemma is necessarily achieved as $\pi_i\|_J \ne \pi_1\|_J$ for all $2 \le i \le k$ by assumption. 
Thus the construction ends with success. 

\paragraph{Lower bound on $c^{(t)}_1$}
It remains to give a lower bound on $c^{(t)}_1$. Substituting \eqref{eq:def-c1} into \eqref{eq:def-d} yields
\begin{align*}
D^{(t+1)} = D^{(t)} + 2 \log_\phi \frac{ \phi^{\ell D^{(t)} } (1 - \phi)^{3 \ell + 1} }{ 8 C_0 (6 \ell)^{2 \ell} } 
= D^{(t)} 
( 1 + 2 \ell ) 
+ 2 \log_\phi \frac{ (1 - \phi)^{3 \ell + 1} }{ 8 C_0 (6 \ell)^{2 \ell} }  . 
\end{align*} 
Combining this relation with $D^{(0)} = 0$, we see that for any $t \le \ell-1$, 
\begin{align*}
D^{(t)} \le 
2 ( 1 + 2 \ell )^{\ell - 1} \log_\phi \frac{ (1 - \phi)^{3 \ell + 1} }{ 8 C_0 (6 \ell)^{2 \ell} }  . 
\end{align*}
It then follows that 
$$
\phi^{D^{(t)}} \ge \Big[ \frac{ (1-\phi)^{3 \ell + 1} }{  8 C_0 (6 \ell)^{2 \ell} } \Big]^{\textstyle 2 (1+2\ell)^{\ell-1} } . 
$$
Therefore, we conclude using definition \eqref{eq:def-c1} that 
\begin{align*}
\cone^{(t)} 
\ge \frac{ (1 - \phi)^{3 \ell} }{ 2 (6 \ell)^{2 \ell} } \Big[ \frac{ (1-\phi)^{3 \ell + 1} }{  8 C_0 (6 \ell)^{2 \ell} } \Big]^{\textstyle 2 \ell ( 1+2\ell )^{\ell-1} } 
\ge \Big[ \frac{ (1-\phi)^{3 \ell + 1} }{  8 C_0 (6 \ell)^{2 \ell} } \Big]^{\textstyle (3 \ell)^\ell } 
\end{align*} 
for any $t \le \ell - 1$. 
This completes the proof. 
\end{proof}

\subsection{Total variation lower bounds}

We first state a general result that implies both Propositions~\ref{prop:tv-lower} and \ref{prop:tv-weight}.
Let $\cS_n|_J$ denote the set of injections from $J$ to $[n]$. 
Recall that for the Mallows model $M(\pi, \phi)$, the marginalized model $M(\pi, \phi)|_J$ is a distribution on $\cS_n|_J$ defined in \eqref{eq:fMJ}.

\begin{lemma} \label{lem:tv-general}
Consider $K$ Mallows models $M(\pi_1), \dots, M(\pi_K)$ on $\cS_n$ with a common noise parameter $\phi \in (0, 1)$. 
Fix a set of indices $J \subset [n]$ and let $\ell \defn |J|$. 
Let $\alpha_1, \dots, \alpha_K$ be real numbers such that: (1) $\alpha_1 = 1$; (2) $\alpha_i \ge 0$ for every $i \in [K]$ such that $\pi_i\|_J = \pi_1\|_J$; (3) $|\alpha_i| \le 1/\gamma$ for every $i \in [K]$ such that $\pi_i\|_J \ne \pi_1\|_J$, where $0 < \gamma \le 1$. 
For any function $v: \cS_n|_J \to \R$, we write $\E_i [v]  \equiv \E_{\rho \sim M(\pi_i)|_J} [ v(\rho) ]$ for each $i \in [K]$. 
Define $\eta(k, \ell, \phi, \gamma)$ as in \eqref{eq:def-eta}. 
Then there exists a test function $v : \cS_n|_J \to [-1,1]$ such that 
\begin{align}
\sum_{i=1}^{K} \alpha_i \E_i [v] 
\ge \frac{2}{\gamma} \eta(K/2, \ell, \phi, \gamma) . 
\label{eq:lb-str}
\end{align}
\end{lemma}

\subsubsection{Proof of Proposition~\ref{prop:tv-lower}}

By the assumption $\{\pi_1\|_J, \dots, \pi_k\|_J\} \ne \{\pi'_1\|_J, \dots, \pi'_k\|_J\}$, 
up to a relabeling of elements within $\{\pi_1\|_J, \dots, \pi_k\|_J\}$ or $\{\pi'_1\|_J, \dots, \pi'_k\|_J\}$, and possibly a swap of the two sets, we may assume that $\pi_1\|_J \ne \pi'_i\|_J$ for any $i \in [k]$. 
To prove that the total variation distance satisfies
$$
\TV( \cM|_J, \cM'|_J ) 
= \frac 12 \sup_{\|v\|_\infty \le 1} \bigg| \E_{\rho \sim \cM|_J} [ v(\rho) ] - \E_{\rho \sim \cM'|_J} [ v(\rho) ]  \bigg|  
\ge \eta(k, \ell, \phi, \gamma) , 
$$
it suffices to find a test function $v : \cS_n|_J \to [-1,1]$ such that 
\begin{align}
\sum_{i=1}^{k} w_i \E_i [v] - \sum_{i=1}^{k} w'_i \E'_i[v] \ge 2 \, \eta(k, \ell, \phi, \gamma) . 
\label{eq:lb-v}
\end{align}
Setting $K = 2k$, $\alpha_i = w_i / w_1$, $\alpha_{k+i} = - w'_i / w_1$ and $\pi_{k+i} = \pi'_i$ for $i \in [k]$, we see that all the conditions in Lemma~\ref{lem:tv-general} are satisfied. Therefore, \eqref{eq:lb-v} follows from \eqref{eq:lb-str}.

\subsubsection{Proof of Proposition~\ref{prop:tv-weight}}

We need to prove 
$$
\TV( \cM|_J , \cM'|_J ) = \frac 12 \sup_{\|v\|_\infty \le 1} \bigg| \sum_{i=1}^{k} w_i \E_i [v] - \sum_{i=1}^{k} w'_i \E_i [v] \bigg| \ge \xi \cdot \eta(k/2, \ell, \phi, 1) . 
$$
Hence, it suffices to find a test function $v : \cS_n|_J \to [-1,1]$ such that 
\begin{align}
\sum_{i=1}^{k} ( w_i - w'_i ) \E_i [v] \ge 2 \xi \cdot \eta(k/2, \ell, \phi, 1) . 
\label{eq:lb-w}
\end{align}
Up to a relabeling, we may assume that $w_1 - w'_1 = \xi > 0$. 
Let us apply Lemma~\ref{lem:tv-general} with $K = k$, $\gamma = 1$ and $\alpha_i = (w_i - w'_i) / \xi$ for $i \in [k]$. 
Then $|\alpha_i| \le 1$, and by assumption, $\pi_i\|_J \ne \pi_1\|_J$ for any $i \ne 1$. 
Hence all the conditions in Lemma~\ref{lem:tv-general} are satisfied. Therefore, \eqref{eq:lb-w} follows from \eqref{eq:lb-str}.

\subsubsection{Proof of Lemma~\ref{lem:tv-general}}

Let us define
$$
I_1 \defn \{ i \in [K] : \pi_i\|_J = \pi_1\|_J \} .
$$
We apply Lemma~\ref{lem:const-gap} to the models $M(\pi_1)$ and $\{M(\pi_{i}): i \in [K] \setminus I_1\}$ with 
$C_0 = 2K /  \gamma $,
to obtain a block structure $\cB = (B_1, B'_1), \dots, (B_m, B'_m)$ where $\bigcup_{j \in [m]} B_j = J$, such that: 
\begin{itemize}[leftmargin=*]
\item
$\p_{\sigma \sim M(\pi_1)} \{ \sigma \text{ satisfies } \cB \} \ge \cone \defn \big[ \frac{ \gamma (1-\phi)^{3 \ell + 1} }{  16 K (6 \ell)^{2 \ell} } \big]^{ (3 \ell)^\ell } ;
$  

\item
There exists $I_2 \subset [K] \setminus I_1$ such that
for each $i \in I_2$, we have $\pi_i\|_{B_j} \ne \pi_1\|_{B_j}$ \emph{for some} $j \in [m]$;  

\item
For each $i \in I_3 \defn [K] \setminus (I_1 \cup I_2)$, we have that $\pi_i\|_{B_j} = \pi_1\|_{B_j}$ \emph{for all} $j \in [m]$, and that
$\p_{\sigma \sim M(\pi_i)} \{ \sigma \text{ satisfies } \cB \} \le \frac{ \cone \gamma }{ 2K }$. 
\end{itemize}

With the block structure $\cB$ constructed above, we define the test function $v$ in \eqref{eq:lb-str} by 
$$
v(\sigma|_J) \defn \1 \{ \sigma \text{ satisfies } \cB \} \cdot \prod_{j = 1}^m u_j (\sigma\|_{B_j}) , 
$$
where each $u_j: \cS_{|B_j|} \to [-1,1]$ is to be specified later. 
Note that $v$ is well-defined because $\bigcup_{j \in [m]} B_j = J$ so that: (1) whether $\sigma$ satisfies $\cB$ is fully determined by $\sigma|_J$, and (2) $\sigma\|_{B_j}$ is fully determined by $\sigma|_J$ for each $j \in [m]$. 
In addition, we clearly have $\|v\|_\infty \le 1$. 

To compute the expectation $\E_i[v]$,  
we use the definitions of $M(\pi_i)|_J$ and $v$ to obtain 
$$
\E_i[v] = \E_{\rho \sim M(\pi_i)|_J} [ v(\rho) ] 
= \E_{\sigma \sim M(\pi_i)} [ v(\sigma|_J) ] 
= \E_{\sigma \sim M(\pi_i)} \bigg[ \1 \{ \sigma \text{ satisfies } \cB \} \cdot \prod_{j = 1}^m u_j (\sigma\|_{B_j}) \bigg] .  
$$
It then follows from the conditional independence in \eqref{eq:prod} that 
\begin{align}
\E_i[v] = \p_{\sigma \sim M(\pi_i) } \{ \sigma \text{ satisfies } \cB \} \cdot \prod_{j=1}^m \E_{M(\pi_i\|_{B_j})} [u_j] . 
\label{eq:v-u-2}
\end{align}

We now define the test function $u_j: \cS_{|B_j|} \to [-1,1]$ for each $j \in [m]$. 
Since $|B_j| \leq |J|= \ell$ is small, we can afford to apply the crude construction in Lemma~\ref{lem:pre-bd}.
First, if $|B_j| = 1$, the set $\cS_{|B_j|}$ consists of a singleton, and we simply define $u_j = 1$. 
Next, if $j \in [m]$ with $|B_j| \ge 2$, let $k'$ be the number of distinct elements of $\{ M(\pi_i\|_{B_j}) : i \in [K] \}$.  Applying Lemma~\ref{lem:pre-bd} with the distinct models in $\{ M(\pi_i\|_{B_j}) : i \in [K] \}$, we obtain $u_j$ such that
\begin{align}
\begin{cases}
\E_{M(\pi_1\|_{B_j})} [ u_j ] \ge \frac{1}{|B_j|!} \Big[ \frac{ (1-\phi)^{|B_j|} }{ \sqrt{ |B_j|! } } \Big]^{k'} 
\ge \frac{1}{\ell!} \Big[ \frac{ (1-\phi)^\ell }{ \sqrt{ \ell! } } \Big]^{K} & \text{ if } \pi_i\|_{B_j} = \pi_1\|_{B_j} , \\
\E_{M(\pi_i\|_{B_j})} [ u_j ] = 0  & \text{ if } \pi_i\|_{B_j} \ne \pi_1\|_{B_j} ,
\end{cases}
\label{eq:3-cases}
\end{align}
where we used the trivial bounds $|B_j| \le \ell$ and $k' \le K$. 

In summary, we have:
\begin{itemize}[leftmargin=*]
\item 
For $i = 1$, we have $\p_{\sigma \sim M(\pi_1)} \{ \sigma \text{ satisfies } \cB \} \ge \cone$. 
Thus we obtain from \eqref{eq:v-u-2} and \eqref{eq:3-cases} that 
$
\E_1[ v ] \ge \cone \cdot \frac{1}{(\ell!)^m} \big[ \frac{ (1-\phi)^\ell }{ \sqrt{ \ell! } } \big]^{Km} . 
$

\item 
For $i \in I_1 \setminus \{1\}$, we have $\E_i [ v ] \ge 0$ by \eqref{eq:v-u-2} and \eqref{eq:3-cases}. 
\item
For $i \in I_2$, we have $\pi_i\|_{B_j} \ne \pi_1\|_{B_j}$ for some $j \in [m]$ by the construction of the block structure;  
for this index $j$, it holds that $\E_{M(\pi_i\|_{B_j})} [ u_j ] = 0$ by \eqref{eq:3-cases}. 
Therefore, $\E_i[v] = 0$ by \eqref{eq:v-u-2}. 

\item
For $i \in I_3$, we have that $\p_{\sigma \sim M(\pi_i)} \{ \sigma \text{ satisfies } \cB \} \le \frac{ \cone  \gamma }{ 2K } \le \frac{ \gamma }{ 2K } \p_{\sigma \sim M(\pi_1)} \{ \sigma \text{ satisfies } \cB \}$ and that $\pi_i\|_{B_j} = \pi_1\|_{B_j}$ for all $j \in [m]$ by our construction of the block structure.  
Together with \eqref{eq:v-u-2}, this implies that 
$0 \le \E_i[v] 
\le \frac { \gamma }{ 2K } \E_1[v] $ since $m \le \ell$.
\end{itemize}
Finally, combining the above with the assumption that $\alpha_1 = 1$, $\alpha_i \ge 0$ for $i \in I_1$, and $|\alpha_i| \le 1/\gamma$ for $[K] \setminus I_1$, we conclude that 
\begin{align*}
\sum_{i=1}^{K} \alpha_i \E_i [v]
&= \E_1[v] + \sum_{i \in I_1 \setminus \{1\}} \alpha_i \E_i [v]  + \sum_{i \in I_2} \alpha_i \E_i [v] + \sum_{i \in I_3} \alpha_i \E_i [v] \\
&\ge \E_1[ v ] - \sum_{i \in I_3} \frac{1}{\gamma} \cdot \frac{\gamma}{2K} \E_1[v] \\
&\ge \frac 12 \E_1[ v ]  
\ge \frac 12 \Big[ \frac{ \gamma (1-\phi)^{3 \ell + 1} }{  16 K (6 \ell)^{2 \ell} } \Big]^{ (3 \ell)^\ell } \cdot 
\frac{1}{(\ell!)^m} \Big[ \frac{ (1-\phi)^\ell }{ \sqrt{ \ell! } } \Big]^{Km} . 
\end{align*}
Using $m \le \ell$ and $\eta(K/2, \ell, \phi, \gamma) =
( \frac{ \gamma }{ 3K } )^{ (3 \ell)^{\ell+1} } ( \frac{1-\phi}{\ell} )^{ (4 \ell)^\ell + K \ell^2 }$, it is not hard to simplify the above bound to obtain \eqref{eq:lb-str}, thereby finishing the proof. 

\subsection{Proof of Proposition~\ref{prop:tv-conv}}


For $J \subset [n]$ and $r \in \naturals$, we define an event $\Sigma(r) \subset \cS_n$ by
$$
\Sigma(r) \defn \big\{ \sigma \in \cS_n : \text{there exists } j \in J \text{ such that } | \sigma(j) - \pi_i(j) | \ge r \text{ for all } i \in [k] \big\} . 
$$
Since $\cM$ is a probability measure, we have $\cM \big( \Sigma(r) \big) = \p_{\sigma \sim \cM} \{ \sigma \in \Sigma(r) \}$. 
Moreover, recall that $\cS_{n,J}$ denotes the set of injections $\rho : J \to [n]$. 
Define an event $\tilde \Sigma(r) \subset \cS_{n,J}$ by 
$$
\tilde \Sigma(r) \defn \big\{ \rho \in \cS_{n,J} : \text{there exists } j \in J \text{ such that } | \rho(j) - \pi_i(j) | \ge r \text{ for all } i \in [k] \big\} . 
$$


%

Note that $\tilde \Sigma(r)$ and $\Sigma(r)$ impose the same constraint, with the former on $\sigma$ and the latter on $\rho$. 
By the definition of $\cM|_J$ in \eqref{eq:fMJ}, we obtain
\begin{align*}
p \defn \calM|_J \big( \tilde \Sigma(r) \big) = \sum_{\rho \in \tilde \Sigma(r)} f_{\cM|_J} (\rho) = \sum_{\rho \in \tilde \Sigma(r)} \p_{\sigma \sim \cM} \big( \sigma|_J = \rho \big)
= \cM \big( \Sigma(r) \big) . 
\end{align*}
For the empirical distribution $\cM_N|_J$ defined by \eqref{eq:emp}, we have 
$$
p_N \defn \calM_N|_J \big( \tilde \Sigma(r) \big) = \sum_{\rho \in \tilde \Sigma(r)} f_{\cM_N|_J} (\rho) = \sum_{\rho \in \tilde \Sigma(r)} \frac 1N \sum_{m=1}^N \1 \big\{ \sigma_m|_J = \rho \big\} 
= \frac 1N \sum_{m=1}^N \1 \big\{ \sigma_m \in \Sigma(r) \big\} .
$$

Having these quantities defined, we can bound the total variation in consideration as 
\begin{align}
2 \, \TV(\cM|_J, \cM_N|_J) \leq p+p_N + \sum_{\rho \in \cS_{n,J} \setminus \tilde \Sigma(r)} \big| f_{\cM|_J} (\rho) - f_{\cM_N|_J} (\rho) \big|.
\label{eq:tv3}
\end{align}
We now control each of the three terms on the right hand side. 
First, a union bound yields that 
\begin{align}
p = \cM \big( \Sigma(r) \big) &\le \sum_{j \in J} \p_{\sigma \sim \cM} \big\{ | \sigma(j) - \pi_i(j) | \ge r \text{ for all } i \in [k] \big\} \notag \\
&\le \sum_{j \in J} \sum_{i=1}^k w_i \p_{\sigma \sim M(\pi_i)} \big\{ | \sigma(j) - \pi_i(j) | \ge r \big\} 
\stackrel{(i)}{\le} \sum_{j \in J} \sum_{i=1}^k w_i \frac{2 \phi^r}{ 1 - \phi }
= \frac{2 \ell \phi^r}{ 1 - \phi } ,
\label{eq:outside}
\end{align}
where $(i)$ follows from Lemma~\ref{lem:dev}. 

Second, since $\1 \{ \sigma_m \in \Sigma(r) \}$ are independent $\Ber \big( \cM \big( \Sigma(r) \big) \big)$ random variables for $m \in [N]$, we have $N \cdot p_N \sim \mathrm{Binomial}(N,p)$ in view of the formulas for $p$ and $p_N$ above. 
Hence Bernstein's inequality yields
$$
\p \big\{ p_N > p + t \big\} \le \exp \Big( \frac{ - N t^2/2 }{ p + t/3 } \Big) 
$$
for any $t>0$. Taking $t = s/2$ in the above bound and combining it with \eqref{eq:outside}, we obtain 
\begin{align}
\p \Big\{ p + p_N
> \frac{4 \ell \phi^r}{ 1 - \phi } + \frac s2 \Big\} \le \exp \Big( \frac{ - N s^2 / 8 }{ 2 \ell \phi^r / (1 - \phi) + s / 6 } \Big) . 
\label{eq:out}
\end{align}

Third, in view of definition~\eqref{eq:emp} where $\1 \{ \sigma_m|_J = \rho \}$ are independent $\Ber \big( f_{\cM|_J}(\rho) \big)$ random variables, Hoeffding's inequality yields
$$
\p \big\{ | f_{\cM_N|_J}(\rho)  - f_{\cM|_J}(\rho) | > t \big\} \le 2 \exp ( - 2 N t^2 ) 
$$
for any $t > 0$. For each $\rho \in \cS_{n,J} \setminus \tilde \Sigma(r)$, we have that for all $j \in J$, $|\rho(j) - \pi_i(j)| < r$ for some $i \in [k]$. Hence there are at most $2kr$ possible choices for each $\rho(j)$, and the cardinality of $\cS_{n,J} \setminus \tilde \Sigma(r)$ is bounded by $(2kr)^\ell$. A union bound over $\cS_{n,J} \setminus \tilde \Sigma(r)$ then implies that, for any $t > 0$, 
\begin{align}
\p \Big\{ \sum_{\rho \in \cS_{n,J} \setminus \tilde \Sigma(r)} \big| f_{\cM|_J} (\rho) - f_{\cM_N|_J} (\rho) \big| > (2kr)^\ell t \Big\} 
\le 2 (2kr)^\ell \exp( - N t^2) . 
\label{eq:in}
\end{align}

Finally, we plug \eqref{eq:out} and \eqref{eq:in} into \eqref{eq:tv3}, choose $r$ to be the smallest integer such that 
$
\frac{4 \ell \phi^r}{ 1 - \phi } \le \frac s2 ,
$
and then set $t = \frac{ s }{ (2kr)^\ell }$, to obtain 
$$
\p \big\{ \TV(\cM|_J, \cM_N|_J) > s \big\} \le \exp \Big( \frac{ - N s^2 / 8 }{ 2 \ell \phi^r / (1 - \phi) + s / 6 } \Big) + 2 (2kr)^\ell \exp \Big( - N \frac{ s^2 }{ (2kr)^{2 \ell} } \Big) . 
$$
By the definition of $r$, we have 
$
\exp \big( \frac{ - N s^2 / 8 }{ 2 \ell \phi^r / (1 - \phi) + s / 6 } \big) 
\le \exp ( - N \frac{ 3 s }{ 10 } ) . 
$
On the other hand, the definition of $r$ also implies that $\frac{4 \ell \phi^r}{ 1 - \phi } > \frac{s \phi}{2}$. Hence we obtain $r < \log_\phi \frac{ s \phi (1-\phi) }{ 8 \ell } \le 1 +  \frac{1}{1 - \phi} \log \frac{8 \ell}{s (1-\phi)} $, which completes the proof. 

\subsection{Proof of Theorem~\ref{thm:subroutine}}

We now prove Theorem~\ref{thm:subroutine}. 
First, we apply Proposition~\ref{prop:tv-conv} with $s = \eta/6$ where $\eta = \eta(k, \ell, \phi, \gamma)$ is defined in \eqref{eq:def-eta}. 
It is easy to check that, for any $\delta \in (0,0.1)$, if $N \ge \frac{36(2kq)^{2\ell}}{\eta^2} \log \frac{8(2kq)^\ell}{\delta}$ where $q = 1 +  \frac{1}{1 - \phi} \log \frac{48 \ell}{\eta (1-\phi)}$, then the tail probability in \eqref{eq:tv-prob} is at most $\delta/2$, that is, 
\begin{align}
\TV( \cM|_J, \cM_N|_J ) \le \eta / 6  
\label{eq:tv1}
\end{align}
with probability at least $1 - \delta/2$. 

Next, recall the collection $\scrM$ of Mallows models with discretized weights as defined in \eqref{eq:mix-class}. 
Consider a mixture $\cM' = \sum_{i=1}^k \frac{r_i}{L} M(\pi_{\rho_i}, \phi) \in \scrM$ and its empirical version $\cM'_{N'}$ as constructed in lines~\ref{line:model-p}--\ref{line:model-p-e} of Algorithm~\ref{alg:suborder}, where $L = \lceil 3k/\eta \rceil$.  
Then Proposition~\ref{prop:tv-conv} can be applied again to show that, if $N' \ge \frac{36(2kq)^{2\ell}}{\eta^2} \log \frac{8(2kq)^\ell L^k n^{k \ell}}{\delta}$, then
\begin{align}
\TV( \cM'|_J, \cM'_{N'}|_J ) \le \eta / 6 
\label{eq:tv2}
\end{align}
with probability at least $1 - \frac{ \delta }{ 2 L^k n^{k \ell} }$. 
Recall that $|\scrM| \leq L^k n^{k \ell}$.
Let $\cE$ denote the event that \eqref{eq:tv1} holds and \eqref{eq:tv2} holds for all $\cM' \in \scrM$. By a union bound, 
$\cE$ has probability at least $1 - \delta$. 

We observe that there exists $\cM' \in \scrM$ for which $| \frac{r_i}{L} - w_i | \le \frac{\eta}{3k}$ and $\pi_{\rho_i}|_J = \pi_i|_J$ for each $i \in [k]$. 
This is because $L \ge 3k / \eta$, and if $\rho_i \defn \pi_i|_J$ then $\pi_{\rho_i}|_J = \pi_i|_J$ by definition. 
For this $\cM'$, we have 
\begin{align*}
\TV( \cM|_J, \cM'|_J ) 
&= \TV\pth{\sum_{i=1}^k w_i M(\pi_i)|_J, \sum_{i=1}^k \frac{r_i}{L} M(\pi_{\rho_i})|_J}\\
&\stackrel{(i)}{=} \TV\pth{\sum_{i=1}^k w_i M(\pi_i)|_J, \sum_{i=1}^k \frac{r_i}{L} M(\pi_i)|_J}\\
&\le \frac 12 \sum_{i=1}^k \Big| w_i - \frac{r_i}{L} \Big| \le \frac 12 \cdot k = \frac{\eta}{6} , 
\end{align*}
where step ($i$) follows from Lemma~\ref{lem:restrict}. 
This combined with \eqref{eq:tv1} and \eqref{eq:tv2} shows that on the event $\cE$,
$$
\TV( \cM_N|_J, \cM'_{N'}|_J ) \le \TV(  \cM|_J , \cM_{N}|_J ) + \TV( \cM|_J, \cM'|_J ) + \TV( \cM'|_J, \cM'_{N'}|_J ) \le \eta / 2. 
$$
Therefore, Algorithm~\ref{alg:suborder} succeeds in returning a nonempty set of relative orders in view of line~\ref{line:return}. 

Suppose that Algorithm~\ref{alg:suborder} returns a set of relative orders $\{ \pi_{\rho_i}\|_J : i \in [k] \}$ that is not equal to the set $\{ \pi_i\|_J : i \in [k] \}$.  Then Proposition~\ref{prop:tv-lower} implies that
$
\TV( \cM|_J, \cM'|_J ) \ge \eta . 
$
As a result, we have that on the event $\cE$, 
$$
\TV( \cM_N|_J, \cM'_{N'}|_J ) \ge \TV( \cM|_J, \cM'|_J ) - \TV(  \cM|_J , \cM_{N}|_J ) - \TV( \cM'|_J, \cM'_{N'}|_J ) \ge 2 \eta / 3 .
$$
This contradicts condition $\TV( \cM'_{N'}|_J , \cM_N|_J ) \le \eta / 2$ in line~\ref{line:return}. Therefore, the set of relative orders returned by Algorithm~\ref{alg:suborder} must be $\{ \pi_i\|_J : i \in [k] \}$. 

Finally, with some tedious but elementary algebra, we can check that the conditions 
\begin{align*}
N &\ge \zeta(k, \ell, \phi, \gamma) \cdot \log \frac{1}{\delta} = e^{ (9 \ell)^{\ell+1} } \Big( \frac{\ell}{1-\phi} \Big)^{ 3 (4 \ell)^\ell +8 k \ell^2 }  \Big( \frac{ k }{ \gamma } \Big)^{ (6 \ell)^{\ell+1} } \log \frac{1}{\delta} , \\
N' &\ge \zeta(k, \ell, \phi, \gamma) \cdot \log \frac{n}{\delta} = e^{ (9 \ell)^{\ell+1} } \Big( \frac{\ell}{1-\phi} \Big)^{ 3 (4 \ell)^\ell +8 k \ell^2 }  \Big( \frac{ k }{ \gamma } \Big)^{ (6 \ell)^{\ell+1} } \log \frac{n}{\delta} ,
\end{align*}
assumed in the statement of Theorem~\ref{thm:subroutine} are stronger than the conditions 
\begin{align*}
N &\ge \frac{36(2kq)^{2\ell}}{\eta^2} \log \frac{8(2kq)^\ell}{\delta} , \\
N' & \ge \frac{36(2kq)^{2\ell}}{\eta^2} \log \frac{8(2kq)^\ell L^k n^{k \ell}}{\delta} ,
\end{align*} 
required above, respectively, where $\eta = \big( \frac{ \gamma }{ 6 k } \big)^{ (3 \ell)^{\ell+1} } \big( \frac{1-\phi}{\ell} \big)^{ (4 \ell)^\ell + 2 k \ell^2 }$, $q = 1 +  \frac{1}{1 - \phi} \log \frac{48 \ell}{\eta (1-\phi)}$, and $L = \lceil 3k/\eta \rceil$. Hence the proof is complete.

\subsection{Proof of Theorem~\ref{thm:weight}}

Since the sample size $N$ is assumed to be sufficiently large, Corollary~\ref{cor:main-sc} guarantees the exact recovery of the central permutations with probability at least $1 - \delta/2$, so we may assume that, up to a relabeling, $\hat \pi_i = \pi_i$ for each $i \in [k]$. 
It remains to study the estimation error for $\hat w$ defined in line~\ref{line:weight} of Algorithm~\ref{alg:weight}. 
Let $\xi > 0$ denote the aimed accuracy of estimating each weight $w_i$. 

Let $\cI$ and $J$ be defined by lines~\ref{line:returned-tuple}--\ref{line:subset-j} of Algorithm~\ref{alg:weight}. 
By Lemma~\ref{lem:distinct}, the tuple $\cI$ returned by Algorithm~\ref{alg:find-tuple} consists of at most $k-1$ pairs of distinct indices in $[n]$ and satisfies that $\chi(\pi_i, \cI) \ne \chi(\pi_j, \cI)$ for any distinct $i, j \in [k]$. 
Since $J$ is the subset of $[n]$ that contains all indices appearing in $\cI$, we have $\ell \defn |J| \le 2k-2$ and $\pi_i\|_J \ne \pi_j\|_J$ for any  distinct $i, j \in [k]$. 
The rest of the proof is analogous to that of Theorem~\ref{thm:subroutine}, so we only present a sketch. 

We first apply Proposition~\ref{prop:tv-conv} with $s = \xi \eta / 6$ where $\eta = \eta(k/2, \ell, \phi, 1)$, to obtain that
$$
\TV( \cM|_J, \cM_N|_J ) \le \xi \eta / 6  
$$
with probability at least $1 - \delta/4$ 
if $N \ge \frac{1}{\xi^2} (\log \frac 1\xi)^{2\ell+1} \zeta(k, 2k-2, \phi, 1) \cdot \log \frac{4}{\delta}$, where $\zeta$ is defined in \eqref{eq:def-zeta}. 
Moreover, let $\cR(L)$ be defined by line~\ref{line:def-r} of Algorithm~\ref{alg:weight}. Then $\cR(L)$ has cardinality at most $L^k$. 
If we choose $N' \ge \frac{1}{\xi^2} (\log \frac 1\xi)^{2\ell+1} \zeta(k, 2k-2, \phi, 1) \cdot \log \frac{4 L^k}{\delta}$,  
then Proposition~\ref{prop:tv-conv} together with a union bound over all $r$ in $\cR(L)$ implies that 
$$
\TV( \cM'(r)|_J, \cM'_{N'}(r)|_J ) \le \xi \eta / 6 
$$
with probability $1 - \delta/4$. 
In the sequel, we condition on the event $\cE$ of probability at least $1 - \delta$ that both of the above bounds hold. 

If we choose $L \ge \frac{3k}{\xi \eta}$, then there exists $r \in \cR(L)$ for which $|\frac{r_i}{L} - w_i| \le \frac{\xi \eta}{3k}$ for any $i \in [k]$. 
Using the same argument as in the proof of Theorem~\ref{thm:subroutine}, we obtain 
$$
\TV( \cM|_J, \cM'(r)|_J ) \le \xi \eta / 6 .
$$ 
As a result, for this $r$ it holds that 
$$
\TV( \cM'_{N'}(r)|_J, \cM_N|_J ) \le \xi \eta / 2 . 
$$

On the other hand, for any $r \in \cR(L)$, if there exists $i \in [k]$ for which $|\frac{r_i}{L} - w_i| \ge \xi$, then Proposition~\ref{prop:tv-weight} implies that 
$$
\TV( \cM'_{N'}(r)|_J, \cM_N|_J ) \ge 2 \xi \eta / 3 
$$
on the event $\cE$. 
Consequently, such an $r$ cannot be equal to $L \hat w$ by the definition of $\hat w$ in line~\ref{line:weight} of Algorithm~\ref{alg:weight}. 
We conclude that $\hat w$ must satisfy that $|\hat w_i - w_i| \le \xi$ for each $i \in [k]$. 

Finally, recall that we required 
$N \ge \frac{1}{\xi^2} (\log \frac 1\xi)^{2\ell+1} \zeta(k, 2k-2, \phi, 1) \cdot \log \frac{4}{\delta}$, so a possible choice of $\xi$ is 
$
\xi = \frac{(\log N)^{\ell+1}}{N^{1/2}} \big( \zeta(k, 2k-2, \phi, 1) \cdot \log \frac{4}{\delta} \big)^{1/2} . 
$
This is the final upper bound on the estimation error for each weight. 
Moreover, by the definitions
$L =  \lceil k N^{1/2} \rceil$ and 
$N' =  \lceil k N \log N \rceil$ 
in Algorithm~\ref{alg:weight}, the conditions $L \ge \frac{3k}{\xi \eta}$ and $N' \ge \frac{1}{\xi^2} (\log \frac 1\xi)^{2\ell+1} \zeta(k, 2k-2, \phi, 1) \cdot \log \frac{4 L^k}{\delta}$ required above are indeed satisfied. 

The time complexity bound follows from combining the complexity in Corollary~\ref{cor:main-sc} and the discussion before the statement of Theorem~\ref{thm:weight}. 


\subsection{A conjecture on group determinant and the proof of Theorem~\ref{thm:high-noise}}
\label{sec:conj}


Recall that Theorem~\ref{thm:high-noise}(a) is stated with a restriction on the number of components, $k \le 255$. 
In this section, we restate Theorem~\ref{thm:high-noise} in a relaxed form and explain the origin of this condition. 
We start by recalling the notion of a \emph{group determinant}.
Given any finite group $G$ and variables $t=(t_g: g \in G)$, 
the group determinant $F(t)$ is the determinant of the $|G|\times |G|$ matrix $(t_{g \circ h^{-1}})_{g,h\in G}$.
For the symmetric group $\calS_n$, a notable example is the determinant in \eqref{eq:zagier-A} studied by Zagier \cite{Zag92}, which is $F(t)$ evaluated at $t_\sigma= \phi^{\dkt(\sigma,\id)}$ with $\id$ being the identity permutation.
To compute this group determinant, Zagier introduced an intermediate one as follows. 
Fix any positive integer $r$. For each $s \in [r+1]$, define a permutation $\tau_s \in \cS_{r+1}$ by 
\begin{align}
\tau_s(i) \defn
\begin{cases}
i & \text{ if } i \le s ,\\
r+1 & \text{ if } i = s + 1 , \\
i-1 & \text{ if } i \ge s + 2 .  
\end{cases}
\label{eq:def-tau-s}
\end{align}
In other words, $\tau_s$ leaves the first $s$ elements unchanged and inserts the last element right after them.
Define a $(r+1)! \times (r+1)!$ matrix $L$ indexed by $\pi, \sigma \in \cS_{r+1}$ by 
\begin{align}
\tilde L_{\pi, \sigma} \defn 
\begin{cases}
q^s & \text{ if } \pi \circ \sigma^{-1}  = \tau_s , \, s \in [r+1] \\
0 & \text{ otherwise} .
\end{cases}
\label{eq:zagier-L}
\end{align}
As studied in \cite[Theorem 2']{Zag92}, this is another instance of group determinant with $t_\sigma=q^s$ if $\sigma=\tau_s$ and $0$ otherwise.

Our restatement of Theorem~\ref{thm:high-noise} involves the following conjecture on a slight variant of the group determinant \eqref{eq:zagier-L}. 
\begin{conjecture} \label{lem:invert}
Define a $(r+1)! \times (r+1)!$ matrix $L$ indexed by $\pi, \sigma \in \cS_{r+1}$ by 
\begin{align}
L_{\pi, \sigma} \defn 
\begin{cases}
s & \text{ if } \pi \circ \sigma^{-1}  = \tau_s , \, s \in [r+1] \\
0 & \text{ otherwise} .
\end{cases}
\label{eq:def-m}
\end{align}
Then the matrix $L$ is invertible. 
\end{conjecture}

Note that the matrix $L$ is defined similarly to $\tilde L$, except that the nonzero entry $q^s$ in $\tilde L$ is replaced by $s$.  
Theorem~2' of~\cite{Zag92} gives a formula for the determinant of $\tilde L$, which in particular implies that $\tilde L$ is invertible unless $q$ is a root of unity. 
However, the proof technique there based on factorizing $\tilde L$ using group algebra does not seem to apply to the matrix $L$. 

Although we do not have a proof of Conjecture \ref{lem:invert} for an arbitrary integer $r$, for small $r$ the invertibility of $L$ can be verified numerically. 
In fact, since $\tau_s(1) = 1$ for any $s \in [r+1]$, it is not hard to see that $L$ is block-diagonal with $r+1$ blocks of size $r! \times r!$. 
We are able to verify the invertibility of the diagonal blocks up to $r = 8$, where each block is of size $40320 \times 40320$.

As made precise by the next result, it turns out that Theorem~\ref{thm:high-noise} holds for all $k$-component Mallows models provided that Conjecture~\ref{lem:invert} holds for $r$ up to $\log_2 k$.
\begin{theorem}[Restatement of Theorem~\ref{thm:high-noise}]
\label{thm:conj}
Let the class of Mallows $k$-mixtures $\scrM_*$ be defined by \eqref{eq:class}. 
We let $\eps \defn 1 - \phi$ and consider the setting where $n$ is fixed and $\eps \to 0$. 
For $\mopt$
defined by \eqref{eq:mopt}, the following statements hold: 
\begin{enumerate}[leftmargin=*,label={(\alph*)}]
\item
Suppose that Conjecture~\ref{lem:invert} holds for all positive integers $r \le r_0$, and that $k \le 2^{r_0} - 1$. 
Then, for any distinct Mallows mixtures $\cM$ and $\cM'$ in $\scrM_*$, we have
$
\TV( \cM, \cM' ) = \Omega( \eps^{\mopt} ) . 
$

\item
On the other hand, for $n \ge 2 \mopt$, there exist distinct Mallows mixtures $\cM$ and $\cM'$ in $\scrM_*$ for which 
$
\TV( \cM, \cM' ) = O( \eps^{\mopt} ) . 
$
\end{enumerate}
The hidden constants in $\Omega(\cdot)$ and $O(\cdot)$ above may depend on $n$ and $k$.
\end{theorem}

As noted above, since Conjecture~\ref{lem:invert} holds up to $r = 8$, Theorem~\ref{thm:high-noise} indeed follows from Theorem~\ref{thm:conj} for $k \le 2^8 - 1 = 255$.

\subsection{Proof of Theorem~\ref{thm:conj}}

Throughout the proof, we let $\cM = \sum_{i=1}^k \frac 1k M(\pi_i)$ and $\cM' = \sum_{i=1}^k \frac 1k M(\pi'_i)$ for permutations $\pi_1, \dots, \pi_k, \pi_1', \dots, \pi_k' \in \cS_n$. 
The key to this proof is to relate the total variation distance between $\cM$ and $\cM'$ to the comparison moments defined in Section~\ref{sec:comb-moment}, which allows us to leverage Theorem~\ref{thm:min-pair}. 
The two parts of Theorem~\ref{thm:conj} are then established in Sections~\ref{sec:pf-a} and~\ref{sec:pf-b} respectively. 

\subsubsection{Total variation distance between two mixtures}

Write $f_i = f_{M(\pi_i)}$ and $f_i' = f_{M(\pi_i')}$ for the PMFs of $M(\pi_i)$ and $M(\pi_i')$ respectively. 
Then we have 
$$
f_i(\sigma) =  \frac{1}{Z(1-\eps)} (1-\eps)^{\dkt( \sigma, \pi_i )} 
= \frac{1}{Z(1-\eps)} \sum_{\ell=0}^{\dkt( \sigma, \pi_i )} \binom{ \dkt( \sigma, \pi_i ) }{ \ell } (- \eps)^\ell , 
$$
where $Z(1-\eps) \to n!$ as $\eps \to 0$. 
Therefore, the total variation between $\cM$ and $\cM'$ is 
\begin{align*}
&\TV( \cM, \cM' ) 
= \frac 1{2k} \bigg\| \sum_{i=1}^k f_i - \sum_{i=1}^k f'_i \bigg\|_1 
= \frac{1}{2k} \sum_{\sigma \in \cS_n} \bigg| \sum_{i=1}^k f_i(\sigma) - \sum_{i=1}^k f'_i(\sigma)  \bigg|\\
&= \frac{1}{2k Z(1-\eps)} \sum_{\sigma \in \cS_n} \Bigg| \sum_{i=1}^k \sum_{\ell=0}^{\dkt( \sigma, \pi_i )} \binom{ \dkt( \sigma, \pi_i ) }{ \ell } (- \eps)^\ell - \sum_{i=1}^k \sum_{\ell=0}^{\dkt( \sigma, \pi_i' )} \binom{ \dkt( \sigma, \pi_i' ) }{ \ell } (- \eps)^\ell  \Bigg| \\
&= \frac{1}{2k Z(1-\eps)} \sum_{\sigma \in \cS_n} \Bigg| \sum_{\ell=0}^{n(n-1)/2} \sum_{i=1}^k  \bigg[ \binom{ \dkt( \sigma, \pi_i ) }{ \ell } - \binom{ \dkt( \sigma, \pi_i' ) }{ \ell } \bigg] (- \eps)^\ell  \Bigg| 
\end{align*} 
with the convention that $\binom{d}{\ell} \defn \frac{ d (d-1) \cdots (d-\ell+1) }{ \ell! } = 0$ if $d < \ell$. 
Then
$ \TV( \cM, \cM' ) 
= O( \eps^{m+1} ) $
if and only if the coefficient of $\eps^\ell$ vanishes in the above formula for all $\ell \in [m]$ and all $\sigma \in \cS_n$, that is, 
\begin{align*}
\sum_{i=1}^k \binom{ \dkt( \sigma, \pi_i ) }{\ell} = \sum_{i=1}^k \binom{ \dkt( \sigma, \pi'_i ) }{ \ell }
\quad \text{ for all } \ell \in [m], \, \sigma \in \cS_n .
\end{align*}
By a simple inductive argument, we see that this is equivalent to 
\begin{align}
\sum_{i=1}^k \dkt( \sigma, \pi_i )^\ell = \sum_{i=1}^k \dkt( \sigma, \pi'_i )^\ell 
\quad \text{ for all } \ell \in [m], \, \sigma \in \cS_n .
\label{eq:powers}
\end{align}
Rewriting \eqref{eq:powers} in terms of expectations, we have proved the following result. 

\begin{proposition}[Distance moment matching]
\label{prop:tv}
Consider random permutations $\pi \sim  \frac 1k \sum_{i=1}^k \delta_{\pi_i} $ and  $\pi' \sim \frac 1k \sum_{i=1}^k \delta_{\pi_i'} $. 
Under the conditions of Theorem~\ref{thm:conj}, we have: 
\begin{itemize}[leftmargin=*]
\item
the order of $\TV( \cM, \cM')$ in $\eps$ is a positive integer; 

\item
$\TV( \cM, \cM') = O(\eps^{m+1})$ if and only if 
$
\E[ \dkt( \sigma, \pi )^\ell ] = \E[ \dkt( \sigma, \pi' )^\ell ]
$
for all 
$
\ell \in [m], \, \sigma \in \cS_n . 
$
\end{itemize}
We refer to $\E[ \dkt( \sigma, \pi )^\ell ]$ as the $\ell$th order \emph{distance moment} of $\pi$ at $\sigma$.
\end{proposition}

\subsubsection{Equivalence of distance moments and comparison moments}

By the above proposition, the order of $\TV( \cM, \cM') $ in $\eps$ is determined by how many distance moments are matched between the two $k$-mixtures $\pi$ and $\pi'$. 
To characterize how this number of matched moments depends on $k$, it suffices to relate distance moments to comparison moments defined in Section~\ref{sec:comb-moment}, because we have studied identifying $k$-mixtures from comparison moments in Theorem~\ref{thm:min-pair}. 

We first set up the notation. 
Recall that $X^\pi_{i,j} \defn \1\{\pi(i) < \pi(j)\}$. 
Viewing each $X^\pi_{i,j}$ as a variable, we use $\bP_{m}(X^\pi)$ to denote any polynomial in $\{X^\pi_{i, j}\}_{i \ne j}$ of degree \emph{at most} $m$, that is, any polynomial of the form 
$$
\sum_{\substack{\text{sets of pairs of distinct indices} \\  (i_1, j_1), \dots, (i_m, j_m) \in [n]^2}} 
X^\pi_{i_1, j_1} \cdots X^\pi_{i_m, j_m} . 
$$
When we need to explicitly specify the variables, we also write $\bP_m(X^\pi_{i_1, j_1}, \dots, X^\pi_{i_\ell, j_\ell})$. For example, the polynomial $X_{1,2} X_{2,3} + X_{1,2} X_{4,5} + X_{3,5}$ can be denoted by $\bP_2( X_{1,2}, X_{2,3}, X_{3,5}, X_{4,5} )$.

In addition, the definition of the Kendall tau distance can be written as 
\begin{align}
\dkt(\sigma, \pi) = \sum_{(i,j): \sigma(i) > \sigma(j)} X^\pi_{i, j} . 
\label{eq:dkt-def}
\end{align} 
Therefore, we have 
\begin{align}
\dkt( \sigma, \pi )^m = \sum_{\substack{(i_1, j_1), \dots, (i_m, j_m) : \\ \sigma(i_1) > \sigma(j_1), \, \dots \, , \, \sigma(i_m) > \sigma(j_m)}} X^\pi_{i_1, j_1} \cdots X^\pi_{i_m, j_m} 
= \bP_{m}(X^\pi) . 
\label{eq:dkt-m}
\end{align}

Moreover, consider real-valued functions $a_1(\pi), \dots, a_{n_1}(\pi)$ and $b(\pi)$ of $\pi \in \cS_n$. 
We say that $b(\pi)$ can be \emph{linearly constructed} from the list of functions $\{ a_1(\pi), \dots, a_{n_1}(\pi) \}$, 
if there exist real coefficients $c_1, \dots, c_{n_1}$ that do not depend on $\pi$, such that $\sum_{i=1}^{n_1} c_i a_i(\pi) = b(\pi)$. 
If every function in $\{b_1(\pi), \dots, b_{n_2}(\pi)\}$ can be linearly constructed from $\{ a_1(\pi), \dots, a_{n_1}(\pi) \}$, we write 
$$
\{ a_1(\pi), \dots, a_{n_1}(\pi) \} \Longrightarrow \{ b_1(\pi), \dots, b_{n_2}(\pi) \} . 
$$ 

By \eqref{eq:dkt-m}, it is clear that $\dkt(\sigma, \pi)^m$ can be linearly constructed from the list 
$
\big\{ X^\pi_{i_1, j_1} \cdots X^\pi_{i_m, j_m} : i_\ell, j_\ell \in [n], \, i_\ell \ne j_\ell, \,  \ell \in [m] \big\} 
$
for any $\sigma \in \cS_n$. 
Therefore, we have 
\begin{align}
\big\{ X^\pi_{i_1, j_1} \cdots X^\pi_{i_m, j_m} : i_\ell, j_\ell \in [n], \, i_\ell \ne j_\ell, \,  \ell \in [m] \big\} 
\Longrightarrow
\big\{ \dkt(\sigma, \pi)^\ell : \sigma \in \cS_n, \, \ell \in [m] \big\} . 
\label{eq:comb-dist}
\end{align}
Note that we do not explicitly have polynomials of degree less than $m$ in the list on the LHS of \eqref{eq:comb-dist}. This is because, using the fact $(X^\pi_{i,j})^r = X^\pi_{i,j}$ for any $r \in \naturals$, we can write any polynomial of degree $\ell \ge 1$ formally as a polynomial of degree $m \ge \ell$ by appending redundant variables $(X^\pi_{i,j})^{m-\ell}$. 
The next lemma states the converse of \eqref{eq:comb-dist}, whose proof is deferred to Section~\ref{sec:lin-sys}. 

\begin{lemma}
\label{lem:lin-sys}
Suppose that Conjecture~\ref{lem:invert} holds for all positive integers $r \le r_0$. 
Then, for any positive integer $m \le r_0$, we have 
\begin{align}
\big\{ \dkt(\sigma, \pi)^\ell : \sigma \in \cS_n, \, \ell \in [m] \big\}
\Longrightarrow
\big\{ X^\pi_{i_1, j_1} \cdots X^\pi_{i_m, j_m} : i_\ell, j_\ell \in [n], \, i_\ell \ne j_\ell, \,  \ell \in [m] \big\} . 
\label{eq:dist-comb}
\end{align}
In other words, all polynomials in $\{X^\pi_{i,j}\}_{i \ne j}$ of degree at most $m$ can be linearly constructed from special polynomials $ \dkt(\sigma, \pi)^\ell $ of degree $\ell \le m$ where $\sigma \in \cS_n$. 
\end{lemma}

From \eqref{eq:comb-dist} or \eqref{eq:dist-comb}, we easily obtain the following equivalence of distance moments and comparison moments.  

\begin{proposition}[Equivalence of distance and comparison moments]
\label{prop:equiv}
Suppose that Conjecture~\ref{lem:invert} holds for all positive integers $r \le r_0$. 
Consider a random permutation $\pi \sim  \frac 1k \sum_{i=1}^k \delta_{\pi_i}$, where $\pi_1, \dots, \pi_k$ are unknown permutations in $\cS_n$. 
For $m \in \naturals$, consider the list of distant moments
\begin{align}
\big\{ \E[ \dkt(\sigma, \pi)^\ell ] : \sigma \in \cS_n, \, \ell \in [m] \big\}
\label{eq:dist-list}
\end{align}
and the list of comparison moments (Definition~\ref{def:comb-moment})
\begin{align}
\big\{ \moment(\pi, \cI) : \cI = \big( (i_1, j_1), \dots, (i_m, j_m) \big) , \, i_\ell, j_\ell \in [n], \, i_\ell < j_\ell, \,  \ell \in [m] \big\} . 
\label{eq:comb-list}
\end{align}
Then \eqref{eq:dist-list} is a deterministic linear function of \eqref{eq:comb-list}, regardless of the unknown permutations $\pi_1, \dots, \pi_k$. Conversely, \eqref{eq:comb-list} is a deterministic linear function of \eqref{eq:dist-list} provided that $m \le r_0$. 
\end{proposition}

\begin{proof}
Crucially, the constructions in \eqref{eq:comb-dist} and \eqref{eq:dist-comb} are linear and do not depend on $\pi$. Therefore, taking the expectation with respect to $\pi \sim  \frac 1k \sum_{i=1}^k \delta_{\pi_i} $, we see that the list of distance moments \eqref{eq:dist-list} 
and the list 
\begin{align}
\big\{ \E[ X^\pi_{i_1, j_1} \cdots X^\pi_{i_m, j_m} ] : i_\ell, j_\ell \in [n], \, i_\ell \ne j_\ell, \,  \ell \in [m] \big\}
\label{eq:inter-list}
\end{align}
are linear functions of each other, independent of $\pi_1, \dots, \pi_k$. 
By Definition~\ref{def:comb-moment} and \eqref{eq:switch}, the list of comparison moments \eqref{eq:comb-list} and the list \eqref{eq:inter-list} both contain all possible expectations of products of $m$ variables, and are therefore linear functions of each other. 
Finally, note that \eqref{eq:dist-comb} holds for $r \le r_0$, so the converse direction holds under the same condition. 
\end{proof}

Having established the equivalence of the two types of moments, we are ready to prove the two parts of Theorem~\ref{thm:conj}. 

\subsubsection{Proof of part~(a)}
\label{sec:pf-a}
Suppose that for the two Mallows mixtures $\cM = \sum_{i=1}^k \frac 1k M(\pi_i)$ and $\cM' = \sum_{i=1}^k \frac 1k M(\pi'_i)$, we have 
$
\TV( \cM, \cM' ) = O( \eps^{m+1} ) . 
$
Considering random permutations $\pi \sim  \frac 1k \sum_{i=1}^k \delta_{\pi_i} $ and  $\pi' \sim \frac 1k \sum_{i=1}^k \delta_{\pi_i'} $, we obtain from Proposition~\ref{prop:tv} that 
$
\E[ \dkt( \sigma, \pi )^\ell ] = \E[ \dkt( \sigma, \pi' )^\ell ]
$
for all 
$
\ell \in [m], \, \sigma \in \cS_n . 
$
As $k \le 2^{r_0}-1$ so that $m = \lfloor \log_2 k \rfloor + 1 \le r_0$, Proposition~\ref{prop:equiv} yields that $\moment(\pi, \cI) = \moment(\pi', \cI)$ for any tuple $\cI$ of pairs of distinct indices $(i_1, j_1), \dots, (i_m, j_m) \in [n]^2$. That is, the group of pairwise comparisons on any $\cI$ coincides for the two mixtures $\pi$ and $\pi'$. 
Since $m = \lfloor \log_2 k \rfloor + 1$, the algorithm from part~(a) of Theorem~\ref{thm:min-pair} can recover the noiseless mixture of permutations from groups of $m$ pairwise comparisons. Consequently, we must have $\frac 1k \sum_{i=1}^k \delta_{\pi_i} = \frac 1k \sum_{i=1}^k \delta_{\pi_i'} $, and therefore $\cM = \cM'$. 

Since the the order of $\TV(\cM, \cM')$ in $\eps$ is necessarily an integer according to Proposition~\ref{prop:tv}, we conclude that, if $\cM \ne \cM'$, then $\TV(\cM, \cM') = \Omega(\eps^m)$.

\subsubsection{Proof of part~(b)}
\label{sec:pf-b}
By part~(b) of Theorem~\ref{thm:min-pair}, there exist distinct mixtures $\frac 1k \sum_{i=1}^k \delta_{\pi_i}$ and $\frac 1k \sum_{i=1}^k \delta_{\pi_i'}$ of permutations in $\cS_n$, that cannot be distinguished using any groups of $m-1$ pairwise comparisons. 
Let $\pi$ and $\pi'$ denote random permutations from the above two mixtures respectively. 
Then we have $\moment(\pi, \cI) = \moment(\pi', \cI)$ for any tuple $\cI$ of $m-1$ pairs of distinct indices in $[n]$. Hence Proposition~\ref{prop:equiv} implies that $
\E[ \dkt( \sigma, \pi )^\ell ] = \E[ \dkt( \sigma, \pi' )^\ell ]
$
for all 
$
\ell \in [m-1], \, \sigma \in \cS_n . 
$
It then follows from Proposition~\ref{prop:tv} that $\TV(\cM, \cM') = O(\eps^m)$ for $\cM = \sum_{i=1}^k \frac 1k M(\pi_i)$ and $\cM' = \sum_{i=1}^k \frac 1k M(\pi'_i)$.

\subsection{Proof of Lemma~\ref{lem:lin-sys}} 
\label{sec:lin-sys}

Throughout this section, we suppose that Conjecture~\ref{lem:invert} holds for all positive integers $r \le r_0$. 

\subsubsection{Preliminary lemmas}

We establish the following lemmas before proving Lemma~\ref{lem:lin-sys}. 
The following result gives conditions under which the polynomial $X^\pi_{i_1, j_1} \cdots X^\pi_{i_m, j_m}$ has degree strictly less than $m$.
For example, $X^\pi_{ab}X^\pi_{ba} \equiv 0$ has degree 0 and 
$X^\pi_{ab}X^\pi_{bc}X^\pi_{ca} = X^\pi_{ab}X^\pi_{bc}$ has degree 2. 
(In the language of the next lemma, the graph $G$ corresponds to a double edge and a triangle respectively).


\begin{lemma}
\label{lem:cycle}
Fix a monomial $X^\pi_{i_1, j_1} \cdots X^\pi_{i_m, j_m}$, where $(i_1, j_1), \dots, (i_m, j_m)$ are pairs of distinct indices in $[n]$.  
Consider the undirected multigraph $G$ with vertex set $\{i_1, j_1, \dots, i_m, j_m\}$ and edge set $\{(i_1, j_1), \dots, (i_m, j_m)\}$. 
If $G$ contains a cycle, then $X^\pi_{i_1, j_1} \cdots X^\pi_{i_m, j_m} = \bP_{m-1}(X^\pi)$.  
\end{lemma}

\begin{proof}
Up to a relabeling, we assume without loss of generality that the cycle is composed of undirected edges $(i_1, j_1), \dots, (i_\ell, j_\ell)$. 
Let the ordered vertex sequence of the cycle be $(v_1, v_2, \dots, v_{\ell}, v_1)$. 
In particular, we have $\{i_1, j_1, \dots, i_\ell, j_\ell\} = \{v_1, v_2, \dots, v_{\ell}\}$ as sets, and each pair $(i_r, j_r)$ is equal to some $(v_s, v_{s+1})$ or $(v_{s+1}, v_s)$. 
Hence we have that either $X^\pi_{i_r, j_r} = X^\pi_{v_s, v_{s+1}}$ or $X^\pi_{i_r, j_r} = 1-X^\pi_{v_s, v_{s+1}}$. 
It then follows that 
$$
X^\pi_{i_1, j_1} \cdots X^\pi_{i_\ell, j_\ell} = X^\pi_{v_1, v_2} \cdots X^\pi_{v_{\ell-1}, v_{\ell}} X^\pi_{v_{\ell}, v_1} + \bP_{\ell-1}(X^\pi). 
$$
In addition, if $X^\pi_{v_1, v_2} \cdots X^\pi_{v_{\ell-1}, v_{\ell}} = 1$, then $\pi(v_1) < \pi(v_2) < \cdots < \pi(v_{\ell})$, so we must have $X^\pi_{v_{\ell}, v_1} = 0$. 
As a result, it holds that $X^\pi_{v_1, v_2} \cdots X^\pi_{v_{\ell-1}, v_{\ell}} X^\pi_{v_{\ell}, v_1} = 0$ and $X^\pi_{i_1, j_1} \cdots X^\pi_{i_\ell, j_\ell} = \bP_{\ell-1}(X^\pi)$. 
We conclude that $X^\pi_{i_1, j_1} \cdots X^\pi_{i_m, j_m} = \bP_{m-1}(X^\pi)$. 
\end{proof}

\begin{lemma}
\label{lem:tree}
With the same notation as in Lemma~\ref{lem:cycle}, if the graph $G$ is a tree, then there exist bijections $\tau_1, \dots, \tau_\beta : [m+1] \to \{i_1, j_1, \dots, i_m, j_m\}$ such that 
\begin{align}
X^\pi_{i_1, j_1} \cdots X^\pi_{i_m, j_m} = \sum_{\alpha = 1}^\beta X^\pi_{\tau_\alpha(1), \tau_\alpha(2)} \cdots X^\pi_{\tau_\alpha(m), \tau_\alpha(m+1)} . 
\label{eq:lin-ext}
\end{align}
\end{lemma}

\begin{proof}
First, since $G$ is a tree with $m$ edges, the cardinality of its vertex set $V \defn \{i_1, j_1, \dots, i_m, j_m\}$ is exactly $m+1$. This justifies the possibility of defining a bijection from $[m+1]$ to $V$. 

Let $\tilde{G}$ denote the directed tree with vertex set $V$ and edge set $\{(i_1, j_1), \dots, (i_m, j_m)\}$; that is, $\tilde{G}$ is the directed version of $G$. It is well known that the reachability\footnote{Reachability refers to the existence of a directed path from one vertex to another in a directed graph.} relations between vertices of any directed tree form a partial order of the vertices. Let $\sP$ denote this partial order for $\tilde{G}$. 
Furthermore, let $\tau_1^{-1}, \dots, \tau_\beta^{-1} : V \to [m+1]$ denote all possible linear extensions\footnote{A linear extension of a partial order is a total order that is compatible with the partial order. Here, we identify each total order on a finite set $V$ with a bijection from $V$ to $[|V|]$.} of the partial order $\sP$. 
That is, each $\tau_\alpha$ is a bijection such that $\tau_\alpha^{-1}(i_\ell) < \tau_\alpha^{-1}(j_\ell)$, where $\alpha \in [\beta]$. 

Furthermore, recall that the permutation $\pi$ induces a total order on $V$, which we denote by $\pi\|_V$ as before. Note that the monomial 
$$
X^\pi_{i_1, j_1} \cdots X^\pi_{i_m, j_m} = \1 \big\{ \pi(i_1) < \pi(j_1) , \cdots , \pi(i_m) < \pi(j_m) \big\} 
$$
is equal to $1$ if and only the total order $\pi\|_V$ is compatible with $\sP$. 

On the other hand, we observe that 
\begin{align*}
X^\pi_{\tau_\alpha(1), \tau_\alpha(2)} \cdots X^\pi_{\tau_\alpha(m), \tau_\alpha(m+1)} 
&= \1 \big\{ \pi \big( \tau_\alpha (1) \big) < \pi \big( \tau_\alpha (2) \big) < \cdots < \pi \big( \tau_\alpha (m+1) \big) \big\} \\
&= \1 \big\{ \pi\|_V \circ  \tau_\alpha (1)  < \pi\|_V \circ  \tau_\alpha (2)  < \cdots < \pi\|_V \circ  \tau_\alpha (m+1) \big\} . 
\end{align*}
Since each $\pi\|_V \circ  \tau_\alpha$ is a permutation on $[m+1]$, the above indicator is equal to $1$ if and only if the two total orders $\pi\|_V$ and $\tau_\alpha^{-1}$ coincide. 

Combining the above pieces, we see that \eqref{eq:lin-ext} is equivalent to stating that
$$
\1 \big\{ \pi\|_V \text{ is compatible with } \sP \big\} = \sum_{\alpha = 1}^\beta \1 \big\{ \pi\|_V = \tau_\alpha^{-1} \big\} ,
$$
which is tautologically true, as $\tau_1^{-1}, \dots, \tau_\beta^{-1}$ are all the linear extensions of $\sP$ by definition. 
\end{proof}

For $n \in \naturals$ and $r \in [n]$, we use the notation 
$$
[n]_r \defn \{ (i_1, \dots, i_r) \in [n]^r : i_1, \dots, i_r \text{ are distinct} \}.
$$

\begin{lemma} \label{lem:one-factor}
For any fixed $\pi \in \cS_n$ and a positive integer $r \le (n-1) \land r_0$, we have 
\begin{align*}
\Big\{  \prod_{s=1}^{r-1} X^\pi_{i_s, i_{s+1}} \cdot \Big( \sum_{t=1}^r X^\pi_{i_t, i_{r+1}}  \Big) : (i_1, \dots, i_{r+1}) \in [n]_{r+1} \Big\} 
\Longrightarrow
\Big\{ \prod_{s=1}^{r} X^\pi_{i_s, i_{s+1}}  : ( i_1, \dots, i_{r+1} ) \in [n]_{r+1} \Big\} .
\end{align*}
\end{lemma}

\begin{proof}
First, we note that for any $t \in [r]$, 
\begin{align*}
&\prod_{s=1}^{r-1} X^\pi_{i_s, i_{s+1}} \cdot  X^\pi_{i_t, i_{r+1}} 
=  \prod_{s=1}^{r-1} \1\{ \pi(i_s) < \pi( i_{s+1}) \} \cdot  \1\{ \pi(i_t) < \pi( i_{r+1}) \}  \\
&= \1 \big\{ \pi(i_1) < \pi(i_2) < \cdots < \pi(i_r), \, \pi(i_t) < \pi( i_{r+1}) \big\} \\
&= \sum_{s=t}^r  \1 \big\{ \pi(i_1) < \pi(i_2) < \cdots < \pi(i_s) < \pi(i_{r+1}) < \pi(i_{s+1}) < \pi(i_{s+2}) < \cdots < \pi( i_r) \big\}  . 
\end{align*}
The last equality holds because if $\pi(i_1) < \pi(i_2) < \cdots < \pi(i_r)$ and $\pi(i_t) < \pi( i_{r+1})$, then the index $i_{r+1}$ can possibly be placed by $\pi$ in any of the $r-t+1$ locations after $i_t$. 
Summing the above equality over $t \in [r]$ yields 
\begin{align*}
& \prod_{s=1}^{r-1} X^\pi_{i_s, i_{s+1}} \cdot \Big( \sum_{t=1}^r X^\pi_{i_t, i_{r+1}}  \Big)  \\
&= \sum_{t=1}^r \sum_{s=t}^r  \1 \big\{ \pi(i_1) < \pi(i_2) < \cdots < \pi(i_s) < \pi(i_{r+1}) < \pi(i_{s+1}) < \pi(i_{s+2}) < \cdots < \pi( i_r) \big\} \\
&= \sum_{s=1}^r s \cdot \1 \big\{ \pi(i_1) < \pi(i_2) < \cdots < \pi(i_s) < \pi(i_{r+1}) < \pi(i_{s+1}) < \pi(i_{s+2}) < \cdots < \pi( i_r) \big\} . 
\end{align*}
On the other hand, we have 
\begin{align*}
\prod_{s=1}^{r} X^\pi_{i_s, i_{s+1}} 
= \1 \big\{ \pi(i_1) < \pi(i_2) < \cdots < \pi(i_{r+1}) \big\} . 
\end{align*}
Hence the linear construction that we need to establish is equivalent to 
\begin{align}
&\Big\{  \sum_{s=1}^r s \cdot \1 \big\{ \pi(i_1) < \pi(i_2) < \cdots < \pi(i_s) < \pi(i_{r+1}) < \pi(i_{s+1}) < \pi(i_{s+2}) < \cdots < \pi( i_r) \big\} : \notag \\
&\qquad \qquad \qquad \qquad \qquad \qquad \qquad \qquad \qquad \qquad \qquad \qquad \qquad \qquad \qquad \qquad (i_1, \dots, i_{r+1}) \in [n]_{r+1} \Big\} \notag \\
&\Longrightarrow
\Big\{ \1 \big\{ \pi(i_1) < \pi(i_2) < \cdots < \pi(i_{r+1}) \big\}   : (i_1, \dots, i_{r+1}) \in [n]_{r+1} \Big\} .
\label{eq:lin-1}
\end{align}

Note that the sets on the left and right hand sides of \eqref{eq:lin-1} are indexed by distinct tuples $i_1, \dots, i_{r+1}$. Next we fix a set of $r+1$ indices in $[n]$, but allow their order to vary. That is, let us fix distinct indices $i_1, \dots, i_{r+1} \in [n]$, and consider $i_{\sigma(1)}, \dots, i_{\sigma(r+1)}$ where $\sigma$ is any permutation in $\cS_{r+1}$. To show \eqref{eq:lin-1}, then it suffices to establish the linear construction 
\begin{align}
&\Big\{  \sum_{s=1}^r s \cdot \1 \big\{ \pi(i_{\sigma(1)}) < \pi(i_{\sigma(2)}) < \cdots < \pi(i_{\sigma(s)}) < \pi(i_{\sigma(r+1)}) \notag \\
&\qquad \qquad \qquad \qquad \qquad \qquad \qquad \qquad < \pi(i_{\sigma(s+1)}) < \pi(i_{\sigma(s+2)}) < \cdots < \pi( i_{\sigma(r)}) \big\} : \sigma \in \cS_{r+1} \Big\} \notag \\
&\Longrightarrow
\Big\{ \1 \big\{ \pi(i_{\sigma(1)}) < \pi(i_{\sigma(2)}) < \cdots < \pi(i_{\sigma(r+1)}) \big\}   : \sigma \in \cS_{r+1} \Big\} .
\label{eq:lin-2}
\end{align}
To prove \eqref{eq:lin-2}, we define a vector $v 
\in \{0, 1\}^{(r+1)!}$, indexed by $\sigma \in \cS_{r+1}$, by
$$
v_{\sigma} \defn \1 \big\{ \pi(i_{\sigma(1)}) < \pi(i_{\sigma(2)}) < \cdots < \pi( i_{\sigma(r+1)}) \big\} . 
$$
With $\tau_s$ defined by \eqref{eq:def-tau-s}, we see that \eqref{eq:lin-2} is equivalent to  
\begin{align}
\Big\{  \sum_{s=1}^r s \cdot v_{\sigma \circ \tau_s} : \sigma \in \cS_{r+1} \Big\} 
\Longrightarrow \Big\{ v_\sigma : \sigma \in \cS_{r+1} \Big\} . 
\label{eq:lin-3}
\end{align}

Finally, let $L$ be defined by \eqref{eq:def-m}, and let $L'$ be the matrix indexed by $\pi, \sigma \in \cS_{r+1}$ defined by $L'_{\pi, \sigma} = L_{\pi^{-1}, \sigma^{-1}}$. 
Then we have  
\begin{align*}
\sum_{s=1}^r s \cdot v_{\sigma \circ \tau_s} = \sum_{\sigma' \in \cS_{r+1}} s \cdot v_{\sigma'} \1\{ \sigma^{-1} \circ \sigma' = \tau_s \} = \sum_{\sigma' \in \cS_{r+1}} L_{\sigma^{-1}, (\sigma')^{-1}} v_{\sigma'} = (L' v)_{\sigma} , 
\end{align*}
Therefore, \eqref{eq:lin-3} holds if $L'$ is invertible. 
Moreover, $L'$ is invertible if and only if $L$ is invertible, because one can be obtained from the other by shuffling the columns and rows. 
Finally, applying Conjecture~\ref{lem:invert} finishes the proof. (In fact, this is the only step of the entire proof where the conjecture is used.)
\end{proof}

For any permutation $\pi \in \cS_n$ and distinct indices $i_1, i_2, \dots, i_{\ell+1} \in [n]$, we define 
\begin{align}
Q^\pi_{i_1, \dots, i_{\ell+1}} &\defn X^\pi_{i_1, i_2} (X^\pi_{i_1, i_3} + X^\pi_{i_2, i_3}) (X^\pi_{i_1, i_4} + X^\pi_{i_2, i_4} + X^\pi_{i_3, i_4}) \cdots (X^\pi_{i_1, i_{\ell+1}} + X^\pi_{i_2, i_{\ell+1}} + \cdots + X^\pi_{i_\ell, i_{\ell+1}}) \notag \\
&= \prod_{s=1}^{\ell} \Big( \sum_{t=1}^s X^\pi_{i_t, i_{s+1}} \Big) ,
\label{eq:def-q}
\end{align}
which is a degree-$\ell$ polynomial in $X^\pi_{i,j}$'s. 

\begin{lemma} \label{lem:qx}
For any fixed $\pi \in \cS_n$ and a positive integer $\ell \le (n-1) \land r_0$, we have 
\begin{align*}
\big\{  Q^\pi_{i_1, \dots, i_{\ell+1}} : (i_1, \dots, i_{\ell+1}) \in [n]_{\ell+1} \big\}
\Longrightarrow
\big\{ X^\pi_{i_1, i_2} X^\pi_{i_2, i_3} \cdots X^\pi_{i_\ell, i_{\ell+1}}  : (i_1, \dots, i_{\ell+1}) \in [n]_{\ell+1}  \big\} .
\end{align*}
\end{lemma}

\begin{proof}
The construction in Lemma~\ref{lem:one-factor} is linear and thus can be applied even if we multiply every polynomial by a common factor $\prod_{s=r+1}^\ell \big( \sum_{t=1}^s X^\pi_{i_t, i_{s+1}} \big)$. 
Therefore, we obtain, for each $r=1,\ldots,\ell+1$,
\begin{align}
&\Big\{  \prod_{s=1}^{r-1} X^\pi_{i_s, i_{s+1}} \cdot \prod_{s=r}^\ell \Big( \sum_{t=1}^s X^\pi_{i_t, i_{s+1}} \Big) : (i_1, \dots, i_{\ell+1}) \in [n]_{\ell+1}  \Big\} \notag \\
&\Longrightarrow
\Big\{  \prod_{s=1}^{r} X^\pi_{i_s, i_{s+1}} \cdot \prod_{s=r+1}^\ell \Big( \sum_{t=1}^s X^\pi_{i_t, i_{s+1}} \Big) : (i_1, \dots, i_{\ell+1}) \in [n]_{\ell+1}  \Big\} . 
\label{eq:inter-step}
\end{align}
Here a product is understood as $1$ if the bottom index exceeds the top, by convention. 
Note that the quantity $\prod_{s=1}^{r-1} X^\pi_{i_s, i_{s+1}} \cdot \prod_{s=r}^\ell \big( \sum_{t=1}^s X^\pi_{i_t, i_{s+1}} \big)$ is equal to $Q^\pi_{i_1, \dots, i_{\ell+1}}$ for $r = 1$ and is equal to $X^\pi_{i_1, i_2} X^\pi_{i_2, i_3} \cdots X^\pi_{i_\ell, i_{\ell+1}}$ for $r = \ell+1$. 
As a result, applying \eqref{eq:inter-step} iteratively with $r = 1, 2, \dots, \ell$ yields the lemma.  
\end{proof}

\subsubsection{The main proof} 

We are ready to prove Lemma~\ref{lem:lin-sys}. 
Let us first establish the statement for $m = 1$: 
\begin{align}
\{ \dkt(\sigma, \pi) : \sigma \in \cS_n \} \Longrightarrow \{ X^\pi_{i_1, j_1} : i_1, j_1 \in [n], \, i_1 \ne j_1 \} . 
\label{eq:m-one}
\end{align}
Toward this end, we choose $\overline{\sigma}, \underline{\sigma} \in \cS_n$ such that 
$\overline{\sigma}(i_1) = 1$, 
$\overline{\sigma}(j_1) = 2$, 
$\underline{\sigma}(i_1) = 2$, 
$\underline{\sigma}(j_1) = 1$, 
and $\overline{\sigma}(r) = \underline{\sigma}(r)$ for $r \ne i_1, j_1$. 
Then it follows from \eqref{eq:dkt-def} that 
$$
\dkt(\underline{\sigma}, \pi) - \dkt(\overline{\sigma}, \pi) 
= \sum_{(i, j): \underline{\sigma}(i) > \underline{\sigma}(j)} X^\pi_{i, j} 
- \sum_{(i, j): \overline{\sigma}(i) > \overline{\sigma}(j)} X^\pi_{i, j} 
= X^\pi_{i_1, j_1} - X^\pi_{j_1, i_1} 
= 2 X^\pi_{i_1, j_1} - 1. 
$$
Therefore, \eqref{eq:m-one} indeed holds.

With the base case $m = 1$ established, we can prove the lemma using an induction on $m$. 
Therefore, it suffices to show that 
\begin{align}
\cP &\defn \big\{ \dkt(\sigma, \pi)^\ell : \sigma \in \cS_n, \, \ell \in [m] \big\} \cup \big\{ X^\pi_{i_1, j_1} \cdots X^\pi_{i_{m-1}, j_{m-1}} : i_\ell, j_\ell \in [n], \, i_\ell \ne j_\ell, \,  \ell \in [m-1] \big\} \notag \\
&\Longrightarrow
\big\{ X^\pi_{i_1, j_1} \cdots X^\pi_{i_m, j_m} : i_\ell, j_\ell \in [n], \, i_\ell \ne j_\ell, \,  \ell \in [m] \big\} , \label{eq:induction}
\end{align}
that is, all degree-$m$ polynomials in $X^\pi_{i,j}$'s can be linearly constructed from $ \dkt(\sigma, \pi)^\ell $ where $\ell \le m$ together with degree-$\ell$ polynomials in $X^\pi_{i,j}$'s where $\ell \le m-1$. 



The proof of \eqref{eq:induction} is split into several steps below. 
For the proof, we recall the assumption that $\ell \le m \le r_0$, and continue to use the notation
$
Q^\pi_{i_1, \dots, i_{\ell+1}} 
$
defined by \eqref{eq:def-q}. 
Moreover, we say that the indices $i_1, \dots, i_\ell$ \emph{appear consecutively} in $\pi$ if 
$$
\pi(i_1) = \pi(i_2) - 1 = \pi(i_2) - 2 = \cdots = \pi(i_\ell) - \ell + 1 . 
$$
For example, the indices $5, 2, 3$ appear consecutively in the permutation $(4, 6, 5, 2, 3, 1)$. 

\paragraph{Step 1.}
We claim that for any fixed $\ell \in \{ 0, 1, \dots, m \}$, any permutation $\sigma \in \cS_n$, and indices $i_1, i_2, \dots, i_{\ell+1}$ appearing consecutively in $\sigma$, the polynomial 
\begin{align}
\dkt( \sigma, \pi )^{m-\ell} 
Q^\pi_{i_1, \dots, i_{\ell+1}} 
\label{eq:swap}
\end{align}
can be linearly constructed from $\cP$. 

Toward this end, we proceed by induction on $\ell$. 
The base case $\ell = 0$ is trivial. 
Now assume that the claim holds for a fixed $\ell \in \{ 0, 1, \dots, m -1 \}$. 
Consider any $\overline{\sigma}, \underline{\sigma} \in \cS_n$ such that: 
\begin{itemize}[leftmargin=*]
\item
$i_1, i_2, \dots, i_{\ell+2}$ appear consecutively in $\overline{\sigma}$;

\item
$\underline{\sigma}(i_r) = \overline{\sigma}(i_r) + 1$ for $r \in [\ell+1]$ and $\underline{\sigma}(i_{\ell+2}) = \overline{\sigma}(i_1)$; that is, $\underline{\sigma}$ is obtained from $\overline{\sigma}$ by inserting $i_{\ell + 2}$ to the position right before $i_1$ (and shifting $i_1, \dots, i_{\ell+1}$ to the right accordingly).  
\end{itemize}
Since $i_1, \dots, i_{\ell+1}$ appear consecutively in both $\overline{\sigma}$ and $\underline{\sigma}$, 
the induction hypothesis implies that the polynomial \eqref{eq:swap} with either $\sigma = \overline{\sigma}$ or $\sigma = \underline{\sigma}$ can be linearly constructed from $\cP$. 

Moreover, $\overline{\sigma}$ and $\underline{\sigma}$ only differ within the labels $i_1, \dots, i_{\ell+2}$. 
Hence, by definition \eqref{eq:dkt-def}, 
\begin{align*}
\dkt(\underline{\sigma}, \pi) &= \dkt(\overline{\sigma}, \pi) - X^\pi_{i_{\ell+2}, i_1} - X^\pi_{i_{\ell+2}, i_2} - \cdots - X^\pi_{i_{\ell+2}, i_{\ell+1}} + X^\pi_{i_1, i_{\ell+2}} + X^\pi_{i_2, i_{\ell+2}} + \cdots + X^\pi_{i_{\ell+1}, i_{\ell+2}} \\
&= \dkt(\overline{\sigma}, \pi) + 2 (X^\pi_{i_1, i_{\ell+2}} + X^\pi_{i_2, i_{\ell+2}} + \cdots + X^\pi_{i_{\ell+1}, i_{\ell+2}}) - \ell - 1 ,
\end{align*}
where the second equality follows from that $X^\pi_{i,j} + X^\pi_{j,i} = 1$. 
Applying this relation and the binomial expansion of $\dkt(\underline{\sigma}, \pi)^{m-\ell}$, we obtain 
\begin{align*}
&\dkt(\underline{\sigma}, \pi)^{m-\ell} =  \dkt(\overline{\sigma}, \pi)^{m-\ell} + 2 (m-\ell) \,  \dkt(\overline{\sigma}, \pi)^{m-\ell-1} (X^\pi_{i_1, i_{\ell+2}} + \cdots + X^\pi_{i_{\ell+1}, i_{\ell+2}}) \\
& \qquad \quad + \sum_{r=2}^{m-\ell} \binom{m-\ell}{r} \dkt(\overline{\sigma}, \pi)^{m-\ell-r}  \bP_r(X^\pi_{i_1, i_{\ell+2}}, \dots, X^\pi_{i_{\ell+1}, i_{\ell+2}}) + \bP_{m-\ell-1}(X^\pi) . 
\end{align*}
Multiplying the above equation by the degree-$\ell$ polynomial 
$Q^\pi_{i_1, \dots, i_{\ell+1}}$, 
we obtain 
\begin{align}
& \quad \  2 (m-\ell) \,  \dkt(\overline{\sigma}, \pi)^{m-\ell-1} Q^\pi_{i_1, \dots, i_{\ell+2}} \label{eq:lhs0} \\
& =  \dkt(\underline{\sigma}, \pi)^{m-\ell} Q^\pi_{i_1, \dots, i_{\ell+1}}
- \dkt(\overline{\sigma}, \pi)^{m-\ell} Q^\pi_{i_1, \dots, i_{\ell+1}}
+ \bP_{m-1}(X^\pi)  \label{eq:rhs1} \\
& \quad \   - \sum_{r=2}^{m-\ell} \binom{m-\ell}{r} \dkt(\overline{\sigma}, \pi)^{m-\ell-r} Q^\pi_{i_1, \dots, i_{\ell+1}}  \bP_r(X^\pi_{i_1, i_{\ell+2}}, \dots, X^\pi_{i_{\ell+1}, i_{\ell+2}}) \label{eq:lhs1} .
\end{align}

The goal of the induction step is to linearly construct $ \dkt(\overline{\sigma}, \pi)^{m-\ell-1} Q^\pi_{i_1, \dots, i_{\ell+2}}$ from $\cP$. 
Therefore, it suffices to show that each term in \eqref{eq:rhs1} and \eqref{eq:lhs1} can be linearly constructed from $\cP$. 
First, note that all the three terms in \eqref{eq:rhs1} can be done so in view of the induction hypothesis, because $\calP$ contains all polynomials of degree at most $m-1$. 
Next, while it is clear that each summand in \eqref{eq:lhs1} is of degree at most $m$, we will show that it is in fact at most $m-1$, which will complete the proof by induction. 
In view of the definition \eqref{eq:def-q}, it suffices to show that for each $2 \le r \le m - \ell$,
\begin{align}
X^\pi_{i_1, i_2} (X^\pi_{i_1, i_3} + X^\pi_{i_2, i_3}) \cdots (X^\pi_{i_1, i_{\ell + 1}} + \cdots + X^\pi_{i_\ell, i_{\ell + 1}}) \bP_r(X^\pi_{i_1, i_{\ell+2}}, \dots, X^\pi_{i_{\ell+1}, i_{\ell+2}}) 
\label{eq:cycle}
\end{align}
is of degree at most $\ell + r - 1$.

Note that the polynomial \eqref{eq:cycle} is a sum of polynomials of the form 
\begin{align}
X^\pi_{j_1, i_2} X^\pi_{j_2, i_3} \cdots X^\pi_{j_{\ell}, i_{\ell+1}} X^\pi_{j_{\ell+1}, i_{\ell+2}} X^\pi_{j'_{\ell+1}, i_{\ell+2}} \bP_{r-2}(X^\pi) , 
\label{eq:add-two}
\end{align}
where $j_r \in \{i_1, \dots, i_r\}$ for $r=1,\ldots,\ell+1$ and $j'_{\ell+1} \in \{i_1, \dots, i_{\ell+1}\}$. 
Consider the undirected multigraph $G$ with  $\ell+1$ vertices 
$i_1, i_2, i_3, \dots, i_{\ell+1}$
and $\ell$ edges $(j_1, i_2), (j_2, i_3), \dots, (j_{\ell}, i_{\ell+1})$. 
Then $G$ is clearly connected since by assumption $j_r \in \{i_1, \dots, i_r\}$ for each $r \in [\ell+1]$. 
Now, if we add one more vertex $i_{\ell+2}$ and two more edges $(j_{\ell+1}, i_{\ell+2}), (j'_{\ell+1}, i_{\ell+2})$ to the graph $G$, where $j_{\ell+1}, j'_{\ell+1} \in \{i_1, \dots, i_{\ell+1}\}$, 
then there must be a cycle in the new multigraph that contains $i_{\ell+2}$. 
Hence Lemma~\ref{lem:cycle} yields that 
$$
X^\pi_{j_1, i_2} X^\pi_{j_2, i_3} \cdots X^\pi_{j_{\ell}, i_{\ell+1}} X^\pi_{j_{\ell+1}, i_{\ell+2}} X^\pi_{j'_{\ell+1}, i_{\ell+2}} = \bP_{\ell+1}(X^\pi) .
$$ 
It follows that the polynomial \eqref{eq:add-two} and thus the polynomial \eqref{eq:cycle} are of degree at most $\ell + r - 1$. 



\paragraph{Step 2.}
We claim that for any set of distinct indices $\{ i^{(r)}_t \in [n] : t \in [ \ell_r + 1] , \, r \in [s], \, \sum_{r=1}^s \ell_r = m \}$, the polynomial 
\begin{align}
\prod_{r=1}^s  Q^\pi_{i_1^{(r)}, \dots , i_{\ell_r + 1}^{(r)}} 
\label{eq:components}
\end{align}
can be linearly constructed from $\cP$. 

In short, this follows from applying Step~1 iteratively. Specifically, we show that  
\begin{align}
\dkt(\sigma, \pi)^{m-\ell} \prod_{r=1}^s  Q^\pi_{i_1^{(r)}, \dots , i_{\ell_r + 1}^{(r)}} 
\label{eq:poly-int}
\end{align}
can be linearly constructed from $\cP$, where $\ell \in \{0, 1, \dots, m\}$, $\sum_{r=1}^s \ell_r = \ell$, and $\sigma$ is any permutation in $\cS_n$ such that $i^{(r)}_1, \dots, i^{(r)}_{\ell_r}$ appear consecutively in $\sigma$ for each $r \in [s]$. 
Note that \eqref{eq:components} is a special case of \eqref{eq:poly-int}  when $\ell = m$. 

Moreover, \eqref{eq:swap} is a special case of \eqref{eq:poly-int} when $s = 1$. With this base case established, we can construct \eqref{eq:poly-int} inductively on $s$. That is, for $s \ge 2$, it suffices to construct \eqref{eq:poly-int} from 
\begin{align}
&\cP \cup \Big\{ \dkt(\sigma, \pi)^{m-\ell'} \prod_{r=1}^{s-1}  Q^\pi_{i_1^{(r)}, \dots , i_{\ell_r + 1}^{(r)}} : \ell' = \sum_{r=1}^{s-1} \ell_r , \notag \\
&\qquad \qquad \qquad \qquad \qquad \qquad 
i^{(r)}_1, \dots, i^{(r)}_{\ell_r} \text{ appear consecutively in } \sigma \text{ for each } r \in [s-1] \Big\} .
\label{eq:s0}
\end{align}

Toward this end, we apply the linear construction in Step~1, with $m$ replaced by $m - \ell'$, with $\ell$ replaced by $\ell_s$, and with the extra constraint that $i^{(r)}_1, \dots, i^{(r)}_{\ell_r}$ appear consecutively in $\sigma$ for each $r \in [s-1]$. Then we see that, 
the polynomial 
\begin{align}
\dkt(\sigma, \pi)^{m-\ell'-\ell_s} Q^\pi_{i_1^{(s)}, \dots , i_{\ell_s + 1}^{(s)}} , 
\label{eq:t1}
\end{align}
where $i^{(r)}_1, \dots, i^{(r)}_{\ell_r}$ appear consecutively in $\sigma$ for each $r \in [s]$, can be linearly constructed from 
\begin{align}
&\big\{ \dkt(\sigma, \pi)^{\ell_s'} : \ell_s' \in [m-\ell'] , \, i^{(r)}_1, \dots, i^{(r)}_{\ell_r} \text{ appear consecutively in } \sigma \text{ for each } r \in [s-1]  \big\} \notag \\
&\cup \big\{ X^\pi_{i_1, j_1} \cdots X^\pi_{i_{m-\ell'-1}, j_{m-\ell'-1}} : i_r, j_r \in [n], \, i_r \ne j_r, \,  r \in [m-\ell'-1] \big\} . 
\label{eq:s1}
\end{align}

Since the construction is linear, if we multiply \eqref{eq:t1} and each polynomial in \eqref{eq:s1} by the same factor $\prod_{r=1}^{s-1}  Q^\pi_{i_1^{(r)}, \dots , i_{\ell_r + 1}^{(r)}}$, the linear construction remains valid. 
With $\ell' = \sum_{r=1}^{s-1} \ell_r$ and $\ell = \ell' + \ell_s$, this shows that \eqref{eq:poly-int} can be linearly constructed from 
\begin{align}
&\Big\{ \dkt(\sigma, \pi)^{\ell_s'} \prod_{r=1}^{s-1}  Q^\pi_{i_1^{(r)}, \dots , i_{\ell_r + 1}^{(r)}} : \ell_s' \in [m-\ell'] , \notag \\ & \qquad \qquad \qquad \qquad \qquad \qquad \qquad i^{(r)}_1, \dots, i^{(r)}_{\ell_r} \text{ appear consecutively in } \sigma \text{ for each } r \in [s-1] \Big\}  \notag \\
&\cup \Big\{ X^\pi_{i_1, j_1} \cdots X^\pi_{i_{m-\ell'-1}, j_{m-\ell'-1}} \prod_{r=1}^{s-1}  Q^\pi_{i_1^{(r)}, \dots , i_{\ell_r + 1}^{(r)}} : i_r, j_r \in [n], \, i_r \ne j_r, \,  r \in [m-\ell'-1] \Big\} . 
\label{eq:s2}
\end{align}
Since every polynomial in \eqref{eq:s2} can be linearly constructed from \eqref{eq:s0}, this completes the induction.

%

\paragraph{Step 3.}
We claim that for any set of distinct indices $\{ i^{(r)}_t \in [n] : t \in [ \ell_r + 1 ] , \, r \in [s], \, \sum_{r=1}^s \ell_r = m \}$, the monomial 
\begin{align}
\prod_{r=1}^s  X^\pi_{i_1^{(r)}, i_2^{(r)}} X^\pi_{i_2^{(r)}, i_3^{(r)}} \cdots X^\pi_{i_{\ell_r}^{(r)}, i_{\ell_r + 1}^{(r)}} 
\label{eq:components2}
\end{align}
can be linearly constructed from $\cP$. 

By the claim in Step~2, it suffices to prove that 
\begin{align}
&\Big\{ \prod_{r=1}^s  Q^\pi_{i_1^{(r)}, \dots , i_{\ell_r + 1}^{(r)}} : i^{(r)}_t \in [n], \, t \in [ \ell_r + 1 ] , \, r \in [s], \, \sum_{r=1}^s \ell_r = m \Big\} \notag \\
&\Longrightarrow
\Big\{ \prod_{r=1}^s  X^\pi_{i_1^{(r)}, i_2^{(r)}} X^\pi_{i_2^{(r)}, i_3^{(r)}} \cdots X^\pi_{i_{\ell_r}^{(r)}, i_{\ell_r + 1}^{(r)}}  : i^{(r)}_t \in [n], \, t \in [ \ell_r + 1 ] , \, r \in [s], \, \sum_{r=1}^s \ell_r = m \Big\} 
\label{eq:multi-sys}
\end{align}
where all the indices $i^{(r)}_t$'s are distinct. 
In fact, this follows from the fact that 
\begin{align}
\Big\{  Q^\pi_{i_1, \dots , i_{\ell + 1}} : i_1, \dots, i_{\ell+1} \in [n] \Big\}
\Longrightarrow
\Big\{ X^\pi_{i_1, i_2} X^\pi_{i_2, i_3} \cdots X^\pi_{i_{\ell}, i_{\ell + 1}}  : i_1, \dots, i_{\ell+1} \in [n] \Big\} 
\label{eq:lin-sys}
\end{align}
which is an immediate consequence of Lemma~\ref{lem:qx}. 

To prove \eqref{eq:multi-sys} using \eqref{eq:lin-sys}, we note that the construction in \eqref{eq:lin-sys} is linear and thus can be applied to the indices $i^{(q)}_1, \dots i^{(q)}_{\ell_q+1}$ where $q \in [s]$ to obtain 
\begin{align}
&\Big\{ \prod_{r=1}^q  Q^\pi_{i_1^{(r)}, \dots , i_{\ell_r + 1}^{(r)}} \prod_{r=q+1}^s X^\pi_{i_1^{(r)}, i_2^{(r)}} \cdots X^\pi_{i_{\ell_r}^{(r)}, i_{\ell_r + 1}^{(r)}} : i^{(r)}_t \in [n], \, t \in [ \ell_r + 1 ] , \, r \in [s], \, \sum_{r=1}^s \ell_r = m \Big\} \notag \\
&\Longrightarrow
\Big\{ \prod_{r=1}^{q-1}  Q^\pi_{i_1^{(r)}, \dots , i_{\ell_r + 1}^{(r)}} \prod_{r=q}^s X^\pi_{i_1^{(r)}, i_2^{(r)}} \cdots X^\pi_{i_{\ell_r}^{(r)}, i_{\ell_r + 1}^{(r)}} : i^{(r)}_t \in [n], \, t \in [ \ell_r + 1 ] , \, r \in [s], \, \sum_{r=1}^s \ell_r = m \Big\} . 
\label{eq:int-sys}
\end{align}
Iteratively applying \eqref{eq:int-sys} with $q = s, s-1, \dots, 1$ then yields \eqref{eq:multi-sys}.

\paragraph{Step 4.}
To finish the proof of \eqref{eq:induction}, fix $m$ pairs of distinct indices $(i_1, j_1), \dots, (i_m, j_m) \in [n]^2$. 
Consider the undirected multigraph $G$ consisting of edges $(i_1, j_1), \dots, (i_m, j_m)$. 
If $G$ contains a cycle, then $X^\pi_{i_1, j_1} \cdots X^\pi_{i_m, j_m} = \bP_{m-1}(X^\pi)$ by Lemma~\ref{lem:cycle}, so it is already in $\cP$. 
Hence we can assume that $G$ is acyclic, that is, it is a forest. 

Let $G_1, \dots, G_s$ denote the connected components of $G$, each of which is a tree. Let $V(G_r)$ and $E(G_r)$ denote the vertex set and the edge set of $G_r$ respectively for each $r \in [s]$. Let $\ell_r \defn |V(G_r)|-1 = |E(G_r)|$ 
so that $\sum_{r=1}^s \ell_r = m$.  Moreover, we can write 
\begin{align}
X^\pi_{i_1, j_1} \cdots X^\pi_{i_m, j_m} = \prod_{r=1}^s \prod_{(i, j) \in E(G_r)} X^\pi_{i, j} .
\label{eq:prod-out}
\end{align}
By Lemma~\ref{lem:tree} applied to $G_r$, there exist bijections $\tau^{(r)}_1, \dots, \tau^{(r)}_{\beta_r} : [\ell_r+1] \to V(G_r)$ such that 
\begin{align}
\prod_{(i, j) \in E(G_r)} X^\pi_{i, j} = \sum_{\alpha = 1}^{\beta_r} X^\pi_{\tau^{(r)}_\alpha(1), \tau^{(r)}_\alpha(2)} \cdots X^\pi_{\tau^{(r)}_\alpha(\ell_r), \tau^{(r)}_\alpha(\ell_r+1)}. 
\label{eq:prod-in}
\end{align}
Combining \eqref{eq:prod-out} and \eqref{eq:prod-in} 
and expanding the product of sums, 
we see that $X^\pi_{i_1, j_1} \cdots X^\pi_{i_m, j_m}$ is a sum of monomials of the form 
$$
\prod_{r=1}^s X^\pi_{\tau^{(r)}_{\alpha_r}(1), \tau^{(r)}_{\alpha_r}(2)} \cdots X^\pi_{\tau^{(r)}_{\alpha_r}(\ell_r), \tau^{(r)}_{\alpha_r}(\ell_r+1)} . 
$$
In fact, this is of the same form as \eqref{eq:components2} and thus can be linearly constructed from $\cP$. 
This shows that $X^\pi_{i_1, j_1} \cdots X^\pi_{i_m, j_m}$ can be linearly constructed from $\cP$, thereby completing the proof.

\subsection{Proof of Corollary~\ref{cor:high-noise}}

\subsubsection{Proof of part~(a)}


Given i.i.d.\ observations $\sigma_1, \dots, \sigma_N \sim \cM$, 
the empirical distribution $\cM_N$ has PMF 
$
f_{\cM_N}(\sigma) = \frac 1N \sum_{i=1}^N \1\{ \sigma_i = \sigma \} . 
$
Hoeffding's inequality then gives 
$$
\p \big\{ | f_{\cM_N}(\sigma)  - f_{\cM}(\sigma) | > t \big\} \le 2 \exp ( - 2 N t^2 ) 
$$
for any $t > 0$. 
Taking a union bound over $\sigma \in \cS_n$ yields 
\begin{align}
\p \big\{ \TV( \cM, \cM_N ) > t n!/2 \big\} \le 2 n! \exp ( - 2 N t^2 )  . 
\label{eq:tv-tail}
\end{align}

On the other hand, by part~(a) of Theorem~\ref{thm:high-noise}, for any $\cM' \in \scrM_*$ distinct from $\cM$, we have $\TV( \cM, \cM' ) \ge c_1 \eps^m$ for a constant $c_1 = c_1(n, k) > 0$. Choosing $t = \frac{ c_1 }{ 2 n! } \eps^m$ in \eqref{eq:tv-tail} yields that
$
\TV( \cM, \cM_N ) \le c_1 \eps^m / 4 
$
with probability at least $1 - 2 n! \exp( - c_2 N \eps^{2m} )$ for a constant $c_2 = c_2(n, k) > 0$. 
On this event, the minimum total variation distance estimator $\widehat{\cM}$ defined by \eqref{eq:min-tv} is equal to $\cM$. 

Finally, it suffices to note that if $N \ge C \log(\frac{1}{\delta}) / \eps^{2m}$ for a sufficiently large constant $C = C(n, k) >0$, then the failure probability can be bounded as 
$
2 n! \exp( - c_2 N \eps^{2m} ) \le \delta . 
$

\subsubsection{Proof of part~(b)}

By part~(b) of Theorem~\ref{thm:high-noise}, there exist distinct Mallows mixtures $\cM, \cM' \in \scrM_*$ for which 
$
\TV( \cM, \cM' ) \lesssim_{n, k} \eps^{m} ,
$
where the notation $\lesssim_{n, k}$ hides a constant factor that may depend on $n$ and $k$. 
Let $f'$ denote the PMF of $\cM'$. 
For $n$ fixed and as $\eps \to 0$, that is, as $\phi \to 1$, $f'$ converges pointwise to $1/(n!)$. 
Therefore, for sufficiently small $\epsilon$, we have $f'(\sigma) \ge 1/(2 n!)$ for each $\sigma \in \cS_n$. 
By reserve Pinsker inequality (see, for example, Theorem~2 of~\cite{Ver14}), it then follows 
\begin{align*}
\KL( \cM, \cM' ) \lesssim \frac{ \TV( \cM, \cM' )^2 }{ \min_{\sigma \in \cS_n} f'(\sigma) } \lesssim_n  \TV( \cM, \cM' )^2  \lesssim_{n, k} \eps^{2m} . 
\end{align*}
Let $\cM^{\otimes N}$ and $(\cM')^{\otimes N}$ denote the distribution of $N$ i.i.d.\ observations from $\cM$ and $\cM'$ respectively. Then Pinsker's inequality together with tensorization of the KL divergence yields 
$$
\TV \big( \cM^{\otimes N} , (\cM')^{\otimes N} \big) \le \sqrt{ \KL \big( \cM^{\otimes N} , (\cM')^{\otimes N} \big) } \lesssim_n \sqrt{ N \eps^{2m} } . 
$$
Finally, applying Le Cam's two-point lower bound (cf.~e.g.~\cite[Sec 2.3]{Tsybakov09}) gives 
$$
\min_{\widetilde{\cM}} \max_{\cM \in \scrM_*} \p_{\cM} \{ \widetilde{\cM} \ne \cM \} \ge \frac 14 \Big( 1 - \TV \big( \cM^{\otimes N} , (\cM')^{\otimes N} \big) \Big) \ge \frac 18 , 
$$
if $N \le c / \eps^{2m}$ for a sufficiently small constant $c = c(n, k) > 0$.

\bibliography{mixture}

\newcommand{\etalchar}[1]{$^{#1}$}
\begin{thebibliography}{BFFSZ19}

\bibitem[ABSV14]{Awaetal14}
Pranjal Awasthi, Avrim Blum, Or~Sheffet, and Aravindan Vijayaraghavan.
\newblock Learning mixtures of ranking models.
\newblock In {\em Advances in Neural Information Processing Systems}, pages
  2609--2617, 2014.

\bibitem[BFFSZ19]{Busetal19}
Robert Busa-Fekete, Dimitris Fotakis, Bal\'{a}zs Sz\"{o}r\'enyi, and Manolis
  Zampetakis.
\newblock Optimal learning of mallows block model.
\newblock In Alina Beygelzimer and Daniel Hsu, editors, {\em Proceedings of the
  Thirty-Second Conference on Learning Theory}, volume~99, pages 529--532,
  2019.

\bibitem[BFHS14]{BusHulSzo14}
R{\'o}bert Busa-Fekete, Eyke H{\"u}llermeier, and Bal{\'a}zs Sz{\"o}r{\'e}nyi.
\newblock Preference-based rank elicitation using statistical models: the case
  of mallows.
\newblock In {\em Proceedings of the 31st International Conference on Machine
  Learning-Volume 32}, pages II--1071. JMLR.org, 2014.

\bibitem[BM09]{BraMos09}
Mark Braverman and Elchanan Mossel.
\newblock Sorting from noisy information.
\newblock {\em arXiv preprint arXiv:0910.1191}, 2009.

\bibitem[BMR10]{BalMakRic10}
Linas Baltrunas, Tadas Makcinskas, and Francesco Ricci.
\newblock Group recommendations with rank aggregation and collaborative
  filtering.
\newblock In {\em Proceedings of the fourth ACM conference on Recommender
  systems}, pages 119--126, 2010.

\bibitem[BOB07]{BusOrbBuh07}
Ludwig~M. Busse, Peter Orbanz, and Joachim~M. Buhmann.
\newblock Cluster analysis of heterogeneous rank data.
\newblock In {\em Proceedings of the 24th international conference on Machine
  learning}, pages 113--120. ACM, 2007.

\bibitem[Bor81]{Bor81}
J.~C. Borda.
\newblock M{\'e}moire sur les {\'e}lections au scrutin.
\newblock {\em Histoire de l'Academie Royale des Sciences pour}, 1781.

\bibitem[CDKL15]{Chietal15}
Flavio Chierichetti, Anirban Dasgupta, Ravi Kumar, and Silvio Lattanzi.
\newblock On learning mixture models for permutations.
\newblock In {\em Proceedings of the 2015 Conference on Innovations in
  Theoretical Computer Science}, pages 85--92, 2015.

\bibitem[Con85]{Con85}
M.~J. Condorcet.
\newblock {\em Essai sur l'application de l'analyse {\`a} la probabilit{\'e}
  des d{\'e}cisions rendues {\`a} la pluralit{\'e} des voix}.
\newblock 1785.

\bibitem[CPS13]{CarProSha13}
Ioannis Caragiannis, Ariel~D Procaccia, and Nisarg Shah.
\newblock When do noisy votes reveal the truth?
\newblock In {\em Proceedings of the fourteenth ACM conference on Electronic
  commerce}, pages 143--160, 2013.

\bibitem[DKNS01]{Dwoetal01}
Cynthia Dwork, Ravi Kumar, Moni Naor, and Dandapani Sivakumar.
\newblock Rank aggregation methods for the web.
\newblock In {\em Proceedings of the 10th international conference on World
  Wide Web}, pages 613--622, 2001.

\bibitem[DOS18]{DeODoSer18}
Anindya De, Ryan O'Donnell, and Rocco Servedio.
\newblock Learning sparse mixtures of rankings from noisy information.
\newblock {\em arXiv preprint arXiv:1811.01216}, 2018.

\bibitem[DPR04]{DoiPekReg04}
Jean-Paul Doignon, Aleksandar Peke{\v{c}}, and Michel Regenwetter.
\newblock The repeated insertion model for rankings: Missing link between two
  subset choice models.
\newblock {\em Psychometrika}, 69(1):33--54, 2004.

\bibitem[DWYZ20]{DosWuYanZho20}
Natalie Doss, Yihong Wu, Pengkun Yang, and Harrison~H Zhou.
\newblock Optimal estimation of high-dimensional gaussian mixtures.
\newblock {\em arXiv preprint arXiv:2002.05818}, 2020.

\bibitem[FKS03]{FagKumSiv03}
Ronald Fagin, Ravi Kumar, and Dandapani Sivakumar.
\newblock Efficient similarity search and classification via rank aggregation.
\newblock In {\em Proceedings of the 2003 ACM SIGMOD international conference
  on Management of data}, pages 301--312, 2003.

\bibitem[FV86]{FliVer86}
Michael~A Fligner and Joseph~S Verducci.
\newblock Distance based ranking models.
\newblock {\em Journal of the Royal Statistical Society: Series B
  (Methodological)}, 48(3):359--369, 1986.

\bibitem[GM08a]{GorMur08b}
Isobel~Claire Gormley and Thomas~Brendan Murphy.
\newblock Exploring voting blocs within the irish electorate: A mixture
  modeling approach.
\newblock {\em Journal of the American Statistical Association},
  103(483):1014--1027, 2008.

\bibitem[GM08b]{GorMur08}
Isobel~Claire Gormley and Thomas~Brendan Murphy.
\newblock A mixture of experts model for rank data with applications in
  election studies.
\newblock {\em The Annals of Applied Statistics}, 2(4):1452--1477, 2008.

\bibitem[HK18]{HK2015}
Philippe Heinrich and Jonas Kahn.
\newblock Strong identifiability and optimal minimax rates for finite mixture
  estimation.
\newblock {\em The Annals of Statistics}, 46(6A):2844--2870, 2018.

\bibitem[ICL19]{IruCalLoz19}
Ekhine Irurozki, Borja Calvo, and Jose~A Lozano.
\newblock Mallows and generalized mallows model for matchings.
\newblock {\em Bernoulli}, 25(2):1160--1188, 2019.

\bibitem[JJ94]{JorJac94}
Michael~I. Jordan and Robert~A. Jacobs.
\newblock Hierarchical mixtures of experts and the em algorithm.
\newblock {\em Neural computation}, 6(2):181--214, 1994.

\bibitem[KCS17]{KorCleSib17}
Anna Korba, Stephan Cl{\'e}men{\c{c}}on, and Eric Sibony.
\newblock A learning theory of ranking aggregation.
\newblock In {\em Artificial Intelligence and Statistics}, pages 1001--1010,
  2017.

\bibitem[LB11]{LuBou11}
Tyler Lu and Craig Boutilier.
\newblock Learning mallows models with pairwise preferences.
\newblock In {\em Proceedings of the 28th International Conference on
  International Conference on Machine Learning}, pages 145--152, 2011.

\bibitem[LB14]{LuBou14}
Tyler Lu and Craig Boutilier.
\newblock Effective sampling and learning for mallows models with
  pairwise-preference data.
\newblock {\em The Journal of Machine Learning Research}, 15(1):3783--3829,
  2014.

\bibitem[LLQ{\etalchar{+}}07]{Liuetal07}
Yu-Ting Liu, Tie-Yan Liu, Tao Qin, Zhi-Ming Ma, and Hang Li.
\newblock Supervised rank aggregation.
\newblock In {\em Proceedings of the 16th international conference on World
  Wide Web}, pages 481--490, 2007.

\bibitem[LM18]{LiuMoi18}
Allen Liu and Ankur Moitra.
\newblock Efficiently learning mixtures of mallows models.
\newblock In {\em 2018 IEEE 59th Annual Symposium on Foundations of Computer
  Science (FOCS)}, pages 627--638. IEEE, 2018.

\bibitem[Mal57]{Mal57}
Colin~L Mallows.
\newblock Non-null ranking models. i.
\newblock {\em Biometrika}, 44(1/2):114--130, 1957.

\bibitem[Mar95]{Mar95}
John~I. Marden.
\newblock {\em Analyzing and modeling rank data}.
\newblock Chapman and Hall/CRC, 1995.

\bibitem[MC10]{MeiChe10}
Marina Meil{\u{a}} and Harr Chen.
\newblock Dirichlet process mixtures of generalized mallows models.
\newblock In {\em Proceedings of the Twenty-Sixth Conference on Uncertainty in
  Artificial Intelligence}, pages 358--367, 2010.

\bibitem[MM03]{MurMar03}
Thomas~Brendan Murphy and Donal Martin.
\newblock Mixtures of distance-based models for ranking data.
\newblock {\em Computational statistics \& data analysis}, 41(3-4):645--655,
  2003.

\bibitem[MPPB07]{Meietal07}
Marina Meil{\u{a}}, Kapil Phadnis, Arthur Patterson, and Jeff Bilmes.
\newblock Consensus ranking under the exponential model.
\newblock In {\em Proceedings of the Twenty-Third Conference on Uncertainty in
  Artificial Intelligence}, pages 285--294, 2007.

\bibitem[MV10]{MV2010}
Ankur Moitra and Gregory Valiant.
\newblock Settling the polynomial learnability of mixtures of {G}aussians.
\newblock In {\em Foundations of Computer Science (FOCS), 2010 51st Annual IEEE
  Symposium on}, pages 93--102. IEEE, 2010.

\bibitem[MW22]{MaoWu20}
Cheng Mao and Yihong Wu.
\newblock Supplement to ``learning mixtures of permutations: Groups of pairwise
  comparisons and combinatorial method of moments''.
\newblock 2022.

\bibitem[Pea94]{Pea94}
Karl Pearson.
\newblock Contributions to the mathematical theory of evolution.
\newblock {\em Philosophical Transactions of the Royal Society of London. A},
  185:71--110, 1894.

\bibitem[Tsy09]{Tsybakov09}
A.~B. Tsybakov.
\newblock {\em Introduction to Nonparametric Estimation}.
\newblock Springer Verlag, New York, NY, 2009.

\bibitem[Ver14]{Ver14}
Sergio Verd{\'u}.
\newblock Total variation distance and the distribution of relative
  information.
\newblock In {\em 2014 Information Theory and Applications Workshop (ITA)},
  pages 1--3. IEEE, 2014.

\bibitem[WY20]{WY18}
Yihong Wu and Pengkun Yang.
\newblock Optimal estimation of {G}aussian mixtures with denoised method of
  moments.
\newblock {\em The Annals of Statistics}, 48(4):1981--2007, 2020.

\bibitem[Zag92]{Zag92}
Don Zagier.
\newblock Realizability of a model in infinite statistics.
\newblock {\em Communications in mathematical physics}, 147(1):199--210, 1992.

\bibitem[ZPX16]{ZhaPieXia16}
Zhibing Zhao, Peter Piech, and Lirong Xia.
\newblock Learning mixtures of plackett-luce models.
\newblock In {\em International Conference on Machine Learning}, pages
  2906--2914, 2016.

\end{thebibliography}
\bibliographystyle{alpha}

\end{document}